%% file: main.tex
\documentclass{lmcs}

\input{makros}


\begin{document}

\title[Network Satisfaction over Four-Atom Relation Algebras]{Network Satisfaction over Four-Atom Relation Algebras: Complexity and Representations}
\titlecomment{An extended abstract announcing some results of the present article appeared in the Proceedings of the 53rd International Colloquium on Automata, Languages, and Programming (ICALP 2026), LIPIcs 374, Article 168, doi:10.4230/LIPIcs.ICALP.2026.168 under the title \emph{The Network Satisfaction Problem for Relation Algebras with at Most Four Atoms}}

\author[M. Bodirsky]{Manuel Bodirsky\lmcsorcid{0000-0001-8228-3611}}
\address{Institut f\"ur Algebra, TU Dresden, Germany}
\email{manuel.bodirsky@tu-dresden.de}
\thanks{Manuel Bodirsky, Moritz Jahn, Mat\v{e}j Kone\v{c}n\'y, and Paul Winkler received funding from the ERC (Grant Agreement no. 101071674, POCOCOP). Views and opinions expressed are however those of the authors only and do not necessarily reflect those of the European Union or the European Research Council Executive Agency.}

\author[M. Jahn]{Moritz Jahn\lmcsorcid{0009-0004-9523-0040}}
\address{Institut f\"ur Algebra, TU Dresden, Germany}
\email{moritz\_peter.jahn@mailbox.tu-dresden.de}

\author[S. Kn\"auer]{Simon Kn\"auer\lmcsorcid{0009-0006-7421-8565}}
\address{Institut f\"ur Algebra, TU Dresden, Germany}

\author[M. Kone\v{c}n\'y]{Mat\v{e}j Kone\v{c}n\'y\lmcsorcid{0000-0001-7391-3247}}
\address{Institut f\"ur Algebra, TU Dresden, Germany}
\email{matej.konecny@tu-dresden.de}

\author[P. Winkler]{Paul Winkler\lmcsorcid{0009-0002-7080-4875}}
\address{Institut f\"ur Algebra, TU Dresden, Germany}
\email{paul.winkler1@tu-dresden.de}

\begin{abstract}
Andréka and Maddux classified the relation algebras with at most 3 atoms, 
and in particular they showed that all of them are representable. Cristani and Hirsch showed that the network satisfaction problem (NSP) for each of these algebras is in $\Ptime$ or $\NP$-complete. 
The literature contains many results on representations of relation algebras; in particular, some relation algebras with four atoms are not representable. We extend the result of Cristani and Hirsch to relation algebras with at most 4 atoms: the NSP is always either in $\Ptime$ or $\NP$-complete. To this end, we construct universal, fully universal, or even normal representations for these algebras, whenever possible. 
\end{abstract}

\maketitle

\section{Introduction}
Relation algebras are algebraic structures that can be used to abstractly reason about binary relations.
They are certain expansions of Boolean algebras;
besides the operations for union, intersection, and complement of binary relations, we also have a binary operation $\circ$ for composition of relations, and a unary operation $\breve{\phantom{}}$ for reversing the arguments of a binary relation. These operations need to satisfy certain natural identities; e.g., we require that  $(a \circ b)\,\breve{\phantom{}} = \breve{b} \circ \breve{a}$ holds for all elements $a$ and $b$ of the relation algebra.

Similarly to the fundamental facts that every group is isomorphic to a permutation group, and every Boolean algebra is isomorphic to an algebra of sets, the initial hope was that every relation algebra has a \emph{representation}, which consists of a set of binary relations on some set that behave as prescribed by the relation algebra \cite{Tarski41}.

\begin{exa}\label{ex:point-algebra}
We first look at a concrete structure from which a small relation algebra can
be obtained. Consider the structure $(\mathbb{Q};<)$, where ${\mathbb Q}$ is the set of rational numbers and $<$ is the usual strict linear order on ${\mathbb Q}$. For any two rational
numbers $x,y$, exactly one of the following holds: $x<y$, $x=y$, or $x>y$.
Thus, the relations $<$, $=$, and $>$ form a partition of $\Q^2$. Moreover,
$\breve{<} = {>}$, where $\breve{~}$ denotes the usual converse operation on
binary relations, defined by $\breve{R} = \{(y,x) \mid (x,y) \in R\}$.

This partition gives rise to a relation algebra $\bA$ with three atoms,
denoted by $r$, $\id$, and $\breve{r}$. Its domain $A$ consists of all
formal unions of these atoms, i.e.,
$A = \{0,\allowbreak r,\allowbreak \breve{r},\allowbreak \id,\allowbreak
r \cup \id,\allowbreak \breve{r} \cup \id,\allowbreak r \cup \breve{r},
\allowbreak r \cup \breve{r} \cup \id\}$. The other Boolean operations are
defined in the usual way.

Let $\circ$ denote the usual
composition of binary relations, i.e., $$R \circ S = \{(x,z) \mid \exists y\colon (x,y) \in R
\text{ and } (y,z) \in S\}.$$ In $(\Q;<)$, for example, we have ${< \circ <} = {<}$ and
${< \circ >} = {(< \cup > \cup =)}$. Accordingly, $\bA$ has a binary operation,
also denoted by $\circ$, satisfying, for example,
$r \circ r = r$ and $r \circ \breve{r} = (r \cup \breve{r} \cup \id)$.

Note that $\bA$ is a Boolean algebra with $8$ elements, together with a
binary operation $\circ$ and a unary operation $\breve{~}$. Instead of obtaining $\bA$ from concrete relations on $\mathbb{Q}$, one may
also start with the abstract algebra $\bA$ and ask whether it has a
`representation', similarly as above.

Formally, a representation $\fB$ of a relation algebra $\bA$ assigns to every $a\in A$ a binary relation $h(a) = a^\fB$ on a common domain $B$, injectively and preserving the Boolean operations, converse, identity, and relational composition, where complementation is taken relative to $1^\fB$. In particular, distinct atoms are
mapped to disjoint binary relations, while formal unions such as $r \cup \id$ are
mapped to the corresponding unions. In the present example, the ordered set
$(\mathbb{Q};<)$ provides such a representation: take $D=\mathbb{Q}$ and map
$r$, $\id$, and $\breve r$ to the relations $<$, $=$, and $>$, respectively. The value of $h$ on all other elements of $A$ is then determined by preservation of the Boolean operations; for example, $h(r \cup \id)= h(r)\cup h(\id)= ({<}\cup {=})$.
\end{exa}

Indeed, all
relation algebras with at most $2^3 = 8$ elements have a representation~\cite{AndrekaMaddux}.
However, Lyndon~\cite{LyndonRelationAlgebras} found a relation algebra with $2^{56}$ elements which is not representable. Monk~\cite{Monk} showed that this problem has no simple fix, because he proved that the class of representable relation algebras cannot be axiomatized by a finite set of identities, and not even by a finite set of first-order sentences, answering a question of Tarski.
However, it follows from results of Tarski~\cite{TarskiRRA} that there is an \emph{infinite} set of identities that axiomatizes the representable relation algebras.
We mention that (integral) relation algebras have also been discovered and studied with a different formalism and under the different name of \emph{hypergroups}~\cite{Hypergroups} (also called \emph{polygroups}~\cite{ComerChromaticPolygroups}).

\subsection{Network Satisfaction Problems and CSP Complexity}
Starting in the 90s, another generation of researchers investigated relation algebras motivated by applications in theoretical computer science, particularly for the theory of constraint satisfaction. In a constraint satisfaction problem, we are given a finite set of variables and a finite set of constraints, and the task is to determine whether there exists an assignment to the variables that satisfies all the constraints. Constraint satisfaction problems have been studied intensively if the variables take values from a fixed finite domain. Relation algebras allow us to formalize interesting constraint satisfaction problems where the variables may take values from an infinite domain.

Concretely, to every finite relation algebra $\bA$, we may associate the so-called \emph{network satisfaction problem}, which is a computational problem where the input consists of
a \emph{`constraint network'}, i.e., a finite set of nodes $V$ and a function labeling pairs of nodes with elements from the relation algebra, and the task is to decide whether there exists a representation of $\bA$ where this network is satisfiable.
Under certain representability conditions, e.g., if $\bA$ has a square representation
(\cite[Section\,1.5]{Book}; definitions may be found in~Section~\ref{sect:prelims}), the network satisfaction problem for $\bA$ is a constraint satisfaction problem as described above, but not necessarily for values from a finite set (the values rather come from the elements of the representation, which might be infinite).
If there is a fully universal representation, then the network satisfaction problem is in $\NP$. However, the complexity sometimes even drops to $\Ptime$ (polynomial time). The computational complexity of the NSP has been shown to be $\NPc$ ($\NP$-complete) or $\Ptime$ for all relation algebras with at most $8=2^3$ elements (called `small relation algebras') by Cristani and Hirsch~\cite{HirschCristani}.

\begin{exa}
Let $\bA$ be the relation algebra from Example \ref{ex:point-algebra}. An \textit{$\bA$-network} consists of finitely many nodes and a function $f\colon V^2 \rightarrow A$ mapping each ordered pair to some element of the relation algebra. Under the representation over $(\Q;<)$, the label $f(x,y)$ specifies the allowed relative position of the values assigned to $x$ and $y$.

For instance, consider a network with nodes $x,y,z$ and function $f$ such that
$f(x,y)=r$, $f(y,z)=r \cup \id$, and $f(z,x)=\breve{r} \cup \id$, and such that all
other pairs are labeled by ${r \cup \breve{r}\cup \id}$. Under the representation over $(\Q;<)$, these constraints mean $x<y$, $y\leq z$, and
$z\geq x$, respectively. This network is satisfiable, for example by assigning
$x=0$, $y=1$, and $z=2$.

On the other hand, if $f(x,y)=r$, $f(y,z)=r$, and $f(z,x)=r$, then the network
is not satisfiable in this representation, since it would require
$x<y<z<x$.

In fact, the representation over $(\Q;<)$ is \textit{normal}, in particular \textit{fully universal}; see Section~\ref{sect:prelims} for the definitions.
The fact that the representation is fully universal implies that the network satisfaction problem for $\bA$ amounts to deciding whether a finite system of constraints over the signature $\{\varnothing, <, \leq, >, \geq, =, \neq, \Q^2\}$ can be satisfied over the rational numbers.
\end{exa}

Hirsch~\cite{Hirsch} asked whether the computational complexity of the network satisfaction problem can be classified for all finite relation algebras. Bodirsky and Kn\"auer~\cite{BodirskyKnaeuerDatalog23} proved that the problem of deciding whether a given finite relation algebra has an NSP which is in $\Ptime$ is undecidable, which can be interpreted as a negative answer to Hirsch's question.
The proof of the undecidability result relies on two earlier results, namely that there are finite relation algebras whose NSP is undecidable~\cite{Hirsch-Undecidable}, and that the question whether a finite relation algebra has a representation is undecidable~\cite{HirschHodkinsonRepresentability}.
The classification for relation algebras
with a fully universal representation, however, remains wide open.

A particularly interesting subclass are the relation algebras that have a \emph{normal representation}, i.e., a representation which is fully universal, square, and homogeneous~\cite{Hirsch}.
While this combination of assumptions still captures a rich class of relation algebras (including standard structures like $\Q$, the random graph, the Henson graph, and the random tournament), it has a wide range of strong and useful consequences, as we will see below.
A relational structure is \emph{homogeneous} if every isomorphism between finite substructures can be extended to an automorphism of the structure.
If a relation algebra $\bA$ has a normal representation, then $\bA$ has a countable normal representation $\fB$, and this representation is unique up to isomorphism. Moreover, one can express important properties of the automorphism group of $\fB$ using $\bA$, and conversely, properties of $\bA$ are reflected in this group. It is known that the question whether a given finite relation algebra has a normal representation is decidable~\cite{BodirskyRamics}.

The NSPs for relation algebras with a normal representation fall into the scope of the  so-called \emph{Bodirsky--Pinsker tractability conjecture} in constraint satisfaction, namely the conjecture that CSPs of \emph{reducts of finitely bounded homogeneous structures} are in $\Ptime$ or $\NP$-complete~\cite{BPP-projective-homomorphisms}. Specifically, the conjecture states that the subclasses of $\NP$ that properly lie between $\Ptime$ and $\NP$ and were constructed by Ladner~\cite{Ladner} (conditionally on $\Ptime \neq \NP$) do not show up when classifying such NSPs.
Such a dichotomy statement has been established for CSPs of finite structures~\cite{Zhuk20} (announced independently by Bulatov~\cite{BulatovFVConjecture} and by Zhuk~\cite{ZhukFVConjecture}).
One of the motivations for the more general conjecture from~\cite{BPP-projective-homomorphisms} is that many of the universal-algebraic methods used by both Bulatov and Zhuk can be applied for the more general class. And indeed, there exists a universal-algebraic condition which is a candidate for separating polynomial-time tractable from $\NP$-complete CSPs for reducts of finitely bounded homogeneous structures, and this condition has numerous equivalent characterisations~\cite{BPP-projective-homomorphisms,wonderland,BKOPP,BartoPinskerDichotomy,Book}. If this condition applies, the CSP is provably $\NP$-hard; if not, then the conjecture states that the CSP is in $\Ptime$.
In our case, the tractability conjecture can be formulated as follows (for the definition of pp-constructability, see~Section~\ref{sect:prelims}). Let $\NAE := \{0,1\}^3 \setminus \{(0,0,0),(1,1,1)\}$ be the ternary Boolean not-all-equal relation.
\begin{conj}\label{conj:tractability}
    Let $\fB$ be a normal representation of a finite relation algebra $\bA$. If $\fB$ does \emph{not} pp-construct $(\{0,1\}; \NAE)$, then $\CSP(\fB)$ and $\NSP(\bA)$
    are in $\Ptime$.\footnote{In the case where $\fB$ is a normal representation of $\bA$, $\CSP(\fB)$ and $\NSP(\bA)$ are the same computational problem up to the translation between $\bA$-networks and $A$-structures.}
\end{conj}

We mention that the scope of the Bodirsky--Pinsker conjecture for relation algebras is larger than just the relation algebras with a normal representation: e.g., if a relation algebra has a representation which is \emph{homogenizable} in the sense that we can add finitely many first-order definable relations to make it homogeneous, then its NSP is covered as well.

Homogeneous structures are also of central interest in model theory; via Fra\"{i}ssé's theorem, they are particularly easy to construct and provide a rich source of examples and counterexamples.
It is conceivable that the class of homogeneous structures with a finite relational signature can be classified in some sense.
There are uncountably many~\cite{Henson}, but the uncountable families usually have a particularly simple shape~\cite{Cherlin}. One of the approaches here is to fix some (typically binary) relational signature (e.g., focusing on graphs~\cite{LachlanWoodrow}, digraphs~\cite{Cherlin}, permutations~\cite{Cameron}, metric spaces~\cite{CherlinMetrically}, etc.) and to then classify the homogeneous structures with this signature. One of the motivations for this research program is its success story of producing beautiful and interesting mathematical structures that have inspired research in the areas mentioned above. Systematically listing finite relation algebras and then testing for the existence of a normal representation is another approach that can produce interesting new examples of homogeneous (or homogenizable) structures.

There are potentially further well-behaved subclasses of representable relation algebras beyond those admitting homogenizable or even normal representations. To identify such classes, it is natural to consider families of relation algebras that are both sufficiently representative and sufficiently small to allow for a detailed analysis. Restricting the number of atoms provides such a family.

\subsection{Contributions}
We extend the mentioned result of Cristani and Hirsch on small relation algebras, and systematically study all relation algebras with at most $16=2^4$ elements. An exhaustive list of the building blocks for such relation algebras has been computed by Maddux~\cite{Maddux2006-dp}, confirming results of Comer~\cite{ComerChromaticPolygroups}, namely a list for the integral symmetric relation algebras with four atoms (numbered $\ra{1}{65}$--$\ra{65}{65}$ by Maddux) and a list for the integral relation algebras with an asymmetric atom among its four atoms (numbered $\ra{1}{37}$--$\ra{37}{37}$ by Maddux).

We prove that the NSPs for these not so small relation algebras are in $\Ptime$ or $\NP$-complete.
In the case where the algebra has a normal representation, our results confirm the mentioned tractability conjecture from~\cite{BPP-projective-homomorphisms}. More formally, we obtain the following results.

\begin{thm}
    Let $\bA$ be a finite relation algebra with at most four atoms. Then $\NSP(\bA)$ is in $\Ptime$ or $\NP$-complete.
\end{thm}

\begin{thm}
    Let $\fB$ be a normal representation of a finite relation algebra $\bA$ with at most four atoms. If $\fB$ does \emph{not} pp-construct $(\{0,1\}; \NAE)$, then $\CSP(\fB)$ and $\NSP(\bA)$
    are in $\Ptime$.
\end{thm}

For all but one of these algebras we prove the containment of the NSP in $\NP$ in this paper; the one exceptional case is $\ra{56}{65}$. Here, the proof requires substantially different techniques and was published separately (see \cite{56paper}).

In addition to our classification of NSP complexity, we revisit representability for all finite relation algebras with at most 4 atoms. Several of the algebras in the list do not have a representation (in which case the NSP is trivial and in $\Ptime$). We determine for each of the representable algebras $\bA$ whether $\bA$ has a normal representation, and if not, whether at least it has a fully universal (or even fully universal square) representation.
This result is of independent interest in the theory of relation algebras; representations for all relation algebras with at most 3 atoms have been found in 1994~\cite{AndrekaMaddux}.
For this purpose, we present a combinatorial characterization for the existence of a fully universal square representation.

\medskip
We think of this paper as exploratory work towards more general classifications of the complexity of NSPs, or representability questions, of suitable subclasses of relation algebras. Maddux's list of relation algebras with at most four atoms is at the same time rich enough to display diverse behaviour, while remaining small enough for detailed analysis, including computer-assisted inspection.

We verified that in many cases, the existing methods apply (e.g.~\cite{BodirskyKnaeuerDatalog23,BodirskyKnaeJAIR,BodirskyKnaeRamics}). At the same time, we also needed to develop new methods (such as proving $\NP$-hardness of the NSPs by reductions from the approximate graph coloring problem $\CSP(K_3,K_5)$~\cite{PCSP}; a so-called \emph{promise constraint satisfaction problem}, see Section~\ref{sect:PCSPs}), as well as several ad hoc arguments for concrete cases. Several of these deserve explicit mention.

\begin{itemize}
    \item $\ra{24}{65}$ (Section~\ref{sect:FU}); this algebra is related to the class of 3-edge-coloured cliques with no rainbow triangle, which exhibits some nice combinatorics, and has a fully universal representation. Our structural insights are then the basis for a polynomial-time algorithm for its NSP (Section~\ref{sect:24-alg}).
    \item
    $\ra{17}{37}$ (Section~\ref{sect:17_37-rep}); this algebra is related to the class of quasi-transitive oriented graphs, which has been studied in graph theory~\cite{BJJ95}.
    There is no fully universal representation and we identify an atomic network with 4 vertices which is consistent but unsatisfiable. By working with the class of consistent atomic networks that exclude this 4-element network, we can construct a universal representation for $\ra{17}{37}$. Our structural insights are then also useful for the polynomial-time algorithm for the NSP of $\ra{17}{37}$ mentioned above.
    \item $\ra{51}{65}$
    (Section~\ref{sect:51_65}); this algebra contains exactly one non-trivial  equivalence relation, and in every representation, the equivalence relation can have at most $3$ classes. Our representability results for this algebra follow a similar pattern as our results for $\ra{17}{37}$: while there is no fully universal representation, we can identify four atomic networks that are consistent but unsatisfiable. Excluding those, we can construct a universal representation; it follows that the NSP is in $\NP$, and in fact it is $\NP$-complete (Section~\ref{sect:PCSPs}).
    \item $\ra{56}{65}$ (Section~\ref{sect:56_65}); this relation algebra is similar to the previous one, but allows one extra triple compared to $\ra{51}{65}$. It has a representation, but not a fully universal one. Its NSP is $\NP$-complete (Section~\ref{sect:PCSPs} and~\cite[Theorem 1.4]{56paper}).
\end{itemize}

There are already several thousand relation algebras with five atoms~\cite{Maddux2006-dp}, making a direct extension of the present classification substantially more challenging. The four-atom examples are expected to correspond to large subclasses of the five-atom list, and analysing these subclasses will help to identify useful generalisations, by methods developed in this article; see Section~\ref{sect:cycle-prod-hardness}.

\subsection{Significance for Constraint Satisfaction}
We believe that network satisfaction problems for relation algebras with a normal representation are an important class
to test the mentioned tractability conjecture.
The reason is that this approach can lead to CSPs with different behaviour, challenging the current methods. Previously, the Bodirsky--Pinsker conjecture has mainly been tested on  structures obtained as follows: fix a homogeneous finitely bounded `base' structure of fundamental interest (such as $({\mathbb Q};<)$ or the Rado graph), and study the CSP of all first-order reducts of it (see, e.g.,~\cite{tcsps-journal,BodPin-Schaefer-both}). The advantage of this approach is that the model-theoretic properties of the entire class of structures obtained like this are inherited from the
base structure, and the remaining task is mostly universal-algebraic. This advantage is at the same time a disadvantage, because by fixing a base structure we might overlook some phenomena that only appear for other base structures. The approach of the present article avoids this; and indeed, our study shows that new structures with different behaviour and interesting polynomial-time CSPs may show up.

\tableofcontents

\section{Preliminaries}
\label{sect:prelims}
In this section we present the background from the theory of relation algebras that we need to state our results. Then we introduce some concepts and results from model theory and universal algebra that we need to prove our results.

Let $\N = \{0, 1, \dots\}$ denote the natural numbers and $\N^+ = \N \setminus \{0\}$ the positive integers.

\subsection{Structures}
\label{sect:structs}
A \emph{signature} $\tau$ is a (possibly infinite) set of function and relation symbols; each symbol has an \emph{arity} $k \in \N$. A function symbol $f \in \tau$ of arity $0$ is called a \emph{constant}.

\begin{defi}
    A \emph{$\tau$-structure} $\fA$ is a set $A$ (the \emph{domain} of $\fA$)\footnote{In this article it will be convenient and without harm to allow the empty set.} together with
    \begin{itemize}
        \item a relation $R^\fA \subseteq A^k$ for every relation symbol $R \in \tau$ with arity $k$,  and
        \item a function $f^\fA \colon A^k \to A$ for every function symbol $f \in \tau$ with arity $k$.
    \end{itemize}
\end{defi}

If the signature $\tau$ only contains function symbols, then a $\tau$-structure is called an \emph{algebra}.
If $\tau$ is just made up of relation symbols, then a $\tau$-structure is called a \emph{relational structure}.
We will use $\bA,\bB,\dots$ to denote algebras,
and $\fA, \fB, \dots$ to denote relational structures, and structures in general. If not stated otherwise, the underlying domains are $A, B, \dots$ respectively.
If the structure $\fA$ is clear from the context, we will disregard the superscript and just write $f$ for the function $f^\fA$ and $R$ for the relation $R^\fA$.

\begin{defi}
    A $\tau$-structure $\fA$ is a \emph{substructure} of a $\tau$-structure $\fB$ if $A \subseteq B$ and
    \begin{itemize}
        \item for every $R\in \tau$ of arity $k$ and $a \in A^k$: $a \in R^\fA \text{ if and only if } a \in R^\fB$, and
        \item for every $f \in \tau$ of arity $k$ and $a \in A^k$: $f^\fA(a) = f^\fB(a)$.
    \end{itemize}
\end{defi}
The (inclusion-wise) smallest substructure of $\fB$ that contains $S \subseteq B$ is called the \emph{substructure of $\fB$ generated by $S$}.

\begin{defi}\label{def:product}
    Let $\fA$ and $\fB$ be $\tau$-structures. The \emph{(direct) product} $\fC = \fA \times \fB$ is the $\tau$-structure where
    \begin{itemize}
        \item $A \times B$ is the domain of $\fC$,
        \item for every relation symbol $R$ of arity $n \in \N$ and every tuple $((a_1, b_1), \dots, (a_n, b_n)) \in (A \times B)^n$, we have that $((a_1, b_1), \dots, (a_n,b_n)) \in R^\fC$ if and only if $(a_1, \dots, a_n) \in R^\fA$ and $(b_1, \dots, b_n) \in R^\fB$, and
        \item for every function symbol $f$ of arity $n \in \N$ and every tuple $((a_1, b_1), \dots, (a_n, b_n)) \in (A\times B)^n$, we have that
        \[f^\fC((a_1, b_1), \dots, (a_n,b_n)) := (f^\fA(a_1, \dots, a_n), f^\fB(b_1, \dots, b_n)).\]
    \end{itemize}
    We write $\fA^2$ for $\fA\times \fA$ and inductively define $\fA^{k+1} := \fA \times \fA^k$ for $k \geq 2$.
\end{defi}

\begin{defi}
    Let $\fA$ and $\fB$ be $\tau$-structures. A \emph{homomorphism} from $\fA$ to $\fB$ is a mapping $\varphi \colon A \to B$ such that for all $a_1, \dots, a_k \in A$ and
    \begin{itemize}
        \item for every $R\in \tau$ with arity $k$: if $(a_1, \dots, a_k) \in R^\fA \text{ then } (\varphi(a_1), \dots, \varphi(a_k)) \in R^\fB$;
        \item for every $f \in\tau$ with arity $k$: $\varphi(f^\fA(a_1, \dots, a_k)) = f^\fB(\varphi(a_1), \dots, \varphi(a_k))$.
    \end{itemize}
\end{defi}

Two structures $\fA$ and $\fB$ are called \emph{homomorphically equivalent} if there is a homomorphism $\varphi$ from $\fA$ to $\fB$ and from $\fB$ to $\fA$.
If $\varphi$ is bijective and $\varphi^{-1}$ is a homomorphism from $\fB$ to $\fA$, then $\varphi$ is called an \emph{isomorphism}. In this case we say that $\fA$ and $\fB$ are \emph{isomorphic} ($\fA \cong \fB$).
The \emph{age} of a relational structure $\fB$ is the class of all finite structures $\fA$ that are isomorphic to a substructure of $\fB$.

An isomorphism from $\fA$ to $\fA$ is called an \emph{automorphism}. A relational $\tau$-structure $\fA$ is \emph{homogeneous}, if for every isomorphism between finite substructures of $\fA$ there exists an automorphism of $\fA$ that extends this isomorphism.
Homogeneous structures with a finite relational signature are \emph{$\omega$-categorical} (see, e.g.,~\cite{Hodges}), i.e., all countable models of the first-order theory of the structure are isomorphic.

If $R_1,R_2 \subseteq B^2$ are two binary relations, then
$R_1 \circ R_2$ denotes the \emph{composition}
\begin{align*}
    R_1 \circ R_2 := \{(x,z) \in B^2 \mid \text{ there exists } y \in B \text{ with } (x,y) \in R_1 \text{ and } (y,z) \in R_2 \}.
\end{align*}

\subsection{Constraint Satisfaction Problems and Universal Algebra}
\label{sect:CSPS-and-universal-algebra}
In this section, we introduce several basic concepts from universal algebra which we will later need for analysing the computational complexity of network satisfaction problems.

Let $\fB$ be a structure with a finite relational signature $\tau$.
Then the \emph{constraint satisfaction problem} $\CSP(\fB)$ is the problem of deciding for a given
finite $\tau$-structure $\fA$ whether there exists a homomorphism to $\fB$. Note that a class of finite $\tau$-structures is of the form
$\CSP(\fB)$ if and only if it is closed under disjoint unions and its complement is closed under homomorphisms.

\begin{exa}\label{expl:hard}
    Let $\NAE$ be the Boolean relation
    $\{0,1\}^3 \setminus \{(0,0,0),(1,1,1)\}$. Then
    $\CSP(\{0,1\};\allowbreak \NAE)$ is the positive not-all-equal 3-SAT problem and $\NP$-complete~\cite{GareyJohnson}.
\end{exa}

We work with classical first-order logic, using $\top$, $\bot$, the Boolean operations $\land$, $\lor$, $\rightarrow$, and $\neg$, the quantifiers $\forall$, $\exists$, and equality. A formula is called \emph{primitive positive} (shortly pp) if it is of the form
\begin{align*}
    \exists x_1 \cdots \exists x_k (\psi_1 \land \cdots \land \psi_n),
\end{align*}
where each $\psi_i$ is atomic, i.e., of the form $\top$, $\bot$, $x = y$, or $R(x_{i1}, \hdots, x_{il})$ for some $l$-ary relation $R$ in the signature.

\begin{defi}\label{def:pp-interpretation}
Let $\fA$ and $\fB$ be structures with the relational signatures $\tau$ and $\sigma$, respectively. A \emph{primitive positive interpretation of dimension $d$} of $\fA$ in $\fB$ is a partial surjection $h\colon \mathrm{dom}(h) \subseteq B^d \rightarrow A$ (called the \emph{coordinate map}) such that for every relation $R$, say of arity $k$, which can be defined by an atomic formula over $\fA$, the $(k \cdot d)$-ary relation
\begin{align*}
    h^{-1}(R) := \{(b_{11}, \hdots, b_{1d}, \hdots, b_{k1}, \hdots, b_{kd}) \in B^{k\cdot d} \mid (h(b_{11}, \hdots, b_{1d}), \hdots, h(b_{k1}, \hdots, b_{kd})) \in R^\fA\}
\end{align*}
can be defined in $\fB$ by a primitive positive formula $\varphi_R$.
Since equality and $\top$ are always allowed as atomic formulas, there must in particular exist
\begin{itemize}
    \item a $\sigma$-formula $\varphi_\top$ (called the \emph{domain formula}) such that $\varphi_\top(b_1, \hdots, b_d)$ holds if and only if $(b_1, \hdots, b_d)$ is in the domain of $h$;
    \item a $\sigma$-formula $\varphi_=$ such that $\varphi_=(b_1, \hdots, b_d, c_1, \hdots, c_d)$ holds if and only if $h(b_1, \hdots, b_d) = h(c_1, \hdots, c_d)$.
\end{itemize}
In this case we also say that $\fB$ \emph{pp-interprets} $\fA$.\
We say that $\fB$ \emph{pp-constructs} $\fA$ if $\fB$ pp-interprets a structure that is homomorphically equivalent to $\fA$.
\end{defi}

The following is well-known.

\begin{prop}[see, e.g., {\cite[Theorem 3.1.4]{Book}}]\label{prop:pp-interpretation-implies-logspace-reducible}
If $\fB$ pp-constructs $\fA$, then $\CSP(\fA)$ is log-space reducible to $\CSP(\fB)$. In particular, if $\fB$ pp-constructs $(\{0,1\};\NAE)$, then $\CSP(\fB)$ is $\NP$-hard.
\end{prop}

An important tool to study primitive positive definability in finite and $\omega$-categorical structures is the concept of a \emph{polymorphism}.

\begin{defi}
Let $\tau$ be finite. A \emph{polymorphism} of a $\tau$-structure $\fB$ is a homomorphism $f$ from $B^k$ to $B$. We write $\Pol(\fB)$ for the set of all polymorphisms of $\fB$.
\end{defi}

\begin{defi}
    Let $\fB$ be a relational $\tau$-structure. An operation $f \colon B^n \to B$ is called \emph{conservative} if for all $x_1, \dots, x_n \in B$ it holds that $f(x_1, \dots, x_n) \in \{x_1, \dots, x_n\}$. The set $\Pol(\fB)$ is \emph{conservative} if every $f \in \Pol(\fB)$ is conservative. In this case, $\fB$ is called \emph{conservative} as well.
\end{defi}

We will use the following important types of operations.

\begin{defi}
    An operation $f \colon B^k \to B$ is called
    \begin{itemize}
        \item a \emph{binary symmetric} operation if $k=2$ and
        $f(x,y) = f(y,x)$ for all $x,y \in B$;
        \item a \emph{majority} operation if $k=3$ and for all $x,y \in B$
        \begin{align*}
            f(x,x,y) = f(x,y,x) = f(y,x,x) = x;
        \end{align*}
        \item a \emph{minority} operation if $k=3$ and for all $x,y \in B$
        \begin{align*}
            f(x,x,y) = f(x,y,x) = f(y,x,x) = y;
        \end{align*}
        \item a \emph{weak near unanimity (WNU)} operation if $k \geq 2$ and for all $x,y \in B$
        \begin{align*}
            f(y, x, \hdots, x) = f(x, y, x, \hdots, x) = \cdots = f(x, \hdots, x, y).
        \end{align*}
    \end{itemize}
\end{defi}

The complexity of the CSP for finite conservative structures has been classified by Bulatov~\cite{Conservative} long before the resolution of the finite domain CSP dichotomy conjecture.

\begin{thm}[Bulatov~\cite{Conservative}]
\label{thm:dichotomyConservative}
    Let $\fB$ be a finite conservative structure with a finite relational  signature.
    If for every $a,b \in B$ there exists a polymorphism $f \in \Pol(\fB)$ such that the restriction of $f$ to $\{a,b\}$ is a binary symmetric, a majority, or a minority operation,
    then $\CSP(\fB)$ is in $\Ptime$. Otherwise, $(\{0,1\};\NAE)$ has a pp-interpretation in $\fB$, and $\CSP(\fB)$ is $\NPc$.
\end{thm}

\begin{defi}[see, e.g., {\cite[Section 6.1]{Book}}]
    Let $S$ be an arbitrary set. An \emph{(operation) clone} (over $S$) is a subset $\mathcal C$ of $\bigcup_{i \in \N} S^{S^i}$, i.e., of functions from finite powers of $S$ to $S$, satisfying the following two properties:
    \begin{itemize}
        \item $\mathcal C$ contains all projections, i.e., for all $i, k \in \N$ with $1 \leq i \leq k$, the operation $\pi(i, k)\colon S^k \rightarrow S$ defined by $\pi(i, k)(x_1, \hdots, x_k) = x_i$ is contained in $\mathcal C$.
        \item $\mathcal C$ is closed under composition, i.e., for every $n$-ary operation $f \in \mathcal C$ and all $m$-ary functions $g_1, \hdots, g_n \in \mathcal C$, the operation $f(g_1, \hdots, g_n)$ defined by
        \begin{align*}
            f(g_1, \hdots, g_n)(x_1, \hdots, x_m) = f(g_1(x_1, \hdots, x_m), \hdots, g_n(x_1, \hdots, x_m))
        \end{align*}
        is contained in $\mathcal C$.
    \end{itemize}
\end{defi}

Note that for every relational structure $\fB$, the set $\Pol(\fB)$ is a clone. Moreover, the set of projections on an arbitrary set $S$ is a clone.

We give modern formulations of the following two classical results:

\begin{thm}[Post \cite{Post1941}]\label{thm:post41}
Let $\mathcal C$ be a clone over a two-element set $\{a, b\}$. Then one of the following holds:
\begin{itemize}
    \item $\mathcal C$ contains a binary symmetric, the ternary majority, the ternary minority operation, or the unary function $f$ defined by $f(a) = b$ and $f(b) = a$.
    \item $\mathcal C$ contains only the projections.
\end{itemize}
\end{thm}

\begin{thm}[Schaefer \cite{Schaefer1978}]
\label{thm:schaefer78}
    Let $\fB$ be a structure over a two-element domain. Then one of the following holds:
    \begin{itemize}
        \item $\Pol(\fB)$
        contains a binary symmetric,
        the ternary majority, or the ternary minority operation, and $\CSP(\fB)$ is contained in $\Ptime$.
        \item $\fB$ pp-interprets $(\{0, 1\}; \mathrm{NAE})$; in this case, $\CSP(\fB)$ is $\NP$-hard.
    \end{itemize}
\end{thm}

Sometimes, it will be convenient to prove the existence of pp-interpretations algebraically; we then need the following definitions and Theorem~\ref{thm:clone-hom-to-proj-implies-NP-hard}.

\begin{defi}[see, e.g., {\cite[Definition 6.5.3.]{Book}}]
Let $\mathcal C$ and $\mathcal D$ be clones. A function $\xi\colon \mathcal C \rightarrow \mathcal D$ is called a \emph{clone homomorphism} if
\begin{itemize}
    \item $\xi$ preserves arities, i.e., for every $n$-ary function $f \in \mathcal C$, the function $\xi(f)$ is also $n$-ary;
    \item $\xi$ preserves the projections, i.e., $\xi(\pi(i, k)^\mathcal C) = \pi(i, k)^\mathcal D$ for all $k \in \N$ and all $i \in \{1, \hdots, k\}$;
    \item for every $n$-ary function $f \in \mathcal C$ and all $m$-ary functions $g_1, \hdots, g_n \in \mathcal C$,
    \begin{align*}
        \xi(f(g_1, \hdots, g_n)) = \xi(f)(\xi(g_1), \hdots, \xi(g_n)).
    \end{align*}
\end{itemize}
\end{defi}

\begin{defi}[see, e.g., {\cite[Section 9.5.1]{Book}}]\label{def:uniformly-continuous-clone-hom}
Let $\mathcal C_1$ and $\mathcal C_2$ be clones over countable sets $C_1$ and $C_2$. A function $\xi\colon C_1 \rightarrow C_2$ is called \emph{uniformly continuous} if for every finite $G \subseteq C_2$ there is a finite set $F \subseteq C_1$ such that for every $k \in \N$ and all operations $f, g \in \mathcal C_1$ with equal arity,
$f\vert_{F^k} = g\vert_{F^k}$ implies $\xi(f)\vert_{G^k} = \xi(g)\vert_{G^k}$.
\end{defi}

\begin{thm}[{\cite[Theorem 28]{TopoBirkhoff}}; see also {\cite[Corollary 9.5.21]{Book}}]\label{thm:clone-hom-to-proj-implies-NP-hard} Let $\fB$ be a countable $\omega$-categorical structure. Then the following are equivalent:
\begin{itemize}
    \item There is a uniformly continuous clone homomorphism from $\Pol(\fB)$ to the clone of projections on a two-element set.
    \item $\fB$ pp-interprets $(\{0, 1\}; \mathrm{NAE})$; in this case, $\CSP(\fB)$ is $\NP$-hard.
\end{itemize}
\end{thm}

\subsection{Relation Algebras}\label{sect:relationalgebras_def}
\begin{defi}\label{def:relation_algebra}
    A \emph{relation algebra} is an algebra $\bA$ with domain $A$ and signature $\{\cup, \bar{\phantom{o}}, 0,1,\id, \breve{\phantom{o}}, \circ \}$ such that
	\begin{enumerate}
		\item the structure $(A; \cup, \cap, \bar{\phantom{o}}, 0,1)$, with $\cap$ defined by $x\cap y := \overline{(\bar{x}\cup \bar{y})}$, is a Boolean algebra;
		\item $\circ$ is an associative binary operation on $A$, called \emph{composition};
		\item for all $a,b,c \in A$: $(a\cup b)\circ c= (a\circ c) \cup (b\circ c)$;
		\item for all $a\in A$: $a\circ \id =a$;
		\item for all $a\in A$: $\Breve{\Breve{a}} = a$;
		\item for all $a,b\in A$: $(a \cup b){\Breve{\phantom{o}}} =\Breve{a}\cup \Breve {b}$;
        \item for all $a,b\in A$: $(a \circ b){\Breve{\phantom{o}}}=\Breve{b} \circ \Breve{a}$;
        \item for all $a,b \in A$: $\bar{b} \cup \left(\Breve{a} \circ  \overline{(a\circ b)  }\right) =\bar{b}$.
	\end{enumerate}
 \end{defi}

\begin{rem}\label{rem:consequences-of-axioms-for-relation-algebra}
    It is easy to see that the axioms in Definition \ref{def:relation_algebra} imply $\breve{0} = 0$, $\breve{\id} = \id$, $\id \circ\,a = a$, $c\,\circ\,(a \cup b) = (c\,\circ\,a) \cup (c\,\circ\,b)$ for all $a, b, c \in A$ and similar statements, which we will freely use in the following. Formal proofs can be found in \cite[Chapter 6]{Maddux2006-dp}.
\end{rem}

We denote the class of all relation algebras by RA.
Since all the given axioms are universal, RA is closed under the formation of direct products (Definition~\ref{def:product}). A relation algebra is called \emph{trivial} if $0=1$.

If $\bA$ is a relation algebra with elements $a$ and $b$, then we write $a \leq b$ if $a \cap b = a$ holds in $\bA$.
 Clearly, $\leq$ defines a partial order on $A$. An element $b \in A \setminus \{0^\bA\}$ is called an \emph{atom} if there is no element $a \in A \setminus \{0^\bA, b\}$ with $a \leq b$. The set of all atoms is denoted by $A_0$. If $a \in A_0$ satisfies $a \leq \id$, it is called an \emph{identity atom} and otherwise a \emph{diversity atom}.
We also write $a^2$ for $a \circ a$.

An element $a \in A$ is called
\begin{itemize}
    \item \emph{symmetric} if $\breve{a} = a$;
    \item \emph{transitive} if $a^2 \leq a$;
    \item \emph{reflexive} if $\id \leq a$;
    \item an \emph{equivalence relation} if $a$ is reflexive, symmetric and transitive.
\end{itemize}

A relation algebra $\bA$ is called
\begin{itemize}
    \item \emph{symmetric} if all $a \in A$ are symmetric;
    \item \emph{integral} if $0 \neq 1$ and $x \circ y = 0$ implies $x=0$ or $y=0$ (see~\cite{Maddux2006});
    \item \emph{simple} if $1 \circ x \circ 1 = 1$ for all $x \neq 0$.
\end{itemize}

\begin{rem}\label{rem:simple}
    We mention that a relation algebra is simple in the sense given above if and only if it is simple in the sense of universal algebra (see, e.g.,~\cite{BS}); moreover, a relation algebra is simple if and only if it is not the direct product of two non-trivial relation algebras~\cite[Theorem 4.14]{TarskiJonnsonOperators}.
\end{rem}

\begin{exa}\label{expl:1}
    There is an up to isomorphism unique relation algebra with the elements $\{0,1\}$ (with $\id=1$).
    This algebra is called $1_1$ in \cite{Maddux2006-dp}. Then $(1_1)^2$ has two atoms and is the smallest non-simple relation algebra.
\end{exa}

Every integral relation algebra is simple~\cite[Theorem 4.18\,(ii)]{TarskiJonnsonOperators}.
The converse is not true as shown in Example~\ref{expl:non-int}.
In fact, a simple relation algebra is integral if and only if $\id$ is an atom~\cite[Theorem 353]{Maddux2006}.

\begin{exa}\label{expl:non-int}
    The following relation algebra with four atoms is simple but not integral.
    It has atoms $a,b,c,d$, where $a$ and $b$ are identity atoms,
    and is uniquely given by requiring $\breve{c}=d$ and the multiplication table given in Figure~\ref{fig:non-int}.
    We will later see that this is the only non-integral simple relation algebra with at most four atoms (Proposition~\ref{prop:non-int-unique}).
\end{exa}

\begin{figure}
    \centering
    \begin{tabular}{|L|LLLL|}\hline
        \circ & a & b & c & \breve{c} \\ \hline
        a & a & 0 & 0 & \breve{c} \\
        b & 0 & b & c & 0 \\
        c & c & 0 & 0 & b \\
        \breve{c} & 0 & \breve{c} & a & 0 \\ \hline
    \end{tabular}
    \caption{Multiplication table of the simple non-integral relation algebra from Example~\ref{expl:non-int}.}
    \label{fig:non-int}
\end{figure}

For a finite relation algebra $\bA$, the operation $\circ$ is completely determined by its restriction to the atoms. A tuple $(x,y,z) \in (A_0)^3$ is called an \emph{allowed triple} if $z \leq x \circ y$. Otherwise, $(x,y,z)$ is called a \emph{forbidden triple}. We denote by
 \[ \Cy(\bA) := \{ (x,y,z) \in A_0^3 \mid z \leq x \circ y \}\]
the set of all allowed triples. The following proposition shows that even less information is needed.
\begin{prop}[cycle law; {\cite[Theorem 294]{Maddux2006-dp}}]\label{prop:cyclelaw}
    Let $\bA$ be a relation algebra. For $x,y,z \in A_0$ it holds that
    \[ z \leq x \circ y \iff y \leq \breve{x} \circ z \iff x \leq z \circ \breve{y} \iff \breve{y} \leq \breve{z} \circ x \iff \breve{x} \leq y \circ \breve{z} \iff \breve{z} \leq \breve{y} \circ \breve{x}.\]
    Moreover, for $x,y,z \in A_0$ and
    \[ [x,y,z] := \{ (x,y,z), (\breve{x},z,y), (z,\breve{y},x), (\breve{z},x,\breve{y}), (y, \breve{z}, \breve{x}), (\breve{y}, \breve{x}, \breve{z}) \},\]
    either $[x,y,z] \cap \Cy(\bA) = \varnothing$ or $[x,y,z] \subseteq \Cy(\bA)$.
\end{prop}

\subsection{Representations}
\begin{defi}\label{defn:repr}
Let $\bA = (A; \cup, \bar{\phantom{o}}, 0,1,\id,  \Breve{\phantom{o}}, \circ)$ be a finite relation algebra. A structure $\fB$ with signature $A$
is called a \emph{representation of $\bA$} if
\begin{enumerate}
    \item\label{ax:repr:1} $0^{\fB} = \varnothing$;
    \item $a^{\fB} \neq \varnothing$ for all atoms $a \in A_0$;
    \item\label{ax:repr:3} $\id^{\fB} = \{(u,u) \mid u \in B\}$;
    \item for all $a \in A$ we have $({\overline{a}})^{\fB} = 1^{\fB} \setminus a^{\fB}$;
    \item for all $a \in A$ we have $(\breve{a})^{\fB} = \{(u,v) \mid (v,u) \in a^{\fB} \}$;
    \item\label{ax:repr:6} for all $a,b \in A$ we have $a^{\fB} \cup b^{\fB} = (a \cup b)^{\fB}$;
    \item\label{ax:repr:7} for all $a,b \in A$ we have $a^{\fB} \circ b^{\fB} = (a \circ b)^{\fB}$.
\end{enumerate}
Relation algebras that have a representation are called \emph{representable}.
\end{defi}
Note that when constructing a representation of $\bA$, it is enough to describe how the atoms of $\bA$ are interpreted in $\fB$, the other relations are then uniquely determined.

\begin{exa}
    Trivial relation algebras can be represented by structures with an empty domain. The relation algebra from Example~\ref{expl:1} has a representation by a structure with one element.
\end{exa}

\begin{lem}[see, e.g.,~\cite{BodirskyKnaeuerDatalog23}]\label{lem:union}
    Let $\fB_1$ be a representation of $\bA_1$ and $\fB_2$ a representation of $\bA_2$, and suppose that $A_1 \cap A_2 = \varnothing$ and $B_1 \cap B_2 = \varnothing$. Then $\bA_1 \times \bA_2$ has a \emph{union representation},
    which is the $(A_1 \times A_2)$-structure $\fB_1 \uplus \fB_2$ with domain $B_1 \cup B_2$ defined by $(a_1,a_2)^{\fB_1 \uplus \fB_2} := a_1^{\fB_1} \cup a_2^{\fB_2}$.
\end{lem}

This lemma also has a converse.

\begin{lem}[see, e.g.,~\cite{BodirskyKnaeuerDatalog23}]\label{lem:factor}
    If $\fB$ is a representation of $\bA_1 \times \bA_2$, then there exist
    a representation $\fB_1$ of $\bA_1$ and a representation $\fB_2$ of $\bA_2$
    such that $\fB$ is isomorphic to $\fB_1 \uplus \fB_2$.
\end{lem}

It is known that the question whether a given finite relation algebra is representable is undecidable~\cite{HirschHodkinsonRepresentability}.
A representation $\fB$ is called \emph{square} if $1^{\fB} = B^2$.

\begin{exa}\label{expl:non-int-rep}
    The non-integral relation algebra from Example~\ref{expl:non-int} has a square representation $\fB$
    with the domain $\{0,1\}$
    which is given by
    $a^{\fB} = \{(0,0)\}$, $b^{\fB} = \{(1,1)\}$, $c^{\fB} = \{(1,0)\}$, and  $d^{\fB} = \{(0,1)\}$.
\end{exa}

\begin{exa}[$\ra{5}{7}$]\label{ex:5_7Part1}
    The relation algebra $\ra{5}{7}$ is described in Table~\ref{fig:3SymAtoms} by its allowed triples. It is a well-known combinatorial fact that every square representation has at most
    five elements.
    In fact, up to isomorphism
    $\ra{5}{7}$ has only one square representation, namely the representation $\fB$
    with domain $B = \{0, \dots, 4\}$,
    where
    \begin{align*}
        \id^\fB := \{(x,x) \mid x \in B\}, \quad a^\fB := \{(x,x \pm_5 1) \mid x \in B\}, \quad b^\fB := B^2\setminus (\id^\fB \cup a^\fB).
    \end{align*}
\end{exa}

Besides the union representation, there is another construction of representations that will be relevant.

\begin{defi}\label{def:prod}
    Let $(\bA_i)_{i \in I}$ be a sequence of relation algebras
    with representations $(\fB_i)_{i \in I}$ with disjoint domains. Define
    $\fB := \bigotimes_{i \in I} \fB_i$ to be the $\bigcup_{i \in I} A_i$-structure with domain
    $\bigcup_{i \in I} B_i$ where $a^{\fB} := \bigcup_{\{i \in I \mid a \in A_i\}} a^{\fB_i}$.
\end{defi}
If $\bA$ is a relation algebra with representation $\fB$,
then we can write $1^{\fB} = \bigcup_{i \in I} E_i^2$,
and for every $i \in I$ call $\fB|_{E_i}$ a \emph{square component} of $\fB$.

\begin{lem}[{\cite[Lemma 3.7]{HirschHodkinson}}]
\label{lem:decomp}
Let $\bA$ be a relation algebra with representation $\fB$,
and let $\{\fB_i \mid i \in I\}$ be the set of square components of $\fB$. For $i \in I$, let $\bA_i$ be the relation algebra with the representation $\fB_i$.
Then $\fB = \bigotimes_{i \in I} \fB_i$,
and for every $i \in I$ there is a surjective homomorphism mapping $\bA$ to
$\bA_i$.
\end{lem}

\begin{cor}\label{cor:simple-implies-square-rep}
    Let $\bA$ be a simple representable relation algebra. Then $\bA$ has a square representation.
\end{cor}
\begin{proof}
    Let $\fB$ be a representation of $\bA$ with square components $\{\fB_i \mid i \in I\}$.
    For $i \in I$, let $\bA_i$ be the relation algebra with representation $\fB_i$. By Lemma \ref{lem:decomp}  there is a surjective
    homomorphism $h_i$ from $\bA$ to $\bA_i$ for each $i \in I$. If $h_i$ were not injective, then its kernel would be a non-trivial congruence relation on $\bA$, in contradiction to $\bA$ being simple (see Remark \ref{rem:simple}). Hence, $h_i$ is an isomorphism and $\fB_i$ is a square representation of $\bA$.
\end{proof}

A representation $\fB$ of $\bA$ is \emph{finitely bounded} if there is a finite family $\mathcal F$ of
finite structures with signature $A$ such that a finite structure $\mathfrak C$ with signature $A$ embeds into $\fB$ if and only if no member of $\mathcal F$ embeds into $\mathfrak C$.

\subsection{Networks}
If $\bA$ is a relation algebra, then an \emph{$\bA$-network} $(V,f)$ consists of a finite set of variables $V$ and a function $f \colon V^2 \to A$  (see, e.g., \cite[Section 1.5.3]{Book}). If $\fB$ is a representation of $\bA$, then $(V,f)$ is called \emph{satisfiable in $\fB$} if there exists a function $s \colon V \to B$ such that for all $x,y \in V$ we have
$(s(x),s(y)) \in f(x,y)^{\fB}$.
An $\bA$-network $(V,f)$ is called
\begin{itemize}
    \item \emph{atomic} if for all $x,y \in V$ we have that $f(x,y)$ is an atom in $\bA$;
    \item \emph{consistent} (also: \emph{path-consistent}, \emph{3-consistent}, or \emph{closed}) if
     for all $x,y,z \in V$ we have
     \begin{align}
         f(x,y) \leq f(x,z) \circ f(z,y);\qquad f(x, y) = f(y,x)\breve{~};\qquad f(x,x) \leq \id.
     \end{align}
\end{itemize}

A representation $\fB$ of $\bA$ is called \emph{universal} if every satisfiable
$\bA$-network is satisfiable in $\fB$, and
\emph{fully universal}
if every consistent atomic $\bA$-network is satisfiable in $\fB$.
Clearly, every fully universal representation is also universal.

\begin{exa}[$\ra{5}{7}$ (cont.)]\label{example:5_7_not_fu}
    As pointed out in Example~\ref{ex:5_7Part1}, every square representation $\fB$ of $\ra{5}{7}$ has exactly five elements and is unique up to isomorphism; therefore, it is in particular universal. Clearly, the consistent atomic $\ra{5}{7}$-network depicted in Figure~\ref{fig:evil_square} is unsatisfiable in this representation (this network appears for example in~\cite{Hirsch-Undecidable}).
    Hence, there is no fully universal representation of $\ra{5}{7}$.
\end{exa}

If $(V,f)$ and $(W,g)$ are $\bA$-networks such that $W\subseteq V$ and $g = f|_{W^2}$, we say that $(W,g)$ is a \emph{subnetwork} of $(V,f)$. A consistent atomic network $(V,f)$ is called \emph{reduced} if for all $x,y \in V$, the atom $f(x,y)$ is an identity atom if and only if $x=y$.

\begin{figure}
    \centering
    \[\begin{tikzcd}
	\bullet && \bullet \\
	\\
	\bullet && \bullet
	\arrow["a", no head, from=1-1, to=1-3]
	\arrow["a"', no head, from=1-1, to=3-1]
	\arrow["b"{pos=0.7}, no head, from=1-1, to=3-3]
	\arrow["a", no head, from=1-3, to=3-3]
	\arrow["b"{pos=0.7}, no head, from=3-1, to=1-3]
	\arrow["a"', no head, from=3-1, to=3-3]
\end{tikzcd}\]
    \caption{The \emph{evil square}; it plays an important role in various proofs that certain relation algebras, e.g., $\ra{5}{7}$, do not have a fully universal square representation. (Throughout the paper, we will represent symmetric relations by undirected edges.)}
    \label{fig:evil_square}
\end{figure}

\begin{obs}\label{obs:fin_bounded_networks}
    Let $\fB$ be a square representation of a finite relation algebra $\bA$. Then $\fB$ is finitely bounded if and only if there is a finite set $\mathcal F'$ of consistent atomic $\bA$-networks such that a consistent atomic $\bA$-network $(V,f)$ is satisfiable in $\fB$ if and only if no member of $\mathcal F'$ is isomorphic to a subnetwork of $(V,f)$.
\end{obs}

\begin{proof}
    First note that every reduced atomic $\bA$-network has a corresponding structure with signature $A$ on the same vertex set (simply define relations based on the network) satisfying axioms~\ref{ax:repr:1} and ~\ref{ax:repr:3}--\ref{ax:repr:6} from Definition~\ref{defn:repr}, and conversely, every such structure comes from a reduced atomic $\bA$-network. Moreover, a reduced atomic $\bA$-network is consistent if and only if the corresponding structure satisfies the ``$\subseteq$'' inclusion of axiom~\ref{ax:repr:7} of Definition~\ref{defn:repr} (and a failure of this is witnessed on three vertices).

    Consequently, a finite structure with signature $A$ corresponds to a reduced consistent atomic $\bA$-network if and only if all its substructures on at most three vertices do, which happens if and only if all its substructures on at most three vertices satisfy axioms~\ref{ax:repr:1} and ~\ref{ax:repr:3}--~\ref{ax:repr:6} and the ``$\subseteq$'' inclusion of axiom~\ref{ax:repr:7} of Definition~\ref{defn:repr}. Since $\bA$ is finite, there are only finitely many structures with signature $A$ on at most three vertices up to isomorphism.

    Finally, if $(V,f)$ is a consistent atomic $\bA$-network and $(V',f\vert_{(V')^2})$ is its (up to isomorphism unique) reduced subnetwork, then $(V,f)$ is satisfiable in $\fB$ if and only if $(V',f\vert_{(V')^2})$ is.

    Consequently, if there is a finite set $\mathcal F'$ as in the statement, then we can assume that it consists of reduced networks and we can put $\mathcal F$ to be the set consisting of all structures corresponding to members of $\mathcal F'$ together with all structures with signature $A$ on at most three vertices which do not embed into $\fB$, and $\mathcal F$ witnesses that $\fB$ is finitely bounded. Conversely, if $\fB$ is finitely bounded witnessed by $\mathcal F$, we can simply put $\mathcal F'$ to be the set of all reduced consistent atomic $\bA$-networks whose corresponding structure is in $\mathcal F$.
\end{proof}

\subsection{Normal Representations}
\label{sect:normal}
A representation $\fB$ of $\bA$ is called \emph{normal} if it is square, fully universal, and homogeneous.
Note that in normal representations $\fB$ of $\bA$, the \emph{orbits} of $\Aut(\fB)$ are in one-to-one correspondence with the identity atoms of $\bA$, and the orbits of the component-wise action of $\Aut(\fB)$ on $B^2$ are in one-to-one correspondence with
the atoms of $\bA$. It follows from Theorem~\ref{thm:two-point-amalgamation} below that the question whether a given finite relation algebra has a normal representation is decidable.

\begin{thm}
    All countable normal representations of a relation algebra $\bA$ are isomorphic.
\end{thm}
\begin{proof}
    Let $\fB$ be a countable normal representation.
    Clearly, every finite substructure of $\fB$ describes a
    consistent reduced atomic $\bA$-network. Conversely,
    the assumption that the representation is fully universal implies
    that every reduced consistent atomic $\bA$-network describes a finite substructure of $\fB$.
    Hence, the age of $\fB$
    is fully specified by $\bA$. The statement now follows from the well-known fact in model theory that a countable homogeneous structure is up to isomorphism uniquely given by its age (this can be shown by a back-and-forth argument; see, e.g.,~\cite{Hodges}).
\end{proof}

\begin{obs}\label{obs:hom-equiv}
    Two countable $\omega$-categorical structures $\fA$ and $\fB$ are homomorphically equivalent if and only if $\CSP(\fA) = \CSP(\fB)$ (see, e.g.,~\cite[Lemma 2]{BodDalJournal}).
\end{obs}

The following definition is standard in model theory when phrased for structures; we present here a relation algebra version.

\begin{defi}[The Amalgamation Property]
\label{def:AP}
    Let $\bA$ be a relation algebra and let
    $\mathcal C$ be a class of consistent atomic $\bA$-networks.
    We say that $\mathcal C$ has
    \begin{itemize}
        \item $\AP(k,l,m)$ (the \emph{$(k,l,m)$-amalgamation property})
        if for any two networks $(V_1,f_1), (V_2,f_2) \in {\mathcal C}$
        with $|V_1|=k, |V_1 \cap V_2| =l$, and $|V_2|=m$, if $f_1(a,b)=f_2(a,b)$ for all $a,b \in V_1 \cap V_2$, then there exists a network
        $(V_1 \cup V_2,f) \in {\mathcal C}$ with
        $f(a,b) = f_i(a,b)$ for all $i \in \{1,2\}$ and $a,b \in V_i$;
        \item $\AP(n)$ if
        the property $\AP(k,l,m)$ holds for all
        $k,l,m$ with $k+m-l=n$;
        \item $\AP$ if the property $\AP(n)$ holds for all $n$;
        \item JEP (the \emph{Joint Embedding Property}) if
        $\AP(k,0,m)$ holds for all $k,m \in \mathbb N$.
    \end{itemize}
\end{defi}

\begin{thm}[{\cite[Theorem 8]{BodirskyRamics}}]
\label{thm:two-point-amalgamation}
    Let $\bA$ be a relation algebra with $k$ atoms and let ${\mathcal C}$ be the class of all consistent atomic $\bA$-networks. Then the following are equivalent.
    \begin{itemize}
        \item $\mathcal C$ has $\AP$;
        \item $\mathcal C$ has $\AP(1,0,1)$ and $\AP(k+1,k,k+1)$;
        \item $\bA$ has a normal representation.
    \end{itemize}
\end{thm}

In some cases we will have to describe the normal representation of a relation algebra concretely. We introduce the following notation,
which is used for example in Tables~\ref{tab:overview_asymmetric} and \ref{tab:overview_symmetric}.
\begin{itemize}
    \item $K_2^a$: the normal representation of $\ra{1}{2}$
    with the elements $0$ and $1$ where $a$ denotes the relation $\{(0,1),(1,0)\}$.
    \item $K_{\omega}^a$: the normal representation of $\ra{2}{2}$
    with domain $\N$ where $a$ denotes the relation
    $\{(x,y) \in \N^2 \mid x \neq y\}$.
    \item ${\mathbb Z}^{a,b}_5$: the up to isomorphism unique representation of $\ra{5}{7}$
    with the elements $\{0,1,2,3,4\}$ where $a$ denotes the undirected edge relation $\{(a,b)\colon |a-b|=\pm 1 \mod 5\}$,
    and $b$ denotes the edge relation of the complement graph.
    \item $\R^{a,b}$: the normal representation of $\ra{7}{7}$,
    which is obtained from the \emph{Rado graph}, i.e., the countable homogeneous graph whose age is the class of all finite undirected graphs, and where
    $a$ denotes the edge relation
    and $b$ denotes the edge relation of the complement graph.
    \item $\Q$: the normal representation of $1_3$ whose elements are the rational numbers and where $r$ denotes the strict linear order of the rationals.
    \item $C_3$: the normal representation of $\ra{2}{3}$ with three elements $0,1,2$ and where $r$ denotes the directed cyclic relation $\{(0,1),(1,2),(2,0)\}$.
    \item $C_4$: the normal representation of $\ra{18}{37}$ with four elements $0,1,2,3$, where $r$ denotes the directed cyclic relation $\{(0,1),(1,2),(2,3),(3,0)\}$ and $a$ denotes the relation $\{(0,2),(2,0),\allowbreak (1,3),(3,1)\}$.
    \item $\T$: the normal representation of $\ra{3}{3}$,
    which is obtained from the
    homogeneous tournament whose age is the class of all finite tournaments, and where $r$ denotes the edge relation of the tournament.
    \item $\hat{\T}$: the normal representation of $\ra{19}{37}$ whose elements are $\T \times \{0,1\}$ such that $a$ is the relation $\{((x,0),(x,1)), ((x,1),(x,0))\colon x\in \T\}$, and $((x,i),(y,j))$ is in the relation $r$ if and only if either $i=j$ and $(x,y)$ is in the relation $r$ of $\T$, or $i\neq j$ and $(y,x)$ is in the relation $r$ of $\T$.
    \item $P_n$ for $n \in \N^+$: the directed graph with vertices $\{0, \dots, n\}$ and edges $E = \{(k,k+1) \mid 0 \leq k \leq n-1\}$.
    \item $K_{2}^c \boxtimes_a K_{2}^b\colon$ the normal representation of $\ra{25}{65}$ with domain $\{0,1,2,3\}$ where $c = \{(0,1),\allowbreak(1,0),\allowbreak(2,3),\allowbreak(3,2)\}$, $b = \{(0,2),(2,0),(1,3),(3,1)\}$, and $a = \{(0,3),(3,0),(1,2),(2,1)\}$.
    \item $K_{\omega}^c \boxtimes_a K_{\omega}^b\colon$ the normal representation of $\ra{29}{65}$ with domain $\N \times \N$ where $c = \{((u, v), (u^\prime, v^\prime)) \mid u = u^\prime \text{~and~} v \neq v^\prime\}$, $b = \{((u, v), (u^\prime, v^\prime)) \mid u \neq u^\prime \text{~and~} v = v^\prime\}$, and $a = (\N\times\N)^2 \backslash (\id \cup\, c \cup b)$.
    \item $\mathbb P$: the normal representation of $\ra{15}{37}$, which is obtained from the homogeneous poset whose age is the class of all finite posets, and where $r$ denotes the strict partial order relation in the poset.
    \item $\mathbb D_2$ and $\mathbb D_\omega$: the normal representations of $\ra{20}{37}$ and $\ra{22}{37}$ which are obtained from the homogeneous bipartite tournament
    and the homogeneous $\omega$-partite tournament, respectively. (These are the Fra\"{i}ssé limits of orientations of complete bipartite resp. multipartite graphs.) In both cases, $a \cup \id$ is an equivalence relation. In $\mathbb D_2$ it has 2 equivalence classes, while in $\mathbb D_\omega$ it has infinitely many, and in both structures pairs of non-equivalent points lie in $r \cup \breve{r}$.
    \item $S(3)$: the normal representation of $\ra{23}{37}$, whose vertices are the points on the unit circle with rational argument and where $(x,y)$ is in the relation $r$ if and only if the angle from $x$ to $y$ is in the interval $(0, 2\pi/3)$.
\end{itemize}

Representations as well as reduced atomic networks of the relation algebras with an asymmetric atom $r$ and a symmetric atom $a$, i.e., $\ra{1}{37}$--$\ra{37}{37}$, can be viewed as oriented graphs (i.e., directed graphs with no directed 2-cycles), where $r$ represents the edge relation and $a$ represents the non-edge relation. This notation will be repeatedly used throughout the paper.

\subsection{Fully Universal Representations}
Some relation algebras do not have a normal representation, but still have a fully universal representation. Such representations can be conveniently constructed with the following theorem, which is essentially due to Comer~\cite{ComerExtensionSchemes}, but phrased there in a different formalism, and so we also sketch the proof.

Let $\mathcal C$ be a class of consistent atomic $\bA$-networks.
We say that a representation $\fB$ of $\bA$ is
\emph{${\mathcal C}$-universal} if an atomic $\bA$-network is satisfiable in $\fB$ if and only if it belongs to ${\mathcal C}$.

\begin{thm}\label{thm:fu}
    Let $\bA$ be a finite relation algebra
    and let ${\mathcal C}$ be a class of consistent atomic $\bA$-networks that is closed under isomorphisms and subnetworks and contains every consistent
    atomic $\bA$-network on at most three vertices.
    Then $\bA$ has
    \begin{itemize}
        \item a ${\mathcal C}$-universal representation
        if and only if $\mathcal C$ has AP$(3,2,n)$ for all $n$;
        \item a ${\mathcal C}$-universal square representation if and only if
        $\mathcal C$ has the JEP and
        AP$(3,2,n)$ for all $n$.
    \end{itemize}
\end{thm}
\begin{proof}[Proof sketch]
    For both statements, the proof can be obtained by adapting the proof of Fra\"{i}ssé's theorem (see, e.g.,~\cite{Hodges}).
    Starting from
    some network in $\mathcal C$, we
    use $\AP(3,2,n)$ to form a chain of larger and larger networks that eventually contains `all required composition witnesses'. Finally, we take an infinite disjoint union over the resulting network for all possible starting networks, and obtain a $\mathcal C$-universal representation.
    For the second statement, in our chain we interleave extension steps and joint embedding steps; in this case, we do not have to take a disjoint union, and obtain a $\mathcal C$-universal square representation.
\end{proof}

\begin{cor}\label{cor:fu}
    Let $\bA$ be a finite relation algebra.
    Then $\bA$ has a
    \begin{itemize}
        \item fully universal representation if and only if the class of all consistent atomic $\bA$-networks has AP$(3,2,n)$ for all $n$;
        \item fully universal square representation if and only if the class of all consistent atomic $\bA$-networks has the JEP and
        AP$(3,2,n)$ for all $n$.
    \end{itemize}
\end{cor}

\subsection{The 2-Cycle Product}
\label{sect:2cycle}

The 2-cycle product is a useful tool to construct more complex relation algebras from simpler ones; the construction preserves representability (and, as we will see, even fully universal representability, normal representability, etc.) and it will also be useful in complexity considerations later.
In the dream situation that $\bA$ and $\bB$ are
finite relation algebras with normal representations $\fC$ and $\fD$, respectively,
then the representation $\fR$ for the cycle product $\bA[\bB]$ that we construct will also be normal; moreover, the automorphism group of
$\fR$ can be described as a wreath product of
$\Aut(\fC)$ and $\Aut(\fD)$.

\begin{rem}\label{rem:clash}
    We note here an unfortunate clash of notation:
    while the relation algebra notation would suggest
    $\fC[\fD]$ as a name for $\fR$,
    it is standard in model theory to write $\fD[\fC]$
    for this structure (see, e.g., Cherlin~\cite{CherlinImPrim}), and indeed this notation is more suggestive, since we can imagine this structure as the structure obtained from $\fD$ by replacing each element by a copy of $\fC$.
    We will follow the model-theoretic convention.
\end{rem}

\begin{defi}\label{def:2cycl}
    Let $\bA$ and $\bB$ be two finite relation algebras such that $A_0 \cap B_0 = \{\id\}$. The \emph{2-cycle product} $\bA[\bB]$ is defined to be the up to isomorphism unique relation algebra $\bC$ such that $C_0 = A_0 \cup B_0$ and
    \[ \Cy(\bC) = \Cy(\bA) \cup \Cy(\bB) \cup \big \{(a,b,b), (\breve b,a,\breve b), (b,\breve b,a) \mid a \in A_0 \setminus \{\id\}, b \in B_0 \setminus \{\id\} \big \}.\]
\end{defi}

\begin{rem}\label{rem:equ_relation}
    If $\bC = \bA[\bB]$ is a 2-cycle product of relation algebras $\bA$ and $\bB$, then $1_\bA$ is an equivalence relation of $\bC$. Moreover, for any representation $\fC$ of $\bC$ and two distinct equivalence classes $C_1, C_2$ of $1_\bA^\fC$ with $x,y \in C_1$ and $z \in C_2$ we have $b^\fC(x,z)$ if and only if $b^\fC(y,z)$ for all $b \in B$.

    If $N$ is a consistent atomic $\bC$-network, then we write $\faktor{N}{1_\bA}$ for the atomic $\bB$-network $(V,g)$ whose variables $V$ are the equivalence classes of $1_\bA^N$ and where $g(C_1,C_2) := f(x_1,x_2)$ for $C_1 \neq C_2$ and some (equivalently, any) $x_1 \in C_1$ and $x_2 \in C_2$, and $g(C, C) = \id$ for all $C$.
\end{rem}

\begin{thm}[{\cite[Theorem 5]{Maddux2006}}, see also~\cite{ComerExtensionSchemes}]
\label{thm:2-cylce-repr}
    Let $\bA$ and $\bB$ be two finite representable integral relation algebras such that $A_0 \cap B_0 = \{\id\}$ and $|A_0|, |B_0| \geq 2$. Let $\fC$ and $\fD$ be square representations of $\bA$ and $\bB$, respectively.\footnote{It follows from Corollary~\ref{cor:simple-implies-square-rep} that every simple representable relation algebra has a square representation.}
    Then $\bA[\bB]$ is representable and has the representation $\fD[\fC]$ on the domain $D \times C$ defined as follows.\footnote{Warning: we intentionally write $\fD[\fC]$ rather than $\fC[\fD]$: see Remark~\ref{rem:clash}.} For all $a \in A_0 \setminus \{\id\}, b \in B_0 \setminus \{\id\}$:
    \begin{align*}
        \id^{\fD[\fC]} & := \{((u_0, v_0), (u_1, v_1)) \mid u_0 = u_1, v_0 = v_1 \}, \\
        a^{\fD[\fC]} & := \{((u_0, v_0), (u_1, v_1)) \mid u_0 = u_1, (v_0, v_1) \in a^{\fC} \big \}, \\
        b^{\fD[\fC]} & :=  \big \{((u_0, v_0), (u_1, v_1)) \mid (u_0, u_1) \in b^{\fD}\}.
    \end{align*}
\end{thm}

\begin{exa}[$\ra{2}{2}{[\ra{5}{7}]}$]
    Denote the non-identity atom of $\ra{2}{2}$ by $b$ and the non-identity atoms of $\ra{5}{7}$ by $a$ and $c$.
    Then the algebras $\ra{2}{2}$ and $\ra{5}{7}$ have the representations $K^b_\omega$ and $\Z_5^{a, c}$, respectively.
    The cycle product $\ra{2}{2}{[\ra{5}{7}]}$ is the algebra $\ra{17}{65}$, which has the representation $\Z_5^{a,c}[K^b_\omega]$ (depicted in Figure \ref{fig:1765_rep}).

\begin{figure}
\begin{center}
    \hspace{3cm}
    \begin{tikzpicture}
        \foreach \i in {1,2,3,4,5} {
            \coordinate (P\i) at ({90 + 72*(\i-1)}:3);
        }
        \draw (P1) -- (P2) node[midway, above] {$a$};
        \draw (P2) -- (P3) node[midway, left] {$a$};
        \draw (P3) -- (P4) node[midway, below] {$a$};
        \draw (P4) -- (P5) node[midway, right] {$a$};
        \draw (P5) -- (P1) node[midway, above] {$a$};
        \draw (P1) -- (P3) node[midway, right] {$c$};
        \draw (P1) -- (P4) node[midway, left] {$c$};
        \draw (P2) -- (P4) node[midway, below] {$c$};
        \draw (P2) -- (P5) node[midway, above] {$c$};
        \draw (P3) -- (P5) node[midway, below] {$c$};
        \foreach \i in {1,2,3,4,5} {
            \draw[thick, fill=white] (P\i) circle (0.9);
            \node at (P\i) { };
            \foreach \j in {1,2,3,4} {
                \coordinate (P\i\j) at ($ (P\i) + ({90*(\j-1)-45}:0.5) $);
            }
            \foreach \j in {9,10,11,12,13,14,15,16,17,18,19,20,21,22} {
                \coordinate (P\i\j) at ($ (P\i) + ({25.714*(\j-9)}:0.9) $);
            }

            \foreach \j in {9,10,11,12,13,14,15,16,17,18,19,20,21,22} {
              \foreach \k in {9,10,11,12,13,14,15,16,17,18,19,20,21,22} {
                    \ifnum\j<\k
                        \draw (P\i\j) -- (P\i\k) node[midway, above] { };
                    \fi
                }
            }
        }
        \coordinate (L1) at ($ (P5) + ({15}:0.9) $);
        \coordinate (L) at ($ (P5) + ({15}:3) $);
        \node[right] at (L) {$K^b_\omega$};
        \draw (L1) -- (L) node[midway, above] { };
    \end{tikzpicture}
    \caption{A representation of $\ra{17}{65} = \ra{2}{2}{[\ra{5}{7}]}$ is given by $\Z_5^{a,c}[K^b_\omega]$.}
    \label{fig:1765_rep}
\end{center}
\end{figure}
\end{exa}

\section{The Network Satisfaction Problem}
\label{sect:NSP}
The network satisfaction problem for a fixed finite relation algebra $\bA$ is the following computational problem, denoted by $\NSP(\bA)\colon$
The input consists of an $\bA$-network $(V,f)$. The task is to decide whether $\bA$ has a representation $\fB$
such that there exists a function $s \colon V \to B$ which satisfies
$(s(x),s(y)) \in f(x,y)^{\fB}$ for all $x,y \in V$; in this case the
$\bA$-network is called \emph{satisfiable}.

We can reduce the complexity classification for network satisfaction problems to the case where the relation algebra is simple.

\begin{lem}[{\cite[Lemma 2.10]{BodirskyKnaeuerDatalog23}}]
\label{lem:prod1}
    Let $\bA$ and $\bB$ be finite representable relation algebras.
    Then there exists a polynomial-time reduction from $\NSP(\bA)$ to $\NSP(\bA \times \bB)$.
\end{lem}

This lemma has a converse; we are not aware of any reference for this lemma, but it will be highly useful in our classification project.

\begin{lem}
\label{lem:prod2}
    Let $\bA_1$ and $\bA_2$ be finite representable relation algebras
    such that $\NSP(\bA_1)$ is in $\Ptime$ and $\NSP(\bA_2)$ is in $\Ptime$. Then $\NSP(\bA_1 \times \bA_2)$ is in $\Ptime$. The same statement holds if we replace $\Ptime$ by $\NP$.
\end{lem}
\begin{proof}
    Let $(V,f)$ be an $(\bA_1 \times \bA_2)$-network. For $i \in \{1,2\}$, we compute an $\bA_i$-network $(V,f_i)$ as follows. Define $f_i(x,y) := f(x,y)_i$. We run our polynomial-time algorithm for $\NSP(\bA_i)$ on $(V,f_i)$.
    We claim that $(V,f)$ is satisfiable if and only if one of $(V,f_1)$ and  $(V,f_2)$ is satisfiable.

    For $i \in \{1,2\}$, let $\fB_i$ be a representation of $\bA_i$.
    Suppose that $\bA_1$ has some representation $\fB_1'$ such that there exists a function $s_1 \colon V \to B_1'$ which satisfies $(s_1(x),s_1(y)) \in f_1(x,y)^{\fB_1'}$ for all $x,y \in V$. We may suppose without loss of generality that $B_1' \cap B_2 = \varnothing$.
    From Lemma~\ref{lem:union}
    we know that $\fB_1' \uplus \fB_2$ is a representation of $\bA_1 \times \bA_2$.
    Then $(s_1(x),s_1(y)) \in f(x,y)^{\fB_1' \uplus \fB_2}$,
    and hence $(V,f)$ is satisfiable. If $(V,f_2)$ is satisfiable, the proof
    that $(V,f)$ is satisfiable is analogous.

    Now suppose that conversely $(V,f)$ is satisfiable, i.e.,
    there exists a representation $\fB$ of $\bA_1 \times \bA_2$
    and a function $s \colon V \to B$ with $(s(x),s(y)) \in f(x,y)^{\fB}$ for all $x,y \in V$. By Lemma~\ref{lem:factor}, there are
    representations $\fB_1$ of $\bA_1$ and $\fB_2$ of $\bA_2$
    such that $\fB$ is isomorphic to $\fB_1 \uplus \fB_2$.
    Since $f$ is a total function on $V^2$, the image of $s$ must fully lie in $B_1$
    or in $B_2$; suppose without loss of generality that the image is in $B_1$.
    Then the representation $\fB_1$ together with $s$ shows that $(V,f_1)$ is satisfiable.
\end{proof}

\begin{rem}\label{rem:red-to-simple}
In order to show that the network satisfaction problems for relation algebras $\bA$ with at most four atoms are in $\Ptime$ or $\NPc$, it suffices to prove this for simple relation algebras $\bA$,
since by Remark~\ref{rem:simple},
the result then follows by Lemma~\ref{lem:prod1} and Lemma~\ref{lem:prod2}.
\end{rem}

The remarkable aspect of the following lemma is that it also applies to relation algebras that do not have a fully universal representation. Also, for this lemma, we are not aware of any reference.

\begin{lem}\label{lem:bounded}
Let $\bA$ be a finite simple relation algebra such that the maximum cardinality of all square representations of $\bA$ equals $n \in \N$. Then $\NSP(\bA)$ is in $\NP$.
If $\bA$ additionally has a square representation with at least three elements, then $\NSP(\bA)$ is $\NP$-complete.
\end{lem}
\begin{proof}
    Containment in $\NP$: for a given $\bA$-network $(V,f)$, non-deterministically guess a square representation $\fB$ of $\bA$ (we have to choose from finitely many finite options) and then guess the mapping $s \colon V \to B$; whether the mapping satisfies the required properties can be verified in polynomial time.
    We claim that $(V,f)$ is satisfiable if and only if for some non-deterministic choice the function $s$ has the required properties.
    Indeed, if $\bA$ has such a square representation $\fB$ and a map then clearly $(V,f)$ is satisfiable. Conversely, if there exists some representation $\fB$ of $\bA$ where $(V,f)$ is satisfiable,
    then $1^{\fB}$ is an equivalence relation $1^{\fB} = \bigcup_{i \in I} E_i^2$; the structure $\fB$
    induces on each equivalence class
    $E_i$ a square representation of some relation algebra $\bA_i$.
    By Lemma~\ref{lem:decomp} we have a homomorphism from $\bA$ to $\bA_i$, and by the simplicity of $\bA$ the homomorphism is in fact an isomorphism. Therefore, $(V,f)$ is also satisfiable in some square representation of $\bA$.

    $\NP$-hardness: by reduction from $n$-colouring, which is $\NP$-hard for $n \geq 3$. A given undirected graph $G$ is reduced to the network $(V(G),f)$ where $f(x,y) = \overline{\id}$
    if $(x,y) \in E(G)$, and $f(x,y) = 1$ otherwise.
 \end{proof}

\begin{lem}\label{lem:fin_bounded_np}
    Let $\bA$ be a finite relation algebra with a finitely bounded universal representation $\fB$. Then $\NSP(\bA)$ is in $\NP$.
\end{lem}
\begin{proof}
    By Observation~\ref{obs:fin_bounded_networks} there is a finite set $\mathcal F'$ of consistent atomic $\bA$-networks such that a consistent atomic $\bA$-network $(V,f)$ is satisfiable in $\fB$ if and only if no member of $\mathcal F'$ is isomorphic to a subnetwork of $(V,f)$.

    For a given $\bA$-network $(V,f)$, non-deterministically guess an atomic network $(V,g)$ with $g(x,y) \leq f(x,y)$ for all $x,y \in V$ and verify, in polynomial time, that $(V,g)$ is consistent and that no subnetwork of $(V,g)$ is isomorphic to a member of $\mathcal F'$.
\end{proof}

\subsection{The Network Consistency Problem}
Unlike the network satisfaction problem of a finite relation algebra $\bA$, which can be undecidable~\cite{Hirsch-Undecidable}, the
\emph{network consistency problem} for $\bA$, denoted by $\NCP(\bA)$, which we introduce below, is always in $\NP$ (in fact, it is expressible in the logic monotone SNP; see~\cite{Book,HerFO}).

The input of the network consistency problem consists of an $\bA$-network $(V,f)$. The task is to decide whether there exists a consistent atomic network $(V,s)$ (which will also be called the \emph{solution} to the given problem instance) on the same set of vertices $V$ such that
for all $x,y \in V$ we have $s(x,y) \leq f(x,y)$. In this case, we say that $(V,f)$ is \emph{solvable}, and \emph{unsolvable} otherwise.
Note that if an $\bA$-network $(V,f)$ is unsolvable, then it is also unsatisfiable as an instance of the network satisfaction problem.

\begin{lem}\label{lem:nsp-ncp}
    If $\bA$ is a finite relation algebra with a fully universal representation, then the network satisfaction problem for $\bA$ equals the network consistency problem for $\bA$.
\end{lem}
\begin{proof}
    Let $\fB$ be the fully universal representation of $\bA$.
    Let $N$ be an $\bA$-network.
    If $N$ is satisfiable in $\fB$, then
    the network consistency problem clearly has a solution.
    Otherwise, if the network consistency problem has a solution, then the solution is satisfiable in $\fB$, since $\fB$ is fully universal, and hence $N$ is a positive instance to $\NSP(\bA)$.
\end{proof}

\begin{cor}
\label{cor:fully-univ-NP}
    Let $\bA$ be a finite relation algebra with a fully universal representation $\fB$. Then $\NSP(\bA)$ is in $\NP$.
\end{cor}

\begin{rem}
    The converse of Lemma~\ref{lem:nsp-ncp} is true as well:
    If $\NSP(\bA) = \NCP(\bA)$, then there exists a fully universal representation (which in general might not be square).
    To see this, we may pick by assumption for every isomorphism type of a consistent atomic $\bA$-network a representation of $\bA$ where the network is satisfiable.
    Let $(\fB_i)_{i \in I}$ be the (countable) collection of structures obtained in this way.
    Then the disjoint union over all the $\fB_i$ (a special case of Definition~\ref{def:prod}) is a representation of $\bA$
    and clearly this representation is fully universal.
\end{rem}

\begin{lem}\label{lem:col}
Let $\bA$ be a relation algebra with a fully universal square representation $\fB$ having $n \in \{3,4,\dots\}$ elements. Then $\NSP(\bA)$ is $\NP$-complete.
\end{lem}
\begin{proof}
    We have already established containment in $\NP$. $\NP$-hardness can be shown by reduction from $n$-colouring as in the proof of Lemma~\ref{lem:bounded}.
\end{proof}
\begin{rem}\label{rem:tractability_conj_1_7_and_2_3}
    If the square representation $\fB$ of some relation algebra $\bA$ has $n \in \{3,4,\dots\}$ elements, then we obtain a one-dimensional pp-interpretation of $K_n$, the complete graph on $n$ vertices, in $\fB$ by defining the interpreting formula for the edge relation as $\varphi(x, y) := (1\backslash\id)^\fB(x, y)$. It is well-known that $K_n$ for $n \geq 3$ in turn pp-constructs $(\{0, 1\}; \NAE)$, which verifies the tractability conjecture (Conjecture \ref{conj:tractability}) for the algebras $\ra{2}{3}$ and $\ra{1}{7}$.
\end{rem}

\begin{rem}\label{rem:reduce} It is well-known that if $\bA$ is a relation algebra and
$S \subseteq A$ is such that every element of $\bA$
can be obtained from $S$ using repeated applications of $\circ$, ${\Breve{\phantom{o}}}$, and intersection, then
the network satisfaction problem
can be reduced to the special case that in the input network $(V,f)$ the image of $f$ is restricted to $S \cup \{1,\id\}$:
for instance, if $a,b \in S$, and the input network $(V,f)$ is such that
\begin{itemize}
    \item
$f(x,y) = a \circ b$, then we define a new network
$(V \cup \{z\},g)$ where $z$ is a new variable and $g$ is defined as follows: $g(u,v) := f(u,v)$ for all $(u,v) \in V^2 \setminus \{(x,y)\}$, $g(x,z) := a$, $g(z,y) := b$,
and $g(u,v) := 1$ otherwise.
\item $f(x,y) = a \cap b$, then we define a new network
$(V \cup \{z\},g)$ where $z$ is a new variable and $g$ is defined as follows: $g(u,v) := f(u,v)$ for all $(u,v) \in V^2 \setminus \{(x,y)\}$, $g(x,z) := \id$, $g(x,y) := a$, $g(z,y) := b$, and $g(u,v) := 1$ otherwise.
\end{itemize}
It is then easy to see that $(V,f)$ is satisfiable in some representation if and only if $(V \cup \{z\},g)$ is satisfiable in some representation.
\end{rem}

\subsection{The Path Consistency Algorithm}
\label{sect:PC}
The path consistency algorithm is a particularly natural polynomial-time algorithm in the context of the network consistency problem and the network satisfaction problem for a fixed relation algebra $\bA$.
The input to the algorithm is an $\bA$-network $(V,f)$;
the algorithm repeatedly considers three variables $u,v,w$ and updates $f(u,v)$ by intersecting it with $f(u,w) \circ f(w,v)$, while keeping track of converse closure; see Algorithm~\ref{alg:pc}.

\begin{algorithm}
\caption{Path Consistency (PC)}
\label{alg:pc}
\begin{algorithmic}[1]
    \Statex \textbf{Input:} Relation algebra $\bA$ and $\bA$-network
    $N = (V,f)$
    \For{$x \in V$}
        \State $f(x,x) \gets f(x,x) \cap \id$
    \EndFor
    \Repeat
        \For{$(x,y,z) \in V^3$}
            \State $f(x,y) \gets f(x, y) \cap (f(x, z) \circ f(z, y))$
            \State $f(y,x) \gets f(y,x) \cap f(x, y)\breve{~}$
        \EndFor
    \Until{$N$ does not change}
    \If{$f(x,y) = 0$ for some $x,y \in V$}
        \State \textbf{return} $N$ is unsolvable
    \Else
        \State \textbf{return} $N$
    \EndIf
\end{algorithmic}
\end{algorithm}

\begin{rem}\label{rem:PC-sound}
If the algorithm returns `unsolvable' on a given $\bA$-network $N$, then $N$ does not have a solution when viewed
as an instance of $\NCP(\bA)$. Therefore, $N$ is an unsatisfiable instance of $\NSP(\bA)$.
\end{rem}

The output network $N$ of the algorithm is called \emph{path consistent}.
We say that path consistency \emph{solves} $\NSP(\bA)$ if  $(V,f)$ is satisfiable if and only if the algorithm does not return `unsolvable'.

\begin{exa}\label{expl:non-int-P}
    It is easy to see that if $\bA$ is the relation algebra from Example~\ref{expl:non-int}, then path consistency solves $\NSP(\bA)$.
    Since Example~\ref{expl:non-int} is the only simple non-integral relation algebra with at most four atoms (Proposition~\ref{prop:non-int-unique}), we can later focus on integral relation algebras (see Remark~\ref{rem:red-to-simple}).
\end{exa}

\subsection{Network Satisfaction Problems and CSPs}
If $\fB$ is a fully universal representation of a finite relation algebra $\bA$, then it is easy to see that $\CSP(\fB)$ is essentially the same problem as $\NSP(\bA)$ (which then equals $\NCP(\bA)$).

However, sometimes a different constraint satisfaction problem can be used to obtain polynomial-time tractability results for $\NSP(\bA)$.

\begin{defi}\label{def:atom-struct}
    Let $\bA$ be a finite relation algebra. The \emph{atom structure} of $\bA$ is the finite relational structure $\fAo$ with domain $A_0$ and the following relations:
    \begin{itemize}
        \item for every $b \in A$ the unary relation $b^{\fAo}:= \{a \in A_0 \mid a \leq b\}$,
        \item the binary relation $E^{\fAo} := \{(a_1, a_2) \in A_0^2
        \mid \Breve{a_1} = a_2\}$, and
        \item the ternary relation $R^{\fAo} := \Cy(\bA)$.
    \end{itemize}
\end{defi}

Clearly, the atom structure $\fA_0$ of a relation algebra $\fA$ is conservative: if $f \in \Pol(\fA_0)$ has arity $k$ and $a_1,\dots,a_k \in A_0$, then $\fA_0$ contains the relation $b^{\fA_0}$ for $b = a_1 \cup \cdots \cup a_k$.
So $f$ must preserve this relation, which means that $f(a_1,\dots,a_k) \in b^{\fA_0} = \{a_1,\dots,a_k\}$.
The utility of atom structures is explained by the following:

\begin{prop}[{\cite[Proposition 2.16]{BodirskyKnaeuerDatalog23}}]
\label{prop:reductionAtomStructure}
    Let $\bA$ be a finite relation algebra with a fully universal representation $\fB$. Then there is a polynomial-time reduction from $\NSP(\bA)$ to $\CSP(\fAo)$.
\end{prop}

Hence, if $\CSP(\fA_0)$ is in $\Ptime$, then so is $\NSP(\bA)$.
The good news here is that there is a complete complexity classification of finite domain CSPs, as we mentioned in the introduction; since $\fA_0$ is conservative, we can even apply Theorem \ref{thm:dichotomyConservative}. The bad news is that
$\CSP(\fA_0)$ might be $\NPc$, while $\NSP(\bA)$ is still in $\Ptime$; this happens for instance for the point algebra $\ra{1}{3}$~\cite{PointAlgebra}. So we need more tools to prove polynomial-time tractability.

Viewing network satisfaction problems of relation algebras with fully universal representations as CSPs is not only useful for proving polynomial-time tractability, but also $\NP$-hardness in a number of cases. This will mainly be done in Section~\ref{sec:pp_reductions}, but we introduce a key concept here:

\begin{defi}\label{def:xy-canonical}
    Let $\fB$ be a normal representation of a finite relation algebra $\bA$ and $X \subseteq A_0$. An operation $f \colon B^n \to B$ is called \emph{$X$-canonical (with respect to $\fB$)} if there exists a function $\overline{f} \colon X^n \to A_0$
    such that for all $x, y \in B^n$ and $a_1, \dots, a_n \in X$, if $(x_i, y_i) \in a_i^\fB$ for all $i \in \{1, \dots, n\}$, then $(f(x), f(y)) \in \overline{f}(a_1,\dots, a_n)^\fB$. The function $\overline{f}$ is called the \emph{behaviour of $f$ on $X$.} An operation $f$ is called \emph{canonical (with respect to $\fB$)} if it is $A_0$-canonical.

    A binary $X$-canonical function $f$ is called \emph{$X$-symmetric} if the behaviour of $f$ on $X$ is symmetric. A ternary $X$-canonical function $f$ is called an \emph{$X$-majority} if the behaviour of $f$ on $X$ is a majority operation, and it is called an \emph{$X$-minority} if the behaviour of $f$ on $X$ is a minority operation.
\end{defi}

\section{Representations for Algebras with at most 4 Atoms}

In this section we revisit the representability for relation algebras with at most 4 atoms. By Lemma~\ref{lem:union} (also see Remark~\ref{rem:simple}) we can focus on simple relation algebras.
In Section~\ref{sect:int} we will show that
we may even focus on integral relation algebras.
Exhaustive lists of these algebras have been computed by Maddux~\cite{Maddux2006-dp}. There are
\begin{itemize}
    \item 2 integral relation algebras with two atoms (Table~\ref{fig:2Atoms}),
    \item 3 integral asymmetric relation algebras with three atoms (Table~\ref{fig:3AsyAtoms}),
    \item 7 integral symmetric relation algebras with three atoms (Table~\ref{fig:3SymAtoms}),
    \item 37 integral asymmetric relation algebras with four atoms (Table~\ref{tab:overview_asymmetric}), and
    \item 65 integral symmetric relation algebras with four atoms (Table~\ref{tab:overview_symmetric}).
\end{itemize}

Since the identity atom behaves the same in all the algebras, it is not needed to distinguish the algebras and it is therefore not shown in the lists. The fact that there are exactly $102 = 37 + 65$ integral relation algebras with four atoms was already known to Comer~\cite{ComerChromaticPolygroups} (discovered in a different formalism with different terminology, but equivalent by a result from~\cite{Comer1984}; also see~\cite{Hypergroups}).

We will first recall results of Maddux
which show that 31 of these relation algebras do not have a representation; the smallest algebras that are not representable have 4 atoms. We will then see that all the remaining 71 integral relation algebras do have a representation; this fact should be considered to be known and we do not claim originality.
Several authors contributed representations, including S. Comer, P. Hertzel, P. Jipsen, R. Kramer, E. Lukács, and R. Maddux,\footnote{We are grateful for helpful email exchange with P. Jipsen and R. Maddux.} and to the best of our knowledge not all of the respective work has been published.
For us it will also be important whether there exists a fully universal representation, or even a normal representation, because this is relevant for analyzing the complexity of the NSP, as we have explained earlier.
We test which of the relation algebras with at most 4 atoms have a normal representation, using Theorem~\ref{thm:two-point-amalgamation}.
If they do not have a normal representation, we search for a fully universal representation using Theorem~\ref{thm:fu}.
It turns out that there are only 9 integral representable relation algebras that do not have a fully universal square representation. Some of them have the property that the cardinality of all square representations is bounded by some finite $n \in {\mathbb N}$, which is useful because of Lemma~\ref{lem:bounded}.
The remaining algebras will have to be inspected closely in a case-by-case fashion. See Table~\ref{fig:numbers}.

\begin{table}
\centering
        \begin{tabular}{|llllllllll|}\hline
            Atoms & Elements & RAs & Simple & Int. & Sym./Asym. &
            Repr. & F.u.s. & Norm. & Flex. \\
            \hline
            0 & 1 & 1 & 1 & 0 & 1/0 & 1 & 1 & 1 & 0 \\
            1 & 2 & 1 & 1 & 1 & 1/0 & 1 & 1 & 1 & 0 \\
            2 & 4 & 3 & 2 & 2 & 2/0 & 2 & 2 & 2 & 1 \\
            3 & 8 & 13 & 10 & 10 & 7/3 & 7/3 & 6/3 & 6/3 & 2/1 \\
            4 & 16 & 119 & 103 & 102 & 65/37 & 45/26 & 37/25 & 34/23 & 10/5 \\ \hline
        \end{tabular}
    \caption{The number of (simple, integral) relation algebras with 0--4 atoms and 1--16 elements. For the integral relation algebras, we specify how many of them have a (fully universal square) representation, which of them are normal and which have a flexible atom (see Definition~\ref{def:flex}), each time distinguishing the number of symmetric and the number of asymmetric relation algebras.}
    \label{fig:numbers}
\end{table}

\subsection{Reduction to the Integral Case}
\label{sect:int}
In this section we show that there is up to isomorphism only one simple non-integral relation algebra with at most four atoms. This fact might be known, but we are not aware of a reference in the literature.

\begin{prop}\label{prop:non-int-unique}
    The algebra from Example~\ref{expl:non-int} is up to isomorphism the only simple non-integral relation algebra with at most four atoms.
\end{prop}
\begin{proof}
    The simple relation algebras with at most three atoms have been determined by Andréka and Maddux \cite{AndrekaMaddux}, and all of them are integral.
    So let $\bA$ be a simple non-integral relation algebra with four atoms. We already mentioned that a simple relation algebra is integral if and only if $\id$ is an atom~\cite[Theorem 353]{Maddux2006}. Hence, $\bA$ contains at least two identity atoms $a$ and $b$. Denote the other two atoms of $\bA$ by $c$ and $d$. For two distinct identity atoms $x$ and $y$, we have $x \circ y \leq x \circ \id = x$ and $x \circ y \leq \id \circ\, y = y$, so $x \circ y \leq x \cap y = 0$. In particular, $a \circ b = b \circ a = 0$.
    Moreover, observe that
    \[ a = a \circ \id = a \circ \bigcup_{\substack{x \leq \id \\ x \text{ atom}}} x = \bigcup_{\substack{x \leq \id \\ x \text{ atom}}} a \circ x. \]
    Since $a \circ x = 0$ for all identity atoms $x \neq a$, we conclude that $a = a \circ a$. Analogously, we obtain $b \circ b = b$.
    \begin{figure}[H]
    \centering
    \begin{tabular}{|L|LLLL|}\hline
        \circ & a & b & c & d \\ \hline
        a & a & 0 & &  \\
        b & 0 & b & &  \\
        c & & & & \\
        d & & & & \\ \hline
    \end{tabular}
\end{figure}

It follows from $a \leq a \circ a$ and the cycle law that $a \leq \breve{a} \circ a$, particularly $\breve{a} \circ a \neq 0$. It is easy to show that $\breve{a}$ has to be an identity atom too. Since $x \circ a = 0$ for all identity atoms $x \neq a$, we conclude that $\breve{a} = a$. Analogously, we get $\breve{b} = b$.

We know that $a \circ c \leq \id \circ c = c$ and $a \circ d \leq \id \circ d = d$. Assume towards a contradiction that $a \circ c = 0$ and $a \circ d = 0$. It follows that $\breve{c} \circ a = 0$ and $\breve{d} \circ a = 0$, which implies $1 \circ a \circ 1 = a$, a contradiction to $\bA$ being simple. Since $c$ and $d$ are indistinguishable up to this point, we can assume without loss of generality that $a \circ d = d$.

As above, we also know that $b \circ c = c$ or $b \circ d = d$. Assume the latter holds. Then $0 = (a \circ b) \circ d = a \circ (b \circ d) = a \circ d = d$, contradicting the associativity of $\circ$; hence, $b \circ c = c$ and $b \circ d = 0$.
Suppose that $a \circ c = c$ holds, then $0 = (a \circ b) \circ c = a \circ (b \circ c) = a \circ c = c$, which is again a contradiction.
\begin{figure}[H]
    \centering
    \begin{tabular}{|L|LLLL|}\hline
        \circ & a & b & c & d \\ \hline
        a & a & 0 & 0 & d \\
        b & 0 & b & c & 0 \\
        c & & & & \\
        d & & & & \\ \hline
    \end{tabular}
\end{figure}

Next, suppose to the contrary that $d = \breve{d}$ (and thus also $c = \breve{c}$). Then $d \circ a = (a \circ d){\Breve{\phantom{o}}} = \breve{d} = d$. Similarly, we get $c \circ a = 0$, $c \circ b = c$, and $d \circ b = 0$. It follows that $1 \circ a \circ 1 = (a \cup d) \circ 1 = (a \circ 1) \cup (d \circ 1) = (a \cup d) \cup (d \circ c) \cup (d \circ d)$. Since this expression has to equal $1$ due to $\bA$ being simple, we infer that $b \leq d \circ c$ or $b \leq d \circ d$. By applying the cycle law, this implies $c \leq b \circ d = 0$ or $d \leq b \circ d = 0$, which are both contradictions. Hence, $c = \breve{d}$.

It follows that
$c \circ a = (a \circ d){\Breve{\phantom{o}}} = c$,
$d \circ a = (a \circ c){\Breve{\phantom{o}}} = 0$,
$c \circ b = (b \circ d){\Breve{\phantom{o}}} = 0$, and
$d \circ b = (b \circ c){\Breve{\phantom{o}}} = d$.
We can now compute $c \circ c = c \circ (b \circ c) = (c \circ b) \circ c = 0 \circ c = 0$. Hence, also $d \circ d = (c \circ c){\Breve{\phantom{o}}} = 0$.

For the last two missing entries in our table, observe that
$d \circ c = d \circ (c \circ a) = (d \circ c) \circ a \leq 1 \circ a = a \cup c$
and
$d \circ c = (a \circ d) \circ c = a \circ (d \circ c) \leq a \circ 1 = a \cup d$; hence, $d \circ c \leq a$. Assume towards a contradiction that $d \circ c = 0$. Then $1 \circ d \circ 1 = 1 \circ d = d \cup (c \circ d)$, and this union has to equal $1$ by the simplicity of $\bA$. In particular, $a \leq c \circ d$. By the cycle law, this implies $\breve{d} \leq a \circ c = 0$, a contradiction. Hence, $d \circ c = a$. With an analogous argument, we obtain $c \circ d = b\colon$
\begin{figure}[H]
    \centering
    \begin{tabular}{|L|LLLL|}\hline
        \circ & a & b & c & d \\ \hline
        a & a & 0 & 0 & d \\
        b & 0 & b & c & 0 \\
        c & c & 0 & 0 & b \\
        d & 0 & d & a & 0 \\ \hline
    \end{tabular}
\end{figure}
Hence, $\bA$ is isomorphic to the relation algebra in Example~\ref{expl:non-int}.
\end{proof}

\subsection{Without Representation}
\label{sec:no_repr}
Comer~\cite{ComerChromaticPolygroups} states that at least 28 of the 102 integral relation algebras with four atoms are not representable. Results of Maddux~\cite{Maddux2006-dp} imply that in fact 31 do not have a representation. (And, as we have mentioned earlier, all others turn out to have a representation.)

\begin{thm}
\label{thm:non-repr}
    The following 31 integral relation algebras have no representation:
    \begin{itemize}
        \item
        the 11 asymmetric relation algebras $\ra{14}{37}$, $\ra{16}{37}$, $\ra{21}{37}$,
        $\ra{24}{37}$--$\ra{29}{37}$, $\ra{32}{37}$, $\ra{34}{37}$;
        \item
        the 20 symmetric relation algebras $\ra{21}{65}$--$\ra{23}{65}$, $\ra{35}{65}$--$\ra{38}{65}$,
        $\ra{40}{65}$--$\ra{45}{65}$, $\ra{47}{65}$--$\ra{50}{65}$, $\ra{54}{65}$, $\ra{58}{65}$,
        $\ra{60}{65}$.
    \end{itemize}
    Therefore, their $\NSP$s are trivial and in $\Ptime$.
\end{thm}
\begin{proof}
    See Chapters 6.64, 6.66, and (for the cases $\ra{32}{37}$ and $\ra{60}{65}$) 6.55 in \cite{Maddux2006-dp}.
\end{proof}

\subsection{Normal Representations}
\label{sec:normal_repr}

Examples of finite relation algebras with a normal representation include those with a \emph{flexible atom} (see, e.g.,~\cite{BodirskyKnaeJAIR}).

\begin{defi}\label{def:flex}
    Let $\bA$ be a relation algebra. An atom $a \in A_0\backslash\{\id\}$ is called a \emph{flexible atom} if $a \leq x \circ y$ holds for all $x,y \in A_0\backslash\{\id\}$.
\end{defi}

The relation algebras with a flexible atom are listed in Table~\ref{fig:flex}. Note that if a relation algebra has a flexible atom then it has a normal representation, as one can, when applying Theorem~\ref{thm:two-point-amalgamation}, always use a flexible atom for all undetermined values.

\begin{table}
\centering
\begin{tabular}{|L|L||L|L|}\hline
    \text{Algebra} & \text{Flexible Atoms} & \text{Algebra} & \text{Flexible Atoms} \\ \hline
    \ra{2}{2}   & a                 & \ra{32}{65} & a \\ \hline
    \ra{3}{3}   & r, \breve{r}      & \ra{33}{65} & a \\ \hline
    \ra{6}{7}   & a                 & \ra{34}{65} & a \\ \hline
    \ra{7}{7}   & a, b              & \ra{55}{65} & a \\ \hline
    \ra{31}{37} & a                 & \ra{57}{65} & a \\ \hline
    \ra{33}{37} & a                 & \ra{59}{65} & a \\ \hline
    \ra{35}{37} & a                 & \ra{61}{65} & a \\ \hline
    \ra{36}{37} & r, \breve{r}      & \ra{63}{65} & a \\ \hline
    \ra{37}{37} & a, r, \breve{r}   & \ra{64}{65} & a, b \\ \hline
                &                   & \ra{65}{65} & a, b, c \\ \hline
\end{tabular}
    \caption{Integral relation algebras with at most  four atoms and at least one flexible atom.}
    \label{fig:flex}
\end{table}

With the criterion presented in Theorem~\ref{thm:two-point-amalgamation} it is easy to see (by hand; we verified the condition in all cases using
the program provided in~\cite{RA4Code}) that
the following relation algebras with at most four atoms have a normal representation:

\begin{align}
\begin{split}
    \ra{1}{2},\ra{2}{2}, \ra{1}{3}, \ra{2}{3}, \ra{3}{3}, \ra{1}{7}, \ra{2}{7}, \ra{3}{7}, \ra{4}{7}, \ra{6}{7}, \ra{7}{7}, \\
    \ra{1}{37}\text{--}\ra{12}{37}, \ra{15}{37}, \ra{18}{37}\text{--}\ra{20}{37}, \ra{22}{37}\text{--}\ra{23}{37}, \ra{31}{37},\ra{33}{37}, \ra{35}{37}\text{--}\ra{37}{37}, \\
    \ra{1}{65}\text{--}\ra{8}{65}, \ra{10}{65}, \ra{11}{65}, \ra{13}{65}, \ra{14}{65}, \ra{16}{65}, \ra{18}{65}\text{--}\ra{20}{65}, \ra{25}{65}\text{--}\ra{29}{65}, \\
    \ra{32}{65}\text{--}\ra{34}{65}, \ra{46}{65}, \ra{52}{65}, \ra{53}{65}, \ra{55}{65}, \ra{57}{65}, \ra{59}{65}, \ra{61}{65}, \ra{63}{65}\text{--}\ra{65}{65}. \label{eq:norm_all}
\end{split}
\end{align}

In Table~\ref{tab:amalgamation-failure}, we give for every representable algebra which is not in \eqref{eq:norm_all} a counterexample for the condition in Theorem~\ref{thm:two-point-amalgamation} using the notation of Figure~\ref{fig:2-pt-amal-size-5}.

\begin{figure}
    \centering
    \[\begin{tikzcd}[row sep=2.5em, column sep=3.5em]
	&& \bullet \\
	\bullet && \bullet && \bullet \\
	&& \bullet
	\arrow["{b_1}"', from=1-3, to=2-1]
	\arrow["{a_1}", swap, from=1-3, to=2-3]
	\arrow["{c_1}", from=1-3, to=2-5]
	\arrow["{a_3}"{pos=0.4}, curve={height=-18pt}, from=1-3, to=3-3]
	\arrow["{b_2}"', from=2-3, to=2-1]
	\arrow["{c_2}", from=2-3, to=2-5]
	\arrow["{a_2}", swap, from=2-3, to=3-3]
	\arrow["{b_3}", from=3-3, to=2-1]
	\arrow["{c_3}", swap, from=3-3, to=2-5]
\end{tikzcd}\]
    \caption{A 2-point-amalgamation diagram of size 5.}
    \label{fig:2-pt-amal-size-5}
\end{figure}

\begin{table}
\centering
\begin{tabular}{|L|L|l|} \hline
    \text{Algebra} & (a_1, a_2, a_3), (b_1, b_2, b_3), (c_1, c_2, c_3) & \text{Fully Universal}\\ \hline
    \ra{5}{7} &  (a,a,b), (a,b,a), (a,b,b) & No, Figure~\ref{fig:evil_square}\\ \hline
    \ra{13}{37} & (a,a,a), (r,a,r), (a,r,r) & Yes, Prop.~\ref{prop:13_37-repr}\\ \hline
    \ra{17}{37} & (r,r,r), (r,\breve{{r}},r), (a,a,\breve{{r}}) & No, Lemma~\ref{lem:nfu-37} \\ \hline
    \ra{30}{37} & (a,a,a), (r,\breve{{r}},a), (r,a,\breve{{r}}) & Yes, Prop.~\ref{prop:30_37-repr} \\ \hline
    \ra{9}{65} & (a,b,b), (b,b,a), (a,b,a) & No, Prop.~\ref{prop:2-cycle-no-fu} \\ \hline
    \ra{12}{65} & (a,b,b), (b,b,a), (a,b,a) & No, Prop.~\ref{prop:2-cycle-no-fu} \\ \hline
    \ra{15}{65} & (a,c,c), (c,c,a), (a,c,a) & No, Prop.~\ref{prop:2-cycle-no-fu} \\ \hline
    \ra{17}{65} & (a,c,c), (c,c,a), (a,c,a) & No, Prop.~\ref{prop:2-cycle-no-fu} \\ \hline
    \ra{24}{65} & (a,a,a), (b,a,b), (a,c,c) & Yes, Prop.~\ref{prop:24_65-repr} \\ \hline
    \ra{30}{65} & (a,a,a), (b,b,c), (a,b,c) & Yes, Prop.~\ref{prop:30_65-31_65-repr} \\ \hline
    \ra{31}{65} & (a,a,a), (b,b,c), (a,b,c) & Yes, Prop.~\ref{prop:30_65-31_65-repr} \\ \hline
    \ra{39}{65} & (a,b,b), (b,b,a), (a,c,c) & No, Prop.~\ref{prop:39_65-repr} \\ \hline
    \ra{51}{65} & (b,a,a), (a,c,c), (a,c,b) & No, Lemma~\ref{lem:5165_unsat} \\ \hline
    \ra{56}{65} & (a,b,b), (b,c,a), (b,b,a) & No, Prop.~\ref{prop:56_65-not-fu} \\ \hline
    \ra{62}{65} & \text{Remark~\ref{rem:62}} & No, Prop.~\ref{prop:62_65-repr} \\ \hline
    \end{tabular}
    \caption{Integral relation algebras with at most four atoms which do not have a normal representation;
    the second column shows a counterexample for the
    condition in Theorem~\ref{thm:two-point-amalgamation}, using the notation of Figure~\ref{fig:2-pt-amal-size-5} (except for the last row for $\ra{62}{65}$, where we need a separate argument). In the third column we give (forward) references concerning the existence of a fully universal representation.}
    \label{tab:amalgamation-failure}
\end{table}

\begin{rem}\label{rem:62} The algebra $\ra{62}{65}$ has the AP$(5)$, but not the AP$(6)$. This proves that the bound in Theorem~\ref{thm:two-point-amalgamation} is optimal in the case of $k = 4$.
The algebra $\ra{62}{65}$
is the only relation algebra with four atoms which admits this behavior. It is the unique relation algebra in which every triple which is not a 1-cycle is allowed. In some way it is the four-atom extension of $\ra{5}{7}$. We note a further similarity between the two algebras: $\ra{5}{7}$ is the only relation algebra with three atoms in which AP$(4)$ holds but AP$(5)$ does not, showing that the bound is optimal in the case of $k = 3$. The optimality of the bound in Theorem~\ref{thm:two-point-amalgamation} for all $k$ can be shown using relation algebras that are analogous to $\ra{62}{65}$.

We present the counterexample for AP$(6)$ in $\ra{62}{65}$ by a matrix $M \in (A_0)^{6 \times 6}$, where we choose $V = \{0, \dots, 5\}$ and denote $m_{ij} = f(i,j)$. The claim is that entries marked by a question mark $?$ cannot be filled by an atom of
$\ra{62}{65}$ such that the resulting atomic network is consistent.

\begin{equation}\label{eq:amalg-62_65}
    M = \begin{pmatrix}
    \id & a & b & c & a & a \\
    a & \id & a & a & b & b \\
    b & a & \id & b & c & c \\
    c & a & b & \id & a & b \\
    a & b & c & a & \id & ? \\
    a & b & c & b & ? & \id
    \end{pmatrix}
\end{equation}
\end{rem}

\subsection{2-Cycle Product}
\label{sect:cycle-prod-rep}

In this section we discuss how representability properties
of relation algebras translate to their cycle products, and apply our results to the case of integral relation algebras with four atoms.

\begin{lem}
\label{lem:2-cycle-properties}
    Let $\bA, \bB$ be finite integral relation algebras with $A_0 \cap B_0 = \{\id\}$ and
    representations $\fC$ and $\fD$. Then
    \begin{enumerate}[label=(\roman*)]
        \item if $\fC$ and $\fD$ are square, then $\fD[\fC]$ is a square representation of $\bA[\bB]$.
        \item if $\fC$ and $\fD$ are fully universal and square, then $\fD[\fC]$ is a fully universal representation of $\bA[\bB]$.
        \item if $\fC$ and $\fD$ are normal, then $\fD[\fC]$ is a normal representation of $\bA[\bB]$.
    \end{enumerate}
\end{lem}
\begin{proof}
    (i): Let $(u_0, v_0), (u_1, v_1) \in D \times C$. If $(u_0, v_0) = (u_1, v_1)$ then $((u_0, v_0), (u_1, v_1)) \in \id^{\fD[\fC]}$. If $u_0 \neq u_1$ then $(u_0, u_1) \in b^\fD$ for some $b \in B_0 \setminus \{\id\}$ since $\fD$ is square. Then $((u_0, v_0), (u_1, v_1)) \in b^{\fD[\fC]}$. Finally, assume that $u_0 = u_1$ and $v_0 \neq v_1$. Then for some $a \in A_0 \setminus \{\id\}$ we have $(v_0, v_1) \in a^\fC$ since $\fC$ is square, and hence $((u_0, v_0), (u_1, v_1)) \in a^{\fD[\fC]}$.

    (ii): Suppose that $\fC$ and $\fD$ are fully universal.
    Let $N = (V, f)$ be a consistent atomic $\bA[\bB]$-network.
    By Remark~\ref{rem:equ_relation} we know that $1_\bA$ is an equivalence relation of $\bA[\bB]$. Then $(1_\bA)^N$ is an equivalence relation on $V$. Let $C_1, \dots, C_k$ be its equivalence classes.
    Then $N_i := (C_i,f|_{C_i})$ is a consistent atomic $\bA$-network. The relation $(1_\bA)^{\fD[\fC]}$ is also an equivalence relation and each class induces a fully universal representation $\fC_i$ of $\bA$. So each $N_i$ can be satisfied in a class of $(1_\bA)^{\fD[\fC]}$.

    By Remark~\ref{rem:equ_relation}
    the contraction $\faktor{N}{1_\bA}$
    is well-defined. This contraction can be regarded as a consistent atomic $\bB$-network. Since $\fD$ is fully universal we find suitable equivalence classes of $(1_\bA)^{\fD[\fC]}$ such that the contraction is satisfiable. So $\fD[\fC]$ is fully universal.

    (iii): It remains to show that $\fD[\fC]$ is homogeneous. Let $f$ be a partial isomorphism of $\fD[\fC].$ First, note that $f$ preserves the relations $a^{\fD[\fC]}$ for $a \in A,$ in particular it preserves the equivalence classes $(C_i)_{i \in I}$ of $(1_\bA)^{\fD[\fC]}.$
    Each of these classes is isomorphic to the normal representation $\fC$ of $\bA,$ which is homogeneous. Thus, for each $i \in I,$ there is at most one $j \in I$ with $f(C_i) \cap C_j \neq \varnothing,$ and the restriction of $f$ to $C_i$ can be extended to an isomorphism $g_i$ between $C_i$ and $C_j.$ Moreover, since $\fD$ is square, there is no $i^\prime \neq i$ with $f(C_{i^\prime}) \cap C_j \neq \varnothing.$ Hence, $f$ induces a partial isomorphism on the contracted structure $\faktor{\fD[\fC]}{1_\bA},$ which is isomorphic to the homogeneous structure $\fD,$ and can thus be extended to an automorphism $h$ of $\faktor{\fD[\fC]}{1_\bA}.$
    If $\mathrm{dom}(f) \cap C_i = \varnothing,$ let $g_i$ be a function which maps $C_i$ isomorphically to $h(C_i).$ It is easy to see that the function $\alpha := \bigcup_{i \in I} g_i$ is an automorphism of $\fD[\fC]$ that extends $f.$
\end{proof}

To obtain a relation algebra from $\ra{1}{37}$--$\ra{37}{37}$ as a 2-cycle product, we have to multiply an algebra from $\ra{1}{3}$--$\ra{3}{3}$ with $\ra{1}{2}$ or with $\ra{2}{2}$. See the first three lines of Table~\ref{fig:2-prod-4atom} for all such products. Since all these algebras have a normal representation (see \eqref{eq:norm_all}), the 2-cycle product has a normal representation as well. A description of these representations can be found in the first twelve entries of Table~\ref{tab:overview_asymmetric}.

To obtain a relation algebra from $\ra{1}{65}$--$\ra{65}{65}$ as a 2-cycle product, we can restrict ourselves to the case where we multiply algebras from $\ra{1}{7}$--$\ra{7}{7}$ with $\ra{1}{2}$ or with $\ra{2}{2}$, see the last seven lines of Table~\ref{fig:2-prod-4atom}. Except for $\ra{5}{7}$, all factors have a normal representation, so the corresponding products have a normal representation. One can find descriptions of the representation in the first 20 rows of Table~\ref{tab:overview_symmetric}. In the following proposition we consider the four relation algebras with $\ra{5}{7}$ as a factor.

\begin{prop}\label{prop:2-cycle-no-fu}
    Let $\bB \in \{\ra{9}{65}, \ra{12}{65}, \ra{15}{65}, \ra{17}{65}\}$. Then $\bB$ does not have a fully universal representation.
\end{prop}

\begin{proof}
    Since $\bB$ is simple, every square component of a representation of $\bB$ is itself a representation of $\bB$. Every network solution is contained in a single square component. It therefore suffices to exhibit a consistent atomic $\bB$-network that is unsatisfiable in every square representation of $\bB$.
    
    Let $\fC$ be a square representation of $\bB$.
    First consider the case that $\bB$ is of the form $\bB = \ra{5}{7}[\bA]$ for some relation algebra $\bA$.  Then $1_{\ra{5}{7}}$ is an equivalence relation of $\bB$ and $(1_{\ra{5}{7}})^{\fC}$ is an equivalence relation on $C$. Consider any equivalence class $D$ of $(1_{\ra{5}{7}})^\fC$. By an analogous argument as in \cite[Theorem 4.5]{AndrekaMaddux} one can show that the induced substructure of $\fC$ by $D$ is isomorphic to the unique representation of $\ra{5}{7}$ (see Example~\ref{ex:5_7Part1}). After renaming the edges accordingly, it follows that the network in Figure~\ref{fig:evil_square} is not satisfiable. This settles the cases of $\ra{9}{65} = \ra{5}{7}[\ra{1}{2}]$ and $\ra{12}{65} = \ra{5}{7}[\ra{2}{2}]$.

    Now consider the case where $\bB = \bA[\ra{5}{7}]$. Then the contraction $\faktor{\fC}{(1_\bA)^\fC}$ is well-defined. Again, by the same argument as in \cite[Theorem 4.5]{AndrekaMaddux}, it follows that this contraction is isomorphic to the unique representation of $\ra{5}{7}$. So with adequate naming, the atomic network shown in  Figure~\ref{fig:evil_square} is again not satisfiable. This settles the cases of $\ra{15}{65} = \ra{1}{2}[\ra{5}{7}]$ and $\ra{17}{65} = \ra{2}{2}[\ra{5}{7}]$. Note that in these cases the atoms of $\ra{5}{7}$ are named $a$ and $c$ in Table~\ref{tab:overview_symmetric}.
\end{proof}

\begin{table}
\centering
\begin{tabular}{|L|L|L|L|L|}\hline
     & \multicolumn{2}{C|}{\ra{1}{2}} & \multicolumn{2}{C|}{\ra{2}{2}} \\ \hline
    \bA & \ra{1}{2}[\bA] & \bA[\ra{1}{2}] & \ra{2}{2}[\bA] & \bA[\ra{2}{2}] \\ \hline \hline
    \ra{1}{3} & \ra{1}{37} & \ra{7}{37} & \ra{2}{37} & \ra{8}{37} \\ \hline
    \ra{2}{3} & \ra{3}{37} & \ra{9}{37} & \ra{4}{37} & \ra{10}{37} \\ \hline
    \ra{3}{3} & \ra{5}{37} & \ra{11}{37} & \ra{6}{37} & \ra{12}{37} \\ \hline \hline
    \ra{1}{7} & \ra{1}{65} & \ra{1}{65} &  \ra{2}{65} & \ra{5}{65} \\ \hline
    \ra{2}{7} & \ra{3}{65} & \ra{2}{65} &  \ra{4}{65} & \ra{6}{65} \\ \hline
    \ra{3}{7} & \ra{5}{65} & \ra{3}{65} &  \ra{6}{65} & \ra{7}{65} \\ \hline
    \ra{4}{7} & \ra{7}{65} & \ra{4}{65} &  \ra{8}{65} & \ra{8}{65} \\ \hline
    \ra{5}{7} & \ra{15}{65} & \ra{9}{65} &  \ra{17}{65} & \ra{12}{65} \\ \hline
    \ra{6}{7} & \ra{16}{65} & \ra{10}{65} &  \ra{18}{65} & \ra{13}{65} \\ \hline
    \ra{7}{7} & \ra{19}{65} & \ra{11}{65} &  \ra{20}{65} & \ra{14}{65} \\ \hline
\end{tabular}
\caption{Not-so-small relation algebras formed by 2-cycle products of small relation algebras.}
\label{fig:2-prod-4atom}
\end{table}

\subsection{Fully Universal Representations}
\label{sect:FU}
We are now left with relation algebras that are representable (as we will see), but that do not have a normal representation.
In this section we investigate which of them have at least a fully universal representation, using Corollary~\ref{cor:fu}.

We start with the relation algebra $\ra{13}{37}$, which is also known as the \emph{left linear point algebra}~\cite{Duentsch,BodirskyKutzAI}.

\begin{prop}
\label{prop:13_37-repr}
    $\ra{13}{37}$ has a fully universal square representation.
\end{prop}

\begin{proof}
    We verify the conditions of Corollary~\ref{cor:fu}.
    For any two disjoint consistent atomic $\ra{13}{37}$-networks, we may connect all vertices by $a$ and obtain another consistent atomic $\ra{13}{37}$-network. So $\ra{13}{37}$ has the JEP.

    To verify $\AP(3,2,n)$,
    let $(V,f)$ be a reduced consistent atomic $\ra{13}{37}$-network, and let $x,y \in V$.
    We have to verify that for all atoms $p,q$ of $\ra{13}{37}$ different from $\id$
    such that $f(x,y) \leq p \circ q$
    we can find a consistent atomic network
    $(V \cup \{z\},g)$, for some new node $z \notin V$, such that $g|_{V^2} = f$,
    $g(x,z) = p$, and $g(z,y) = q$.
    Without loss of generality we can assume that if $a\in \{p,q\}$ then $p=a$. Let $W := V \setminus \{x,y\}$.
    \begin{itemize}
        \item $p=q=a$. Define $g(z,v) = a$ for all $v \in W$.
        \item $(p,q) \in \{(a,\breve{r}), (r, r)\}$. Define $g(z,v) = f(y,v)$ for all $v \in W$.
        \item $(p, q) = (r, \breve{r})$. Set $g(z, v) = \breve{r}$ for all $v \in W$.
        \item $(p, q) = (a, r)$. For all $v \in W$, define $g(z, v) = r$ if $f(y, v) = r$ and $g(z, v) = a$ otherwise.
        \item $(p, q) = (\breve{r}, \breve{r})$. For all $v \in W$, we set $g(z, v) = f(x, v)$.
        \item $(p,q) = (\breve{r},r)$. For all $v \in W$, define $g(z, v) = r$ if $r\in \{f(x, v),f(y,v)\}$ and $g(z, v) = a$ otherwise.
    \end{itemize}
    In each case it is straightforward to check that
    $g(u,v) \leq g(u,z) \circ g(z,v)$ for all $u,v \in V$, and hence that
    the resulting atomic network is consistent.
\end{proof}

The reduced consistent atomic networks of $\ra{30}{37}$ can be viewed as oriented graphs that do not embed the oriented graph $(\{0,1,2\}; \{(0,1),(0, 2)\})$.

\begin{prop}
\label{prop:30_37-repr}
    $\ra{30}{37}$ has a fully universal square representation.
\end{prop}

\begin{proof}
    For any two disjoint consistent atomic $\ra{30}{37}$-networks, we may connect all vertices by $a$ and obtain another consistent atomic $\ra{30}{37}$-network. So $\ra{30}{37}$ has the JEP.

    Let $(V,f)$ be a reduced consistent atomic $\ra{30}{37}$-network, and let $x,y \in V$.
    We have to verify that for all atoms $p,q$ of $\ra{30}{37}$ different from $\id$
    such that $f(x,y) \leq p \circ q$
    we can find a consistent atomic network
    $(V \cup \{z\},g)$, for some new node $z \notin V$, such that $g|_{V^2} = f$,
    $g(x,z) = p$, and $g(z,y) = q$. Without loss of generality we can assume that if $a\in \{p,q\}$ then $p=a$. Moreover, we can assume that there is no $v \in V$ such that $f(x, v) = p$ and $f(v, y) = q$.

    We prove the statement by considering the following cases. Define $W := V \setminus \{x,y\}$.
    \begin{itemize}
        \item If $p=q=a$ then define $g(z,u):=a$ for all $u \in W$
        \item If $p,q \in \{r, \breve{r}\}$, then define $g(z,u):=\breve{r}$ for all $u \in W$.
        \item If $(p,q)  = (a,r)$, then define $g(z,u):= a$ for all $u \in W$.
        \item If $(p,q)  = (a,\breve{r})$, we let $X$ be the set of all vertices $v \in W$ such that $f(y,v) = r$. Observe that $f(u,v) \neq a$ for all $u,v \in X$. Define $g(z,v) = r$ for all $v \in X$ and $g(z,v) = a$ for all $v \in W \setminus X$.
    \end{itemize}
    In each of the cases it is easy to check that the resulting atomic network is consistent.
\end{proof}

The consistent atomic networks of the relation algebra $\ra{24}{65}$ can be viewed as edge-3-colourings of cliques that avoid `rainbows', i.e., triangles that carry all three colors.

\begin{prop}\label{prop:24_65-repr}
    $\ra{24}{65}$ has a fully universal square representation.
\end{prop}

\begin{proof}
    It is clear that the class of all consistent atomic $\ra{24}{65}$-networks has the JEP, since we may add the same atom between all vertices from one
    consistent atomic $\ra{24}{65}$-network and another one, and in this way clearly avoid rainbows.

    To verify that the class of reduced consistent atomic $\ra{24}{65}$-networks has AP$(3,2,n)$ for all $n \in {\mathbb N}$,
    let $(V,f)$ be a reduced consistent atomic $\ra{24}{65}$-network, and let $x,y \in V$.
    We have to verify that for all atoms $p,q$ of $\ra{24}{65}$ different from $\id$
    such that $f(x,y) \leq p \circ q$
    we can find a consistent atomic network
    $(V \cup \{z\},g)$, for some new node $z \notin V$, such that $g|_{V^2} = f$,
    $g(x,z) = p$, and $g(z,y) = q$.
    We distinguish two cases:
    \begin{itemize}
        \item $p=q$. Then
        define $g(z,v) := p$ for all $v \in V$.
        Then clearly no triangle containing $z$ is rainbow.
        \item $p \neq q$. Then $f(x,y)\in\{p,q\}$, as otherwise $\{z,x,y\}$ forms a rainbow.
        Say that $f(x,y) = q$. In this case for every vertex $v \in V \setminus \{x,y\}$, define $g(z,v) := f(x,v)$.
    \end{itemize}
    We claim that
    $(V \cup \{z\},g)$ is consistent. Suppose for contradiction that we found a rainbow triangle containing $z$. It cannot contain $x$ (as in every triangle $\{v,x,z\}$ we have that $g(z,v) = f(x,v)$, so it is $\{v,w,z\}$ for some $v,w \in V \setminus \{x\}$. But then this triangle is isomorphic to $\{x,v,w\}$, which cannot be rainbow.
\end{proof}

In the relation algebras $\ra{30}{65}$ and $\ra{31}{65}$ the relation $c \cup \id$ is an equivalence relation. In every fully universal representation
every class of $(c \cup \id)^\fB$ has infinitely many elements.

\begin{prop}
\label{prop:30_65-31_65-repr}
    If $\bA \in \{\ra{30}{65}, \ra{31}{65}\}$, then $\bA$ has a fully universal square representation.
\end{prop}

\begin{proof}
    For any two disjoint consistent atomic $\bA$-networks, we may connect all vertices by $a$ and obtain another consistent atomic $\bA$-network. So $\bA$ has the JEP.

    Let $(V,f)$ be a reduced consistent atomic $\bA$-network, and let $x,y \in V$.
    We have to verify that for all atoms $p,q \in A_0$
    such that $f(x,y) \leq p \circ q$
    we can find a consistent atomic network
    $(V \cup \{z\},g)$, for some new node $z \notin V$, such that $g|_{V^2} = f$,
    $g(x,z) = p$, and $g(z,y) = q$. Without loss of generality we can assume that if $c\in \{p,q\}$ then $p=c$. Moreover, we can assume that there is no $v \in V$ such that $f(x, v) = p$ and $f(v, y) = q$.

    We prove the statement by considering the following cases. Define $W := V \setminus \{x,y\}$.
    \begin{itemize}
        \item If $p\neq c$ (and so $q\neq c$), define $g(z,v) = a$ for all $v \in W$.
        \item If $p = c$, let $X$ be the set of all vertices $v \in W$ such that $f(v,x) = c$ and observe that $f(v,u) = c$ for all distinct $u,v \in X$, and if $v \in X$ and $f(u,v)=c$ for some $u \in W$, then $u \in X$. Define $g(z,v) = c$ for all $v \in X$ and $g(z,v) = a$ for all $v \in W \setminus X$. Note that if $q=c$ then $f(x,y) = c$, and so for all $v \in W$ we have $f(v,x) = c$ if and only if $f(v,y) = c$.
    \end{itemize}

    In each case and for each of the algebras $\ra{30}{65}$ and $\ra{31}{65}$ it is straightforward to check that
    $g(u,v) \leq g(u,z) \circ g(z,v)$ for all $u,v \in V$ and hence
    the resulting atomic network is consistent.
\end{proof}

\subsection{Bounded Square Representations}
\label{sect:bounded-square}
In this section we show the relation algebras $\ra{39}{65}$
and $\ra{62}{65}$ have representations, and that the size of all square representations of these algebras is bounded by some finite number. Via Lemma~\ref{lem:bounded}, this implies hardness
of the respective NSPs (see Proposition~\ref{prop:bounded-hard}). We first need to introduce Ramsey numbers.

\begin{defi}
    Let $c \in \N^+$ and $n_1, \dots, n_c \in \N^+$. The \emph{Ramsey number} $R = R(n_1, \dots, n_c)$ denotes the smallest natural number such that if the edges of a complete graph of order $R$ are colored with $c$ different colors, then for some $i \in \{1, \hdots, c\}$, the graph must contain a complete subgraph of order $n_i$ whose edges are all in color $i$. Ramsey's theorem assures that this number always exists.
\end{defi}

\begin{lem}
\label{lem:finite_classes}
    Let $\bA$ be a finite representable relation algebra with an equivalence relation $e \in A$. Furthermore, let $S := \{a \in A_0 \mid a\not\leq e \circ e = e\}$. If $S \neq \varnothing$ and for every $a \in S$, the triple $(a,a,a)$ is forbidden, then in every square representation $\fB$ of $\bA$, the number of equivalence classes of $e^\fB$ is bounded by $R(\underbrace{3, \dots, 3}_{|S|~\text{times}}) - 1 < \infty$.
\end{lem}
\begin{proof}
    Let $\fB$ be a square representation of $\bA$ and $e^\fB$ an equivalence relation. It holds that
    \[n := R(\underbrace{3, \dots, 3}_{|S|~\text{times}})< \infty ,\]
    since $\bA$ is finite. Assume that $e^\fB$ has at least $n$ classes and choose representatives $x_1, \dots, x_n$ of $n$ distinct classes. For $1 \leq i < j \leq n$ and every $x_i, x_j$ there exists an atom $a \in S$ such that $(x_i, x_j) \in a^\fB$. If we again interpret the relations as colors then from the definition of the Ramsey number it follows that there exists an atom $a \in S \subseteq A_0$ such that $x_i,x_j,x_k \in B$ with $(x_i, x_j), (x_j,x_k), (x_i,x_k) \in a^\fB$ exist. However, this contradicts the fact that $(a,a,a)$ is a forbidden triple. So $e^\fB$ cannot have more than $n-1$ classes.
\end{proof}

In~\cite{GG1955} it was proved that $R(3,3,3) = 17$.

\begin{prop}
\label{prop:39_65-repr}
    $\bA = \ra{39}{65}$ has a square representation $\fB$ with 7 elements, but it has no fully universal representation. Moreover, all square representations have at most 16 elements.
\end{prop}
\begin{proof}
    It is straightforward to verify that the following structure $\fB$ is a representation of $\bA\colon$
     \begin{align*}
        B &= \Z_7, \\
        a^\fB &= \{(i, i \pm_7 1) \mid i \in B\}, \\
        c^\fB &= \{(i, i \pm_7 2) \mid i \in B\}, \\
        b^\fB &= B^2 \setminus (a^\fB \cup c^\fB \cup \id^\fB) = \{(i, i \pm_7 3) \mid i \in B\}.
    \end{align*}

    However, $\ra{39}{65}$ does not have AP$(3, 2, 4)$; see Figure \ref{fig:ap-failure-39-65}. Hence, it does not have a fully universal representation by Theorem~\ref{thm:fu}.
    \begin{figure}
    \centering
    \begin{tikzcd}
    && \bullet && \bullet \\
	\bullet \\
	&& \bullet && \bullet
	\arrow["a", no head, from=1-3, to=1-5]
	\arrow["a"', no head, from=1-3, to=3-3]
	\arrow["c"{pos=0.7}, no head, from=1-3, to=3-5]
	\arrow["a", no head, from=1-5, to=3-5]
	\arrow["b", no head, from=2-1, to=1-3]
	\arrow[dashed, no head, from=2-1, to=1-5]
	\arrow["b"', no head, from=2-1, to=3-3]
	\arrow[dashed, no head, from=2-1, to=3-5]
	\arrow["c"{pos=0.7}, no head, from=3-3, to=1-5]
	\arrow["a"', no head, from=3-3, to=3-5]
\end{tikzcd}
    \caption{Failure of AP$(3, 2, 4)$ for $\ra{39}{65}$.}
    \label{fig:ap-failure-39-65}
\end{figure}

\begin{figure}
\centering
\begin{tikzpicture}[
    scale=1,
    vertex/.style={circle, fill=black, inner sep=1.4pt},
    edge/.style={line width=0.5pt, shorten <=5pt, shorten >=5pt},
    missing/.style={dashed, line width=0.5pt, shorten <=5pt, shorten >=5pt}
]
\coordinate (0) at (-1.4, 0.8);
\coordinate (1) at (0, 1.6);
\coordinate (2) at (1.4, 0.8);
\coordinate (3) at (1.4,-0.8);
\coordinate (4) at (0,-1.6);
\coordinate (5) at (-1.4,-0.8);

\coordinate (z) at (-3.0,0);

\draw[edge, black] (0) -- (1);
\draw[edge, blue]  (1) -- (2);
\draw[edge, red] (2) -- (3);
\draw[edge, blue]  (3) -- (4);
\draw[edge, black] (4) -- (5);
\draw[edge, blue]  (5) -- (0);

\draw[edge, black] (0) -- (2);
\draw[edge, black] (0) -- (3);
\draw[edge, black] (0) -- (4);

\draw[edge, blue]  (1) -- (3);
\draw[edge, red] (1) -- (4);
\draw[edge, black] (1) -- (5);

\draw[edge, blue]  (2) -- (4);
\draw[edge, black] (2) -- (5);

\draw[edge, black] (3) -- (5);

\draw[edge, black] (z) -- (0);
\draw[edge, blue]  (z) -- (5);

\draw[missing, out=50, in=190] (z) to (1);
\draw[missing] (z) -- (2);
\draw[missing] (z) -- (3);
\draw[missing, out=-50, in=170] (z) to (4);

\node[vertex] at (0) {};
\node[vertex] at (1) {};
\node[vertex] at (2) {};
\node[vertex] at (3) {};
\node[vertex] at (4) {};
\node[vertex] at (5) {};
\node[vertex] at (z) {};
\end{tikzpicture}
\caption{Failure of AP$(3,2,6)$ for $\ra{62}{65}$.
Edges colored blue, red, and black\\ are labeled $a$, $b$, and $c$, respectively.}
\label{fig:ap-failure-62-65}
\end{figure}

    By Lemma~\ref{lem:finite_classes} for $e = \id$ it follows that every square representation of $\bA$ has size at most $16 = R(3,3,3)-1$.
\end{proof}

\begin{prop}\label{prop:62_65-repr}
    $\ra{62}{65}$ has a square representation $\fB$ with 13 elements,
    but no fully universal representation. All square representations have at most 16 elements.
\end{prop}

\begin{proof}
    We learned the following representation $\fB$ of $\ra{62}{65}$ from P.\ Jipsen:
    \begin{align*}
        B &= \Z_{13}, \\
        a^\fB &= \{(x,x+i) \in B^2 \mid i \in \{1,5,8,12\} \}, \\
        b^\fB &= \{(x,x+i) \in B^2 \mid i \in \{2,3,10,11\} \}, \\
        c^\fB &= B^2 \setminus (a^\fB \cup b^\fB \cup \id^\fB) = \{(x,x+i) \in B^2 \mid i \in \{4,6,7,9\} \}.
    \end{align*}
        However, $\ra{62}{65}$ does not have AP$(3,2,6)$; see Figure~\ref{fig:ap-failure-62-65}. By Theorem~\ref{thm:fu}, there is no fully universal representation.
      By Lemma~\ref{lem:finite_classes} for $e = \id$ it follows that every square representation of $\bA$ has size at most $16 = R(3,3,3)-1$.
\end{proof}

\subsection{Representable Algebras without Fully Universal Representation}
\label{sec:nfu}
In this section we treat representable relation algebras which have infinite representations, but no fully universal representations.

\subsubsection{The Relation Algebra \texorpdfstring{$\ra{17}{37}$}{17\textunderscore 37}}
\label{sect:17_37-rep}
The reduced consistent atomic $\ra{17}{37}$-networks can be viewed as oriented graphs that do not embed the oriented graph $P_2 = (\{0,1,2\}; \{(0,1),(1,2)\})$. These structures have been studied before and are called \emph{quasi-transitive oriented graphs}~\cite{BJJ95}.
First, we will identify one consistent atomic network on four vertices which is not satisfiable, and then we will construct a representation in which all consistent atomic networks avoiding the unsatisfiable one are satisfiable.

Let $(A,f)$ be an atomic $\ra{17}{37}$-network. An \emph{induced $a$-path} in $(A,f)$ is a sequence of distinct vertices $a_0, \dots, a_{n-1} \in A$
such that $f(a_i, a_j) = a$ if and only if $|i-j|=1$. An induced $a$-path $a_0,\dots,a_{n-1}$ is \emph{independent} from some $b\in A$ if for every $i\in \{0,\dots,n-1\}$ we have that $f(a_i,b) \notin \{\id, a\}$.

\begin{lem}
    Let $(A,f)$ be a consistent atomic $\ra{17}{37}$-network, and let $a_0, \ldots, a_{n-1} \in A$ be an induced $a$-path in $(A,f)$. Then for every $x\in A$ such that $a_0, \ldots, a_{n-1}$ is independent from $x$ we have that $f(x,a_i) = f(x,a_j)$ for every $i,j\in \{0,\ldots,n-1\}$.
\end{lem}
\begin{proof}
    For every $i<n-1$ we have that $f(x,a_i) = f(x,a_{i+1})$, as otherwise $x,a_i,a_{i+1}$ would be a forbidden triple. The lemma now follows by transitivity of equality.
\end{proof}

\begin{lem}
    Let $(A,f)$ be a consistent atomic $\ra{17}{37}$-network and let $a_0, \ldots, a_{n-1} \in A$ be an induced $a$-path in $(A,f)$. Then for all $i,i',j,j' \in \{0,\dots,n-1\}$ such that $i < j-1$ and $i' < j'-1$, it holds that $f(a_i,a_j) = f(a_{i'},a_{j'})$.
\end{lem}
\begin{proof}
    Let $0 \leq \ell < n$ and fix an arbitrary $k < \ell-1$ and observe the following properties (all follow directly from the fact that $(A,f)$ does not contain the single forbidden triple):
    \begin{enumerate}
        \item If $\ell+1 < n$ then $f(a_k, a_\ell) = f(a_k, a_{\ell+1})$,
        \item if $k < \ell-2$ then $f(a_k,a_\ell) = f(a_k,a_{\ell-1})$,
        \item if $k < \ell-2$ then $f(a_k,a_\ell) = f(a_{k+1},a_\ell)$, and
        \item if $k > 0$ then $f(a_k, a_\ell) = f(a_{k-1},a_\ell)$.
    \end{enumerate}
    Now we can prove the statement of the lemma by repeated applications of the first or second, and the third or fourth points (depending on the order of $j$ and $j'$, and $i$ and $i'$ respectively).
\end{proof}

\begin{cor}\label{cor:1737induced}
    Two induced $a$-paths are isomorphic if and only if they have the same number of vertices.
\end{cor}

This justifies that we can call the $\ra{17}{37}$-network $(P,f_P)$ with $P = \{0,1,2,3\}$, $f_P(0,1)=f_P(1,2)=f_P(2,3) = a$, and $f_P(0,2)=f_P(0,3)=f_P(1,3) = r$ the \emph{induced $a$-path on 4 vertices}. Note that when viewed as an oriented graph, $(P,f_P)$ is acyclic, and is in fact the irreflexive variant of the poset sometimes called $N$. The $N$-free posets are known as \emph{series-parallel}~\cite{SP}. In Section~\ref{sect:17-alg} we will employ a decomposition theorem for these posets in order to obtain an algorithm for $\NSP(\ra{17}{37})$.

\begin{lem}\label{lem:nfu-37}
    The induced $a$-path on 4 vertices is a consistent atomic $\ra{17}{37}$-network which is unsatisfiable.
\end{lem}
\begin{proof}
    Clearly, $(P,f_P)$ is consistent and atomic. Assume for a contradiction that there is some representation of $\ra{17}{37}$ satisfying $(P, f_P)$; then there is also a square representation $\fB$ of $\ra{17}{37}$ satisfying $(P, f_P)$ by Corollary~\ref{cor:simple-implies-square-rep}. That is, there is a map $s\colon P\to B$ such that for every $x,y\in P$ it holds that $(s(x),s(y)) \in f_P(x,y)^{\fB}$. Since $\fB$ is a representation of $\ra{17}{37}$ and $r \leq \Breve{r} \circ \Breve{r}$, there is some $x\in B$ with $(s(0),x) \in \Breve{r}^\fB$ and $(x,s(2)) \in \Breve{r}^\fB$. If $(s(1),x) \in r^\fB$ or $(x,s(1)) \in r^\fB$ then one of $(s(1),x,s(0))$ and $(s(2),x,s(1))$ is a forbidden triple; hence, $(s(1),x) \in a^{\fB}$. But now every selection for the edge $(s(3),x)$ leads immediately to a contradiction.
\end{proof}
\begin{cor}
    Every atomic $\ra{17}{37}$-network containing an induced $a$-path on at least $4$ vertices is unsatisfiable.
\end{cor}

Let $\mathcal C$ be the class of all finite consistent atomic $\ra{17}{37}$-networks containing no induced $a$-path on $4$ vertices.
\begin{lem}\label{lem:1737jep}
    $\mathcal C$ has JEP.
\end{lem}
\begin{proof}
    Pick arbitrary $(V_1,f_1),(V_2,f_2)\in \mathcal C$ with $V_1\cap V_2 = \varnothing$. Put $V = V_1\cup V_2$ and define a $\ra{17}{37}$-network $(V,g)$ with $g|_{V_1^2} = f_1$, $g|_{V_2^2} = f_2$ and $g(x,y) = a$
    if $x\in V_1$ and $y\in V_2$ or $x\in V_2$ and $y\in V_1$.
    It is straightforward to verify that $(V,g)\in \mathcal C$.
\end{proof}

We will spend the remainder of this section proving the following lemma.

\begin{lem}\label{lem:1737ap}
    $\mathcal C$ has AP$(3,2,n)$ for every $n$.
\end{lem}
Before we proceed with a proof of Lemma~\ref{lem:1737ap}, observe that assuming it holds, we
obtain a universal square representation of $\ra{17}{37}$. In fact, using Theorem~\ref{thm:fu} we
immediately obtain the following.

\begin{thm}
    There is a square representation $\fB$ of $\ra{17}{37}$ such that the following are equivalent for every $\ra{17}{37}$-network $(A,f)$:
    \begin{enumerate}
        \item $(A,f)$ is satisfiable,
        \item $(A,f)$ is satisfiable in $\fB$,
        \item there is a consistent atomic $\ra{17}{37}$-network $(A,g)$ containing no induced $a$-path on $4$ vertices such that $g(x,y) \leq f(x,y)$ for every $x,y\in A$.
    \end{enumerate}
\end{thm}

\begin{proof}[Proof of Lemma~\ref{lem:1737ap}]
    Pick some reduced $(A,f)\in \mathcal C$ with $z\notin A$, pick $x,y\in A$ and some non-identity atoms  $p,q$ of $\ra{17}{37}$ such that $f(x,y)\leq p\circ \breve{q}$. Put $B = A\cup \{z\}$. We will show that there is $(B,g)\in \mathcal C$ such that $g(x,z) = p$, $g(y,z) = q$,
    and $g$ restricted to $A^2$ is equal to $f$. We distinguish four cases:
    \begin{enumerate}
        \item If $p=q$ then we can define $g(v,z) = p$ for every $v\in A$ which clearly works. So $p\neq q$.
        \item\label{lem:1737ap:3} If $f(x,y) = p$ then we can put $g(v,z) = f(v,y)$ for every $v\in A\setminus \{y\}$ (essentially making $z$ a copy of $y$), and $g(y,z) = q$.
        Note that $g(x,z) = f(x,y) = p$, so we only need to verify that $(B,g) \in \mathcal C$.
        Assume for a contradiction that this is not the case, that is, $(B,g)$ contains either the forbidden triple, or the induced $a$-path on 4 vertices. Since $(A,f) \in \mathcal C$, it follows that the hypothetical forbidden picture contains $z$. If it does not contain $y$ then we obtain an isomorphic subnetwork of $(A,f)$ by exchanging $z$ for $y$, hence it contains both $y$ and $z$. However, it is easy to verify that neither the forbidden triple, nor the induced $a$-path on 4 vertices contain a pair of vertices connected in the same way to every other vertex, which is a contradiction.
        \item If $f(y,x) = q$ then we can put $g(v,z) = f(v,x)$ for every $v\in A\setminus \{x\}$ (essentially making $z$ a copy of $x$), and $g(x,z) = p$.
        Note that $g(y,z) = f(y,x) = q$, and
        we can use the same argument as in the previous case to prove that $(B,g)\in \mathcal C$.
        \item If none of the above holds, observe that $f(x,y) \in \{r,\Breve{r}\}$. By symmetry, we can assume that $f(x,y) = r$. Since $p\neq q$, we get that at least one of them is different from $a$, and so either $p = g(x,z) = \Breve{r}$, or $q = g(y,z) = r$. In fact, in both cases it follows that $p =  \Breve{r}$ and $q = r$. In other words, $(x,y,z)$ is a cyclically oriented triangle.

        Note that as $(A,f)\in \mathcal C$, all induced $a$-paths in $(A,f)$ have at most 3 vertices. Observe that there is no $v\in A$ for which there are both an induced $a$-path $v = a_0,\ldots,a_{n-1} = x$ independent
        from $y$ and an induced $a$-path $v = b_0,\ldots,b_{m-1} = y$ independent from $x$ for some $n,m \in \{2,3\}$. This is clear if $f(x,v) = a$ or $f(y,v) = a$. So $m = n = 3$. Note that $f(y,a_1)\neq a$ and $f(x,b_1)\neq a$. If $f(a_1,b_1) = a$ then $x,a_1,b_1,y$ is an induced $a$-path on 4 vertices, a contradiction. So $f(a_1,b_1)\neq a$, but then $x,a_1,v,b_1,y$ is an induced $a$-path, again a contradiction.
        We put
        $$g(v,z) = \begin{cases}
            f(y,v) &\text{if there is an induced $a$-path $v = a_0,\ldots,a_{n-1} = x$ independent from $y$}\\
            f(x,v) &\text{if there is an induced $a$-path $v = a_0,\ldots,a_{n-1} = y$ independent from $x$}\\
            f(v,x) &\text{otherwise}.
        \end{cases}
        $$
        By the paragraph above, this is well-defined. It is easy to see that $g(x,z) = p$ and $g(y,z) = q$, so we only need to prove that $(B,g)\in \mathcal C$. Observe that for $v\in A$ we have that $g(v,z) = a$ if and only if $f(x,v) = f(y,v) = a$. We will use this below. Suppose for a contradiction that $(B,g)\notin \mathcal C$. We need to distinguish four cases.

        \begin{enumerate}
            \item First assume that there are $u,v\in A$ such that $g(u,v) = r$, $g(v,z) = r$, and $g(u,z) = a$. We know that $f(x,u) = f(y,u) = a$. It follows that $f(x,v) \neq a$ and $f(y,v) \neq a$ (both cannot be equal to $a$ at the same time as $g(v,z)\neq a$, and if only one is equal to $a$ then we find an induced $a$-path on vertices $x,y,u,v$).
            If there is an induced $a$-path $v = a_0,\ldots,a_{n-1} = x$ independent from $y$,
            then we know that $n = 3$.
            Clearly, $f(a_1,y)\neq a$. If $f(a_1,u)\neq a$ then we have the induced $a$-path $v,a_1,x,u$, and if $f(a_1,u) = a$ then we find the induced $a$-path $v,a_1,u,y$.
            So there is no induced $a$-path $v = a_0,\ldots,a_{n-1} = x$ independent from $y$, and similarly we can prove that there is no induced $a$-path $v = a_0,\ldots,a_{n-1} = y$ independent from $x$. Consequently, $r = g(v,z) = f(v,x)$, but as we also have $r = g(u,v) = f(u,v)$ and $a = f(x,u)$, it follows that $(A,f)$ induces the forbidden triple on $u,v,x$, a contradiction.
            \item If there are $u,v\in A$ such that $g(u,v) = \breve{r}$, $g(v,z) = \breve{r}$, and $g(u,z) = a$ then we can proceed analogously as in the previous case. One can also use the fact that changing the direction of all edges, that is, exchanging $r$ and $\breve{r}$, preserves membership in $\mathcal C$ and reduces this case to the previous one.
            \item Next assume that there are $u,v\in A$ such that $g(u,v) = a$, $g(v,z) = r$, and $g(z,u) = r$. As $g(v,z)$ and $g(u,z)$ are not equal to $a$, we know that at most one of $f(x,u)$ and $f(y,u)$ is equal to $a$ and at most one of $f(x,v)$ and $f(y,v)$ is equal to $a$. If $f(x,u) = f(y,v) = a$ then we find an induced $a$-path on vertices $x,u,v,y$, similarly if $f(x,v) = f(y,u) = a$. So these cases do not happen. It follows that $g(u,z) = f(y,u)$ if and only if $g(v,z) = f(y,v)$ (and in this case we get that $f(y,u) =\breve{r}$ and $f(y,v) = r$, so the triple $u,v,y$ is forbidden), that $g(u,z) = f(x,u)$ if and only if $g(v,z) = f(x,v)$ (and in this case the triple $u,v,x$ is forbidden), and that $g(u,z) = f(u,x)$ if and only if $g(v,z) = f(v,x)$ (and in this case the triple $u,v,x$ is again forbidden). In all cases, we arrived to a contradiction.

            \item The above three points imply that all triples in $(B,g)$ are allowed, and so $(B,g)$ contains an induced $a$-path on 4 vertices. Clearly, this path contains $z$. There are two possibilities, either $z$ is an endpoint of this path, or it is not.

            If $z$ is an endpoint, call the other vertices of the path $u,v,w$ such that $f(u,v) = f(v,w) = g(w,z) = a$. As $g(w,z) = a$, we know that $f(x,w) = f(y,w) = a$. Now, since $g(u,z),g(v,z) \neq a$, it follows that at most one of $f(x,u)$ and $f(y,u)$ is equal to $a$ and similarly for $f(x,v)$ and $f(y,v)$. On the other hand, we know that $a\in \{f(x,u),f(x,v)\}$ (as otherwise $u,v,w,x$ would be an induced $a$-path on 4 vertices in $(A,f)$, and similarly $a\in \{f(y,u),f(y,v)\}$. It follows that either $f(x,u) = f(y,v) = a$, or $f(x,v) = f(y,u) = a$. In the first case $x,u,v,y$ is an induced $a$-path on 4 vertices in $(A,f)$, and in the second case $x,v,u,y$.

            So $z$ is not an endpoint. Call the other vertices of the path $u,v,w$ such that $f(u,v) = g(v,z) = g(z,w) = a$. Then we know that $f(v,x) = f(v,y) = f(w,x) = f(w,y) = a$, and at least one of $f(u,x)$ and $f(u,y)$ is not equal to $a$. If $f(u,x) \neq a$ then $u,v,x,w$ is an induced $a$-path on 4 vertices in $(A,f)$, and if $f(u,y) \neq a$ then $u,v,y,w$ is an induced $a$-path on 4 vertices in $(A,f)$, a contradiction. Consequently, $(B,g) \in \mathcal C$. \qedhere
        \end{enumerate}
    \end{enumerate}
\end{proof}

\subsubsection{The Relation Algebra \texorpdfstring{$\ra{51}{65}$}{51\textunderscore 65}}
\label{sect:51_65}

The relation algebra $\ra{51}{65}$ is given
by the facts that $b \cup \id$ is an equivalence relation and $(a, a, a)$ and $(c, c, c)$ are forbidden triples.
We will find a universal, but not fully universal representation $\fB$ of $\ra{51}{65}$ such that a finite $\ra{51}{65}$-network is satisfiable in some representation of $\ra{51}{65}$ if and only if it is satisfiable in $\fB$. This representation will not be fully universal, and we will characterize the consistent atomic $\ra{51}{65}$-networks which are satisfiable by means of four forbidden subnetworks on four vertices. These will be the following:

Put $A = \{0,1,2,3\}$ and define four atomic $\ra{51}{65}$-networks $(A,f_1)$, $(A,f_2)$, $(A,f_3)$, and $(A,f_4)$ as follows:
\begin{enumerate}
    \item $f_1(0,1) = f_1(1,2) = f_1(2,3) = f_1(3,0) = a$, and $f_1(0,2) = f_1(1,3) = c$;
    \item $f_2(0,1) = f_2(1,2) = f_2(2,3) = f_2(3,0) = c$, and $f_2(0,2) = f_2(1,3) = a$ (so $(A,f_1)$ and $(A,f_2)$ only differ by swapping $a$ with $c$);
    \item $f_3(0,1) = f_3(1,2) = f_3(2,3) = a$, and $f_3(0,2) = f_3(0,3) = f_3(1,3) = c$;
    \item $f_4(0,1) = f_4(2,3) = b$, $f_4(0,2) = f_4(1,3) = a$, $f_4(0,3) = f_4(1,2) = c$.
\end{enumerate}

Note that $(b \cup \id) \circ (b\cup \id) = b\cup \id =: e$ is an equivalence relation. Abusing notation, if $(A,f)$ is a consistent atomic $\ra{51}{65}$-network, we will also consider $e$ as an equivalence relation on $(A,f)$ whose equivalence classes are the maximal subsets of $A$ such that the range of the restriction of $f$ is $\{b,\id\}$.

\begin{lem}\label{lem:5165_unsat}
All networks $(A,f_1)$, $(A,f_2)$, $(A,f_3)$, and $(A,f_4)$ are consistent, but none is satisfiable in any representation of $\ra{51}{65}$.
\end{lem}
\begin{proof}
    Consistency is easy to check. Since $(A,f_1)$ and $(A,f_2)$ only differ by swapping $a$ and $c$, it is enough to verify unsatisfiability for $(A,f_1)$, $(A,f_3)$, and $(A,f_4)$.

    We start with $(A,f_1)$. Assume for a contradiction that there is some representation $\fC$ of $\ra{51}{65}$ satisfying $(A,f_1)$, that is, there is a map $s\colon A\to C$ such that for every $x,y\in A$ it holds that $(s(x),s(y)) \in f_1(x,y)^{\fC}$. Since $\fC$ is a representation of $\ra{51}{65}$, it follows that there exists $x\in C$ such that $(s(0),x) \in b^{\fC}$ and $(s(2),x)\in a^{\fC}$. It follows that $(s(3),x) \in c^{\fC}$ (it cannot be in $a^{\fC}$ because of the triple $(s(2),s(3),x)$, it cannot be in $b^{\fC}$ because of the triple $(s(0),s(3),x)$, and clearly $s(3) \neq x$). By an analogous argument we get $(s(1),x) \in c^{\fC}$, but then $(s(1),s(3),x)$ is a forbidden triple in $\fC$, a contradiction.

    Note that the same argument works also when one puts $f_1(0,3) = c$ instead, which proves that $(A,f_3)$ is not satisfiable.

    Finally, we look at $(A,f_4)$. Analogously as for $(A,f_1)$ we assume that it is satisfiable in some $\fC$, obtain a map $s\colon A\to C$, and find a vertex $x\in C$ such that $(s(0),x), (s(2),x) \in c^{\fC}$. It follows that $(s(1),x),(s(3),x) \in a^{\fC}$, and so the triple $(s(1),s(3),x)$ is forbidden, a contradiction.
\end{proof}

\begin{cor}\label{cor:5165_unsat}
    Let $(A,f)$ be a consistent atomic $\ra{51}{65}$-network with at least 4 equivalence classes of $e$. Then $(A,f)$ contains one of $(A,f_1)$, $(A,f_2)$, or $(A,f_3)$ as a subnetwork, and hence is unsatisfiable.
\end{cor}
\begin{proof}
    Find $B\subseteq A$ of size 4 such that no two vertices of $B$ are equivalent. Consequently, for every $x\neq y\in B$ it holds that $f(x,y) \in \{a,c\}$. It is easy to verify that every such $(B, f|_{B^2})$ is isomorphic to one of $(A,f_1)$, $(A,f_2)$, or $(A,f_3)$.
\end{proof}

Comer~\cite[Proposition 3]{ComerColorSchemes} gave a representation of $\ra{51}{65}$. He left the proof as an exercise to the reader. We will verify his representation and prove that, in fact,
Comer's representation is universal
(every satisfiable $\ra{51}{65}$-network is satisfiable in this representation).

\begin{prop}
\label{prop:51rep}
    There is a square representation $\fB$ of $\ra{51}{65}$ such that a consistent atomic $\ra{51}{65}$-network is satisfiable in $\fB$ if and only if it avoids $(A,f_1)$, $(A,f_2)$, $(A,f_3)$, and $(A,f_4)$. Consequently, $\fB$ is a finitely bounded universal representation of $\ra{51}{65}$.
\end{prop}
\begin{proof}
    First, we will construct $\fB$ and prove that it is a representation of $\ra{51}{65}$. We will follow the construction by Comer.

    Let $B:=\Q$, and let $B_0, B_1, B_2$ be a partition of $\Q$ into three dense subsets. Such a partition exists, for instance by enumerating all non-empty open intervals with rational endpoints and recursively choosing in each interval a fresh element for each of the three sets. Note that every dense subset of $\Q$ is cofinal (i.e., $\forall x \in \Q \ \exists y \in B_i\colon x \leq y$) and coinitial (i.e., $\forall x \in \Q \ \exists y \in B_i\colon y \leq x$).

    We now define a representation $\fB$ with domain $B\colon$
    \begin{align*}
        \id^\fB   & := \{(x,x) \mid x \in \Q \}, \\
        a^\fB     & := \{(x,y) \mid \exists i \in \{0,1, 2\}\, (\min(x,y) \in B_i \land \max(x,y) \in B_{i +_3 1})\}, \\
        b^\fB     & := \{(x,y) \mid x \neq y \land \exists i \in \{0,1, 2\}\, (x,y \in B_i)\}, \\
        c^\fB     & := B^2 \setminus (\id^\fB \cup \, a^\fB \cup b^\fB).
    \end{align*}
    Obviously, every relation of $\fB$ is symmetric. We will now compute the composition table of $\fB$. For the rest of the proof, we omit specifying $\fB$ in the superscript of the relations, addition will always be modulo $3$, and the variable $i$ will always be in $\{0,1,2\}$. Note that $\id \circ\,d = d \circ \id = d$ for all $d \in \{a, b, c\}$.

    For $x,z \in \Q$ the following holds:
    \begin{align*}
       (x,z) \in b \circ b
       &\iff \exists y \in \Q\, ((x,y), (y,z) \in b) \\
       &\iff \exists i  \, \exists y \in B_i\, (x\neq y \land y \neq z \land x,y,z \in B_i) \\
       &\iff \exists i  \, (x, z \in B_i) \\
       &\iff (x,z) \in b \cup \id.\\ \\ \\
        (x, z) \in a \circ a
        &\iff \exists y \in \Q\, ((x,y) \in a \land (y,z) \in a) \\
        &\iff \exists i   \, \exists y \in B_i\,
        \begin{cases}
            &((x, z< y) \land x, z \in B_{i+2})\\
            \lor &((y < x,z) \land x, z \in B_{i+1})\\
            \lor &(x < y < z \land x \in B_{i+2} \land z \in B_{i+1})\\
            \lor &(z < y < x \land x \in B_{i+1} \land z \in B_{i+2})
        \end{cases}\\
        &\iff \exists i \, ((x, z \in B_i) \lor (x < z \land z \in B_i \land x \in B_{i+1}) \lor (z < x \land x \in B_i \land z \in B_{i+1}))\\
        &\iff (x, z) \in b \cup c \cup \id.
    \\ \\ \\
        (x, z) \in c \circ c
        &\iff \exists y \in \Q\, ((x,y) \in c \land (y,z) \in c) \\
        &\iff \exists i   \, \exists y \in B_i\,
        \begin{cases}
            &((x, z < y) \land x, z \in B_{i+1})\\
            \lor &((y < x,z) \land x, z \in B_{i+2})\\
            \lor &(x < y < z \land x \in B_{i+1} \land z \in B_{i+2})\\
            \lor &(z < y < x \land x \in B_{i+2} \land z \in B_{i+1})
        \end{cases}\\
        &\iff \exists i \, ((x, z \in B_i) \lor (x < z \land x \in B_i \land z \in B_{i+1}) \lor (z < x \land z \in B_i \land x \in B_{i+1}))\\
        &\iff (x, z) \in a \cup b \cup \id.
    \\ \\ \\
        (x, z) \in a \circ b
        &\iff \exists y \in \Q\, ((x,y) \in a \land (y,z) \in b) \\
        &\iff \exists i   \, \exists y \in B_i\, (y \neq z \land z \in B_i \land ((x < y \land x \in B_{i+2}) \lor (y < x \land x \in B_{i+1})))\\
        &\iff \exists i  \, (z \in B_i \land x \not\in B_i)\\
        &\iff (x, z) \in a \cup c.\\ \\ \\
        (x, z) \in b \circ c
        &\iff \exists y \in \Q\, ((x,y) \in b \land (y,z) \in c) \\
        &\iff \exists i   \, \exists y \in B_i\, (x \neq y \land x \in B_i \land ((z < y \land z \in B_{i+1}) \lor (y < z \land z \in B_{i+2})))\\
        &\iff \exists i  \, (x \in B_i \land z \not\in B_i)\\
        &\iff (x, z) \in a \cup c.\\ \\ \\
        (x, z) \in a \circ c
        &\iff \exists y \in \Q\, ((x,y) \in a \land (y,z) \in c) \\
        &\iff \exists i   \, \exists y \in B_i\,
        \begin{cases}
        &(y < x < z \land x \in B_{i+1} \land z \in B_{i+2})\\
        \lor &(y < z < x \land x \in B_{i+1} \land z \in B_{i+2})\\
        \lor &(x < y < z \land x \in B_{i+2} \land z \in B_{i+2})\\
        \lor &(x < z < y \land x \in B_{i+2} \land z \in B_{i+1})\\
        \lor &(z < y < x \land x \in B_{i+1} \land z \in B_{i+1})\\
        \lor &(z < x < y \land x \in B_{i+2} \land z \in B_{i+1})\\
        \end{cases}\\
        &\iff (x, z) \in a \cup b \cup c.
    \end{align*}

    Since $a \circ b$, $a \circ c$, and $b \circ c$ are unions of $a$, $b$, $c$, and $\id$, which are all symmetric, it follows that $a \circ b = b \circ a$, $a \circ c = c \circ a$, and $b \circ c = c \circ b$, which concludes the proof that $\fB$ is indeed a representation of $\ra{51}{65}$.

    \medskip

    Next, we prove that a consistent atomic $\ra{51}{65}$-network $(C, f)$ is satisfiable in $\fB$ if and only if it avoids $(A,f_1)$, $(A,f_2)$, $(A,f_3)$, and $(A,f_4)$. Clearly, if $(C,f)$ contains one of $(A,f_1)$, $(A,f_2)$, $(A,f_3)$, and $(A,f_4)$ then it is not satisfiable at all by Lemma~\ref{lem:5165_unsat}.

    So assume that $(C,f)$ is a consistent atomic $\ra{51}{65}$-network which avoids $(A,f_1)$, $(A,f_2)$, $(A,f_3)$, and $(A,f_4)$. We will find a partition $C = C_0\sqcup C_1 \sqcup C_2$ and a linear order on $C$ such that any increasing function $\varphi\colon C\to B$ such that $\varphi[C_i]\subseteq B_i$ witnesses that $\fB$ represents $(C,f)$.

    By Corollary~\ref{cor:5165_unsat} we know that $(C,f)$ has three (possibly empty) equivalence classes of $e$; put $C_0\sqcup C_1 \sqcup C_2 = C$ to be the classes. Note that, for fixed $0\leq i < j\leq 2$ and for every $x\in B_i, y\in B_j$, the relation of $(x,y)$ in $\fB$ only depends on the order of $x$ and $y$. Consequently, for every $0\leq i < j\leq 2$ and for every $x\in C_i, y\in C_j$, the value $f(x,y)$ determines the order of $\varphi(x)$ and $\varphi(y)$ for every possible function $\varphi\colon C\to B$ which maps $C_k\to B_k$ for every $k\in \{0,1,2\}$ and witnesses that $\fB$ represents $(C,f)$. Let $<$ be the binary relation on $C$ obtained from this. Clearly it is antisymmetric. If there is a linear extension of $<$ then we are done.

    Assume for a contradiction that there is no such linear extension. It follows that there is a sequence $c_0,c_1,\ldots, c_{n-1}\in C$ such that, for every $0\leq i < n$, $c_i < c_{i+_n 1}$. We can assume that $c_0,c_1,\ldots, c_{n-1}\in C$ is the shortest such sequence. Observe that, for every $0\leq i < n$, we have that $c_i$ and $c_{i +_n 1}$ are in different classes of $e$ in $C$. Also note that from the minimality of $n$ it follows that $<$ is not defined on any pair $c_i, c_j$ with $\lvert i-j\rvert \geq 2$. Consequently, we have that $n\in \{3,4\}$.

    If $n=3$ then we have that, without loss of generality $c_0\in C_0$, $c_1\in C_1$, and $c_2\in C_2$. However, from the definition of relations in $\fB$ it follows that $(c_0,c_1,c_2)$ is a forbidden monochromatic triple, a contradiction.

    So $n=4$, and we have that $c_0,c_2\in C_i$ and $c_1,c_3\in C_j$ for some $i\neq j\in \{0,1,2\}$. Consequently, from the definition of relation on $\fB$ we get that $f(c_0,c_1) = f(c_2,c_3) \neq f(c_0,c_3) = f(c_1,c_2) \in \{a,c\}$. This implies that for $D = \{c_0,c_1,c_2,c_3\}$ we have that $(D, f|_{D^2})$ is isomorphic to $(A,f_4)$, a contradiction.

    Hence there is a linear extension of $<$. Note that since all $B_i$'s are dense, cofinal and coinitial, and $C$ is finite, there exists a desired function $\varphi$.
\end{proof}

\subsubsection{The Relation Algebra \texorpdfstring{$\ra{56}{65}$}{56\textunderscore 65}}
\label{sect:56_65}
The relation algebra $\ra{56}{65}$ has almost the same forbidden triples as $\ra{51}{65}$, but additionally allows the triple $(a,b,b)$ (and so $b$ is no longer an equivalence relation). This relation algebra has a representation, which according to R.~Maddux and P.~Jipsen (personal communication) was found by E.~Lukács.
However, it does not have a fully universal one (Proposition~\ref{prop:56_65-not-fu}).

\begin{prop}
\label{prop:56rep}
    $\ra{56}{65}$ has a square representation.
\end{prop}
\begin{proof}
Let $S_0, S_1, \hdots, S_5$ be a partition of the open interval $B := (0, 1) \cap {\mathbb Q}$ such that $S_i$ is dense in $(0, 1)$ for each $i \in \{0, \hdots, 5\}.$ In what follows, all operations with indices of the sets $S_i$ are performed in $\mathbb Z_6$. Let $\fB$ be the structure with domain $B$, where, for $i, j \in \{0, \hdots, 5\},$ $x \in S_i$ and $y \in S_j$ with $x < y,$
\begin{align*}
    (x, y), (y, x) \in a^{\mathfrak B} &\iff j-i \in \{2, 5\},\\
    (x, y), (y, x) \in b^{\mathfrak B} &\iff j-i \in \{0, 1\}, \text{~and~}\\
    (x, y), (y, x) \in c^{\mathfrak B} &\iff j-i \in \{3, 4\}.\\
\end{align*}
To prove that $\fB$ is indeed a representation of $\ra{56}{65}$, one needs to compute the composition table as in Proposition~\ref{prop:51rep}. We will only exemplarily show that $b \circ b = a \cup b \cup \id$; the other cases can be treated very similarly.

For one direction, suppose that $(x, z) \in b \circ b$ with $x < z.$ By the definition of $b \circ b,$ there is some $y \in B$ such that $(x, y) \in b$ and $(y, z) \in b.$ Let $i, j, k \in \{0, \hdots, 5\}$ such that $x \in S_i, y \in S_j$ and $z\in S_k.$
        If $y < x < z,$ we have
        \begin{align*}
            k-i = (k-j)+(j-i) = (k-j) - (i-j) \in \{0,1\}-\{0,1\} = \{0,1,5\};
        \end{align*}
        if $x < y < z,$ we have
        \begin{align*}
            k-i = (k-j)+(j-i) \in \{0,1\}+\{0,1\} = \{0,1,2\};
        \end{align*}
        and if $x < z < y,$ we have
        \begin{align*}
            k-i = (k-j)+(j-i) = -(j-k)+(j-i) \in -\{0,1\}+\{0,1\} = \{0, 5\}+\{0, 1\} = \{0, 1, 5\}.
        \end{align*}
        Thus, $(k-i) \in \{0, 1, 2, 5\}$ and consequently $(x, z) \in a \cup b.$

    In the other direction, let $(x, z) \in a \cup b \cup \id$ with $x \leq z$ and let $i, k \in \{0, \hdots, 5\}$ be such that $x \in S_i$ and $z \in S_k.$ If $(x, z) \in \id$, i.\,e.\ $x = z,$ choose an arbitrary $y \in S_i$ different from $x$ to obtain $(x, y) \in b,$ $(y, z) \in b$ and hence $(x, z) \in b \circ b.$ For the case $(x, z) \in a,$ we distinguish two cases: If $k-i = 2,$ choose an arbitrary $y \in S_{i+1}$
        with $x < y < z;$ if $k-i = 5,$ choose an arbitrary $y \in S_i$ with $x < z < y.$ If $(x, z) \in b,$ we can choose an arbitrary $y \in S_i$ with $x < y < z.$ In all cases, it holds that $(x, y) \in b$ and $(y, z) \in b.$
\end{proof}

\begin{figure}
    \centering
    \[
    \begin{tikzcd}
        \bullet && \bullet \\
        \\
        \bullet && \bullet & \bullet\,x
        \arrow["c", no head, from=1-1, to=1-3]
        \arrow["c"', no head, from=1-1, to=3-1]
        \arrow["a"{pos=0.7}, no head, from=1-1, to=3-3]
        \arrow["c", no head, from=1-3, to=3-3]
        \arrow["a"{pos=0.7}, no head, from=3-1, to=1-3]
        \arrow["c"', no head, from=3-1, to=3-3]
        \arrow["b", no head, from=1-3, to=3-4]
        \arrow["c", no head, bend right=40, from=3-1, to=3-4]
        \arrow[dashed, no head, bend left=90, from=1-1, to=3-4]
        \arrow[dashed, no head, from=3-3, to=3-4]
    \end{tikzcd}
    \]\caption{AP$(3, 2, 4)$ fails for $\ra{56}{65}$.}
    \label{fig:ap-failure-56}
\end{figure}

\begin{prop}\label{prop:56_65-not-fu}
$\ra{56}{65}$ does not have a fully universal representation.
\end{prop}
\begin{proof}
Consider the diagram from Figure~\ref{fig:ap-failure-56}.
It is easy to see that all assignments of atoms to the dashed edges induce one of the forbidden triples $(a,a,a)$, $(c,c,c)$,  or $(c,b,b)$, so $\ra{56}{65}$ does not have AP$(3,2,4)$. Hence, it does not have a fully universal representation by Theorem \ref{thm:fu}.
\end{proof}

Recently, it was shown that the algebra $\ra{56}{65}$ has a finitely bounded universal representation~\cite{56paper} (this was posed as a question in an early preprint of this paper); hence, its NSP is contained in $\NP$ by Lemma \ref{lem:fin_bounded_np}. The representation is closely connected to the generic circular triangle-free graph studied in~\cite{BGP2025Circular}. The construction involves deeper results about signed graphs and requires substantially different techniques.

\section{NP-hard Algebras with at most 4 Atoms}
\label{sect:NPhard}
In this section we present $\NP$-hardness results for relation algebras with at most 4 atoms. If the relation algebra has a normal representation, it falls into the scope of the tractability conjecture (Conjecture~\ref{conj:tractability}); in order to confirm the conjecture for the relation algebras with a normal representation, we also explain how to obtain a pp-construction of $(\{0,1\};\NAE)$ in each of these cases.

For the relation algebras with at most 3 atoms, a classification has already been established by~\cite{HirschCristani}, with two corrections in~\cite{BodirskyKnaeRamics}; in particular, $\ra{6}{7}$ is $\NP$-complete.

For the relation algebras with 4 atoms,
the NSP for
$\ra{15}{37}$ has been shown to be $\NP$-complete by Broxvall and Johnsson~\cite{BroxvallJonsson}:
it has a normal representation (Section~\ref{sec:normal_repr}) and in the countable normal representation, the relation $r$ denotes a homogeneous strict partial order.
The CSP of all reducts of this representation have been classified by Broxvall and Johnsson~\cite{BroxvallJonsson}, and the hardness result follows since the maximal tractable classes presented in Section~4 there do not cover all the relations of $\ra{15}{37}$.
We thus have the following.

\begin{prop}\label{prop:15_37}
    $\NSP(\ra{15}{37})$ is $\NP$-complete.
\end{prop}

The hardness proof given in ~\cite{BroxvallJonsson} is a primitive positive construction of the Betweenness relation $\Betw$ on ${\mathbb Q}$; since the structure $({\mathbb Q};\Betw)$ primitively positively interprets $(\{0,1\};\NAE)$ (see, e.g.,~\cite{BP-reductsRamsey}), this confirms the tractability conjecture for the normal representation of $\ra{15}{37}$.
The remaining hardness proofs are grouped in the following subsections according to the applied proof method.

\subsection{Hardness for Bounded Square Representations}
\label{sect:bounded-square-hard}
We have already studied representations of $\ra{39}{65}$
and $\ra{62}{65}$ in Section~\ref{sect:bounded-square} (they do not have fully universal representations).

\begin{prop}
\label{prop:bounded-hard}
    $\NSP(\ra{39}{65})$ and $\NSP(\ra{62}{65})$
    are $\NP$-complete.
\end{prop}

\begin{proof}
    The statement follows from Lemma~\ref{lem:bounded} together with Proposition~\ref{prop:39_65-repr} and Proposition~\ref{prop:62_65-repr}, respectively.
\end{proof}

\subsection{Symmetric Algebras with a Flexible Atom}
\label{subsec:flex_atoms}

The network satisfaction problem for finite symmetric relation algebras with a flexible atom is in $\Ptime$ or $\NPc$~\cite{BodirskyKnaeJAIR}.
The tractability conjecture (Conjecture~\ref{conj:tractability}) has already been confirmed for this case (compare to \cite[Theorem 9.1]{BodirskyKnaeuerDatalog23}).
To determine which of the two cases applies in the concrete cases of relation algebras with four atoms in this article, we need the following proposition.

\begin{prop}[{\cite[Theorem 2.21 and Proposition 6.1]{BodirskyKnaeJAIR}}]\label{prop:flexNPcBin}
    Let $\bA \in \RA$ be a finite, symmetric, integral relation algebra with a flexible atom and $\fB$ a normal representation of $\bA$. If $\fB$ does not have a binary injective polymorphism, then $\CSP(\fB)$ is $\NP$-complete.
\end{prop}

\begin{prop}[{\cite[Lemma 4.3]{BodirskyKnaeuerDatalog23}}]\label{prop:normBin1Cycle}
    Let $\bA \in \RA$ be a finite symmetric relation algebra with a normal representation $\fB$ that has a binary injective polymorphism. Then $\bA$ has all 1-cycles.
\end{prop}

\begin{prop}
\label{prop:flexNo1CycleNSP}
    If $\bA \in \{\ra{32}{65}, \ra{33}{65}, \ra{55}{65}, \ra{57}{65}, \ra{59}{65}, \ra{63}{65}, \ra{64}{65}\}$, then $\NSP(\bA)$ is $\NP$-complete.
\end{prop}

\begin{proof}
    All the relation algebras appear in Table~\ref{fig:flex}, so they have a flexible atom. Moreover, they are symmetric and except for the case of $\ra{59}{65}$ they lack the 1-cycle $(c,c,c)$. The algebra $\ra{59}{65}$ does not have the 1-cycle $(b,b,b)$. Thus, the normal representation $\fB$ of $\bA$ has no binary injective polymorphism by Proposition~\ref{prop:normBin1Cycle},
    and therefore, by Proposition~\ref{prop:flexNPcBin}, the NSP is $\NPc$.
\end{proof}

\begin{thm}[{\cite[Theorem 9.1]{BodirskyKnaeJAIR}}]
\label{thm:flex_atom_dichotomy}
    Let $\bA$ be a finite symmetric relation algebra with a flexible atom and atom structure $\fAo$. Then $\CSP(\fA_0)$ and $\NSP(\bA)$ are polynomial-time equivalent.
\end{thm}

\begin{prop}
\label{prop:34_65}
    $\NSP(\ra{34}{65})$ is $\NP$-complete.
\end{prop}

\begin{proof}
    We want to show that there is no polymorphism of the atom structure $\fA_0$ of $\ra{34}{65}$ such that its restriction to $\{b, c\}$ is a binary symmetric, majority or minority operation.
    By Theorem~\ref{thm:dichotomyConservative}, $\CSP(\fA_0)$ is $\NPc$ and by Theorem~\ref{thm:flex_atom_dichotomy}, $\NSP(\bA)$ is $\NPc$.

    First, assume that $m$ is a binary symmetric polymorphism of $\fA_0$
    when restricted to $\{b,c\}$. We consider the case that $m(b,c) = m(c,b) = b$. This implies $m(\id,c) = m(c,\id) = \id$ as otherwise $(m(\id,c),m(b,c),m(b,c))$ would equal a forbidden triple (recall that $\fA_0$ is conservative). But now $(m(\id,c),m(c,\id),m(c,c)) = (\id,\id,c)$ is a forbidden triple, a contradiction. Similar reasoning can be applied in the case of $m(b,c) = m(c,b) = c$. So $m$ cannot exist.

    Next, assume for contradiction that there exists a polymorphism $m$ which is a majority when restricted to $\{b,c\}$.
    If we apply $m$ to the allowed triples $(b,b,b), (b,c,a), (c,c,a)$, we obtain
    \begin{align*}
        (m(b,b,c), m(b,c,c), m(b,a,a)) = (b,c,m(b,a,a)) \in \Cy^{\fA_0}
    \end{align*}
    and hence $m(b,a,a) = a$.
    If we apply $m$ to the allowed triples $(c,c,\id), (b,b,b), (c,c,\id)$, we obtain
    \begin{align*}
        (m(c,b,c), m(c,b,c), m(\id,b,\id)) = (c,c,m(\id,b,\id)) \in \Cy^{\fA_0}
    \end{align*}
    and hence $m(\id,b,\id) = \id$. By permuting the triples, we analogously get $m(\id,\id,b) = \id$.
    Using this, we apply $m$ to the allowed triples $(\id,b,b), (b,c,a), (\id,b,b)$ and get
    \begin{align*}
        (m(\id,b,\id), m(b,c,b), m(b,a,b)) = (\id,b,m(b,a,b)) \in \Cy^{\fA_0},
    \end{align*}
    hence $m(b,a,b) = b$.
    Putting it all together, we apply $m$ to the allowed triples $(\id,b,b), (\id,a,a), (b,b,a)$ and get
    \begin{align*}
        (m(\id,\id,b), m(b,a,b), m(b,a,a)) = (\id,b,a) \in \Cy^{\fA_0},
    \end{align*}
    which is a contradiction.

    Finally, assume towards a contradiction that there is a polymorphism $m$ that is a minority when restricted to $\{b,c\}$.
    If we apply $m$ to the allowed triples $(b,c,a), (b,b,b), (c,c,a)$, we obtain
    \begin{align*}
        (m(b,b,c), m(c,b,c), m(a,b,a)) = (c,b,m(a,b,a)) \in \Cy^{\fA_0}
    \end{align*}
    and hence $m(a,b,a) = a$.
    Applying $m$ to the allowed triples $(c,c,\id), (b,b,\id), (b,b,b)$, we get
    \begin{align*}
        (m(c,b,b), m(c,b,b), m(\id, \id, b)) = (c,c,m(\id,\id,b)) \in \Cy^{\fA_0}
    \end{align*}
    and hence $m(\id,\id,b) = \id$.
    Finally, applying $m$ to the allowed triples $(\id,b,b), (\id,b,b), (b,c,a)$ yields
    \begin{align*}
        (m(\id,\id,b), m(b,b,c), m(b,b,a)) = (\id,c,m(b,b,a)) \in \Cy^{\fA_0};
    \end{align*}
    since $(\id, c, x)$ is only allowed for $x = c$, this implies $m(b,b,a) = c$, contradicting the conservativity of $m$.
\end{proof}

\subsection{Finitely Many Equivalence Classes}
\label{subsec:finitelyclasses}

An equivalence relation $E \subseteq B^2$ is called \emph{proper} if $E$ is a proper subset of $B^2$; it is called \emph{trivial} if it is the equality relation on $B$, and non-trivial otherwise.

\begin{thm}[{\cite[Theorem 23]{BodirskyKnaeRamics}}]
\label{thm:normal_hard_finite}
    Let $\bA$ be a finite representable relation algebra with normal
    representation $\fB$ and a non-trivial proper equivalence relation $e \in A$ such that $e^\fB$ has only finitely many classes. Then $\CSP(\fB)$ is $\NP$-complete.
\end{thm}

The proof of Theorem~\ref{thm:normal_hard_finite} in \cite{BodirskyKnaeRamics} already confirms the tractability conjecture (Conjecture~\ref{conj:tractability}) in this case.

\begin{table}
\centering
    \begin{tabular}{|L|L|L|}
        \hline
        \text{Algebra} & \text{Equivalence Relation} & \text{Forbidden Triples}  \\ \hline
        \ra{3}{37} & a \cup \id & (r,r,r) \\ \hline
        \ra{4}{37} & a \cup \id & (r,r,r) \\ \hline
        \ra{7}{37} & r \cup \Breve{r} \cup \id & (a,a,a) \\ \hline
        \ra{9}{37} & r \cup \Breve{r} \cup \id & (a,a,a) \\ \hline
        \ra{11}{37} & r \cup \Breve{r} \cup \id & (a,a,a) \\ \hline
        \ra{18}{37} & a \cup \id & (r,r,r) \\ \hline
        \ra{20}{37} & a \cup \id & (r,r,r) \\ \hline
        \ra{1}{65} & a \cup \id & (b,b,b), (c,c,c) \\ \hline
        \ra{2}{65} & a \cup \id & (b,b,b), (c,c,c) \\ \hline
        \ra{3}{65} & a \cup b \cup \id & (c,c,c) \\ \hline
        \ra{4}{65} & a \cup b \cup \id & (c,c,c) \\ \hline
        \ra{10}{65} & a \cup b \cup \id & (c,c,c) \\ \hline
        \ra{11}{65} & a \cup b \cup \id & (c,c,c) \\ \hline
        \ra{25}{65} & a \cup \id & (b,b,b), (c,c,c) \\ \hline
        \ra{26}{65} & a \cup \id & (b,b,b), (c,c,c) \\ \hline
        \ra{28}{65} & a \cup \id & (b,b,b), (c,c,c) \\ \hline
    \end{tabular}
    \caption{Relation algebras where Theorem~\ref{thm:normal_hard_finite} applies.}
    \label{tab:normal_hard_finite}
\end{table}

\begin{prop}
\label{prop:normal_hard_finite}
    Let
    \begin{align*}
    \bA \in \{\ra{3}{37}, \ra{4}{37}, \ra{7}{37}, \ra{9}{37}, \ra{11}{37}, \ra{18}{37}, \ra{20}{37}, \ra{1}{65}, \ra{2}{65}, \ra{3}{65}, \ra{4}{65}, \ra{10}{65}, \ra{11}{65}, \ra{25}{65}, \ra{26}{65}, \ra{28}{65}\},
    \end{align*}
    then $\NSP(\bA)$ is $\NP$-complete.
\end{prop}

\begin{proof}
    Use Lemma~\ref{lem:finite_classes} to see that the conditions for Theorem~\ref{thm:normal_hard_finite} are satisfied. Table~\ref{tab:normal_hard_finite} shows the choice of $e$ for each of the relation algebras.
\end{proof}

\subsection{No Non-trivial Equivalence Relation}
\label{subsec:nonontrivial}

A permutation group is called \emph{primitive} if it does not preserve a non-trivial proper equivalence relation.
Note that if $\bA$ is a finite relation algebra with a normal representation $\fB$, then $\Aut(\fB)$ is primitive if and only if $\bA$ does not contain equivalence relations besides $1$ and $\id$.

\begin{thm}[{\cite[Theorem 29]{BodirskyKnaeRamics}}]
\label{thm:primitiv_hard}
    Let $\bA$ be a relation algebra with a normal representation $\fB$ such that $|B| > 2$. If $\Aut(\fB)$ is primitive and $a \in A_0$ is symmetric such that $(a,a,a)$ is forbidden in $\bA$, then $\CSP(\fB)$ and $\NSP(\bA)$ are $\NP$-complete.
\end{thm}

The proof of Theorem~\ref{thm:primitiv_hard} in \cite{BodirskyKnaeRamics} already confirms the tractability conjecture (Conjecture~\ref{conj:tractability}) in this case.

\begin{prop}\label{prop:prim}
    If $\bA \in \{\ra{23}{37},\ra{36}{37},\ra{46}{65}\}$, then $\NSP(\bA)$ is $\NP$-complete.
\end{prop}

\begin{proof}
    The relation algebras $\ra{23}{37},\ra{36}{37}$, and $\ra{46}{65}$ have normal representations (see~\eqref{eq:norm_all}).
    In the case of $\ra{23}{37}$ and $\ra{36}{37}$, the triple
    $(a,a,a)$ is forbidden,
    and in the case of $\ra{46}{65}$ the triple $(c,c,c)$ is forbidden.
    Moreover, each of the three algebras does not contain equivalence relations besides $1$ and $\id$.
    We verify this for $\ra{46}{65}$, and leave the proof for the other two algebras to the reader.
    \[ (a \cup \id)^2 = 1, \quad (b \cup \id)^2 = a \cup b \cup \id, \quad (c \cup \id)^2 = b \cup c \cup \id.\]
    Furthermore, since $(a,b,c)$ is an allowed triple,
    \[ (a \cup b \cup \id)^2 = (a \cup c \cup \id)^2 = (b \cup c \cup \id)^2 = 1.\]
    So the statement follows from Theorem~\ref{thm:primitiv_hard}.
\end{proof}

\subsection{Hardness for 2-Cycle Products}
\label{sect:cycle-prod-hardness}
The following two lemmas imply that if a relation algebra is a 2-cycle product of two relation algebras $\bA$ and $\bB$, and $\bA$ or $\bB$ has an $\NP$-complete NSP, then the
network satisfaction problem for the 2-cycle product has an $\NP$-complete NSP as well.

\begin{lem}
\label{lem:2-cycle-np}
    Let $\bA$ and $\bB$ be finite integral relation algebras such that
    $\NSP(\bA)$ and $\NSP(\bB)$ are both in $\NP$.
    Then $\NSP(\bA[\bB])$ is in $\NP$ as well.
\end{lem}
\begin{proof}
    Let $N = (V, f)$ be an $\bA[\bB]$-network. The algorithm proceeds as follows: First, guess an equivalence relation $W \subseteq V^2$. Let $C_1, \hdots, C_k$ be the equivalence classes of $W.$

    Define a function $h$ on these classes by
    \begin{align*}
        h(C_i, C_j) := (1_\bB \cap \overline{\id}) \cap \bigcap_{\substack{x \in C_i\\ y\in C_j}} f(x,y)
    \end{align*}
    for distinct $i, j \in \{1, \hdots, k\}$ and by $h(C_i, C_i) = \id$ for $i \in \{1, \hdots, k\}$. Now test if
    \begin{itemize}
        \item for each $i \in \{1, \hdots, k\}$, the network $(C_i, f\vert_{C_i^2} \cap 1_\bA)$ is a positive instance of $\NSP(\bA)$;
        \item the contracted network $\faktor{N}{W} := \left(\faktor{V}{W}, h\right)$ is a positive instance of $\NSP(\bB)$.
    \end{itemize}

    Note that $\faktor{V}{W}$ is well-defined by the definition of $h$.
    It is clear that if $N$ is satisfiable then any satisfying assignment gives an equivalence relation $W\subseteq V^2$ which, if guessed at the beginning by the algorithm, will lead to the algorithm identifying $N$ as satisfiable, which proves correctness of the algorithm.

    Conversely, assume that $N$ is accepted; let $\fB_1, \hdots, \fB_k$ be respective representations of $\bA$ witnessing the satisfiability of the networks $(C_i, f\vert_{C_i^2} \cap 1_\bA)$ and let $\fD$ be a representation satisfying $\faktor{N}{W}$ with solution $s$. Note that $s$ is injective on $\faktor{V}{W}$ since $h(C_i, C_j) \leq \overline{\id}$ for distinct $i, j \in \{1, \hdots, k\}$.
    Since $\bA$ and $\bB$ are integral, in particular simple, we can assume that these representations are all square (Corollary~\ref{cor:simple-implies-square-rep}).
    Define the structure $\fD[\fB_1, \hdots, \fB_k]$ from $\fD$ by replacing, for every $i \in \{1, \hdots, k\},$ the element $s(x_i) \in D$, where $x_i$ represents the class $C_i$, by $\fB_i$, and all other vertices by a fixed square representation $\fB_0$ of $\bA$. Then $\fD[\fB_1, \hdots, \fB_k]$ is a representation of $\bA[\bB]$ which satisfies $N$.
\end{proof}

\begin{lem}
\label{thm:2-cycle-red}
    Let $\bA$ and $\bB$ be finite representable relation algebras such that $A_0 \cap B_0 = \{\id\}$. Then there is a
    polynomial-time reduction from $\NSP(\bA)$ to $\NSP(\bA[\bB])$ and from  $\NSP(\bB)$ to $\NSP(\bA[\bB])$.
\end{lem}

\begin{proof}
    Let $N = (V, f)$ be a given $\bA$-network or $\bB$-network, respectively. Since $\bA, \bB \subseteq \bC := \bA[\bB]$ we know that $N$ is also a $\bC$-network. We will show that $N$ is a satisfiable $\bA$- or $\bB$-network, respectively, if and only if it is a satisfiable $\bC$-network.

    For the first direction, we will only consider the case of $N$ being an $\bA$-network, but the other case is analogous. Suppose that $N$ is satisfiable in a square representation $\fC$ of $\bA$, then choose any square representation $\fD$ of $\bB$. Then $\fD[\fC]$, which by Theorem~\ref{thm:2-cylce-repr} is a representation of $\bC$, also satisfies the $\bC$-network $N$. If $N$ is only satisfiable in a non-square representation $\fB$ of $\bA$, then consider the union over all structures $\fD[\fB_i]$, where $\{\fB_i \mid i \in I\}$ are the square components of $\fB$. This shows one direction for both cases.

    For the converse, first suppose that $N$ is an $\bA$-network which is satisfiable in a representation $\fB$ of $\bC$. Since $f$ only maps to $A$, $N$ is also satisfiable in the $A$-reduct of $\fB$ which is a (not necessarily square) representation of $\bA$.

    Next, let $N$ be a $\bB$-network which is satisfiable in a representation $\fB$ of $\bC$. By Remark~\ref{rem:equ_relation}, $1_\bA$ is an equivalence relation in $\bC$, so the contracted structure $\faktor{\fB}{1_\bA}$ is well-defined and a representation of $\bB$ which satisfies $N$.
\end{proof}

The proof of the following lemma will be important to show that our complexity dichotomy confirms the tractability conjecture (Conjecture~\ref{conj:tractability}). It can be shown analogously to the previous lemma.

\begin{lem}\label{lem:cycle-pp-construct}
If $\fA$ and $\fB$ are normal representations of $\bA$ and $\bB$, respectively, then $\fB[\fA]$ pp-constructs
$\fA$ and $\fB$.
\end{lem}

\begin{prop}
\label{prop:2CycleProNPc}
    If $\bA \in \{\ra{10}{37}, \ra{5}{65}, \ra{6}{65}, \ra{9}{65}, \ra{12}{65},  \ra{13}{65}, \ra{15}{65}, \ra{16}{65}, \ra{17}{65}, \ra{18}{65}\}$, then $\NSP(\bA)$ is $\NP$-complete.
    If $\bA$ has a normal representation $\fB$, then $\fB$ pp-constructs $(\{0,1\};\NAE)$.
\end{prop}

\begin{proof}
    In every case, the algebra $\bA$ is the 2-cycle product with one component being from $\ra{2}{3}$,
    $\ra{1}{7}$, $\ra{2}{7}$, $\ra{5}{7}$, and $\ra{6}{7}$ (see Table~\ref{fig:2-prod-4atom}).
    The $\NSP$ of these algebras is $\NP$-hard as stated in Table~\ref{fig:3AsyAtoms} and~\ref{fig:3SymAtoms}, so $\NSP(\bA)$ is $\NP$-hard as well by Lemma~\ref{thm:2-cycle-red}.
    The $\NP$-membership of $\NSP(\bA)$ follows from Lemma~\ref{lem:2-cycle-np}, as the NSP of both components of the 2-cycle product is in $\NP$.

    Since the normal representations of $\ra{2}{3}$, $\ra{1}{7}$, $\ra{2}{7}$, and $\ra{6}{7}$ all pp-construct $(\{0, 1\}; \NAE)$ (see Remark~\ref{rem:tractability_conj_1_7_and_2_3} and the already mentioned results in \cite{BodirskyKnaeRamics}), so do the representations of the respective algebras with normal representations ($\ra{10}{37}$, $\ra{5}{65}$, $\ra{6}{65}$, $\ra{13}{65}$, $\ra{16}{65}$, and $\ra{18}{65}$) by Lemma~\ref{lem:cycle-pp-construct} and the transitivity of pp-constructibility.
\end{proof}

\subsection{Hardness via Primitive Positive Constructions}
\label{sec:pp_reductions}

In the following, we will prove the $\NP$-hardness of seven NSPs by means of pp-interpretations. For that we will use various concepts from Section \ref{sect:CSPS-and-universal-algebra} as well as Definition \ref{def:xy-canonical}. For the first two cases, we will infer from the nonexistence of certain polymorphisms that the respective normal representation pp-interprets $(\{0, 1\}; \mathrm{NAE})$; the argument is summarized in the following proposition.

\begin{prop}\label{prop:normal-canonical-schaefer}
Let $\bA$ be a finite relation algebra with a normal representation $\fB$. If there are two distinct elements $x, y \in A_0$ such that every $f \in \Pol(\fB)$ is $\{x, y\}$-canonical, but not $\{x, y\}$-symmetric, a $\{x, y\}$-majority, or a $\{x, y\}$-minority, then $\NSP(\bA)$ is $\NP$-complete.
\end{prop}

\begin{proof}
Since $\fB$ is a normal representation of $\bA$, $\NSP(\bA) = \CSP(\fB)$ is contained in $\NP$.
Because every polymorphism of $\fB$ is $\{x, y\}$-canonical, the function $\xi$ with domain $\Pol(\fB)$ given by
\begin{align*}
        f \mapsto \text{~(the behaviour of $f$ on~} \{x, y\})
\end{align*}
is well-defined.
Note that, since
the behaviour of $f$ on $\{x,y\}$ is conservative,
$\xi(\Pol(\fB))$ is a clone on the two-element set $\{x, y\}$.
It is easy to see that $\xi$ is a clone homomorphism.
Note that if we have four vertices $x_1, x_2, x_3, x_4 \in B$ with $x^\fB(x_1, x_2)$ and $y^\fB(x_3, x_4)$, then, for any $f \in \Pol(\fB)$, the values of $f$ on tuples with components from $\{x_1, x_2, x_3, x_4\}$ suffice to determine $\xi(f)$; hence, $\xi$ is uniformly continuous (see Definition \ref{def:uniformly-continuous-clone-hom}).
By assumption, $\Pol(\fB)$ contains no $\{x, y\}$-symmetric operation, no $\{x, y\}$-minority and no $\{x, y\}$-majority. Since the behaviour of $f$ on $\{x,y\}$ is conservative, $\xi(\Pol(\fB))$ also does not contain the function $f$ satisfying $f(x) = y$ and $f(y) = x$.
Thus, by Post's Theorem \ref{thm:post41}, $\xi(\Pol(\fB))$ is the projection clone.
Since $\fB$ is homogeneous, it is in particular $\omega$-categorical. It follows from Theorem \ref{thm:clone-hom-to-proj-implies-NP-hard} that $\fB$ pp-constructs $(\{0, 1\}; \NAE)$ and $\CSP(\fB) = \NSP(\bA)$ is $\NP$-hard.
\end{proof}

In the relation algebra $\ra{29}{65}$, both
$b \cup \id$ and $c \cup \id$ denote an equivalence relation with infinitely many infinite classes.

\begin{prop}
    \label{prop:29_65}
    $\NSP(\ra{29}{65})$ is $\NP$-complete.
\end{prop}

\begin{proof}
Let $\fB$ be the normal representation of $\ra{29}{65}$ (see~\eqref{eq:norm_all}).
It can be checked by various case distinctions that the relation
\begin{align*}
    R = \{(x_1, x_2, y_1, y_2) \in B^4 \mid & \; (b \cup c)^\fB(x_1, x_2) \text{~and~} (b \cup c)^\fB(y_1, y_2) \\
    & \text{~and~} b^\fB(x_1, x_2) \iff b^\fB(y_1, y_2)\}
\end{align*} has the following pp-definition in $\fB$ (see Figure~\ref{fig:gadget_29_65}):
\begin{align*}
    (b\cup c)(x_1, x_2) \land
    (b\cup c)(y_1, y_2) \land
    \exists z_1, z_2, z_3, z_4, z_5, z_6\, \biggl(
    (b \cup c)(z_2, z_3) \land
    (b \cup c)(z_6, y_1)\\ \land\,
    (b \cup c)(x_1, z_1) \land
    (b \cup c)(x_2, z_2) \land
    (b \cup c)(z_1, z_2) \land
    (b \cup c)(z_3, z_4) \land
    (b \cup c)(z_3, z_5) \\\land\,
    (b \cup c)(z_4, z_6) \land
    (b \cup c)(z_5, z_6) \\\land\,
    (b \cup c \cup \id)(z_2, z_4) \land
    (b \cup c \cup \id)(z_5, y_1) \land
    (b \cup c \cup \id)(z_6, y_2) \land
    (b \cup c \cup \id)(z_1, z_3) \\ \land\,
    a(x_2, z_1) \land
    a(x_1, z_2) \land
    a(z_3, z_6) \land
    a(z_4, z_5)\bigg).
\end{align*}

\begin{figure}
    \centering
    \[
    \begin{tikzcd}
        x_1 && z_1 && z_3 && z_5 && y_1 \\ \\
        x_2 && z_2 && z_4 && z_6 && y_2
        \arrow["b \cup c", no head, from=1-1, to=1-3]
        \arrow["b \cup c \cup \id", no head, from=1-3, to=1-5]
        \arrow["b \cup c", no head, from=1-5, to=1-7]
        \arrow["b \cup c \cup \id", no head, from=1-7, to=1-9]
        \arrow["b \cup c", swap, no head, from=3-1, to=3-3]
        \arrow["b \cup c \cup \id", swap, no head, from=3-3, to=3-5]
        \arrow["b \cup c", swap, no head, from=3-5, to=3-7]
        \arrow["b \cup c \cup \id", swap, no head, from=3-7, to=3-9]
        \arrow["b \cup c", swap, no head, from=1-1, to=3-1]
        \arrow["b \cup c", no head, from=1-9, to=3-9]
        \arrow["b \cup c", no head, from=1-7, to=3-7]
        \arrow["b \cup c", no head, from=1-3, to=3-3]
        \arrow["b \cup c", no head, from=1-5, to=3-5]
        \arrow["b \cup c", swap, no head, from=3-7, to=1-9]
        \arrow["b \cup c", swap, no head, from=3-3, to=1-5]
        \arrow["a"{pos=0.6}, no head, from=1-3, to=3-1]
        \arrow["a"{pos=0.6}, no head, from=1-7, to=3-5]
        \arrow["a"{pos=0.6}, no head, from=1-1, to=3-3]
        \arrow["a"{pos=0.6}, no head, from=1-5, to=3-7]
    \end{tikzcd}
    \]\caption{In every solution $g$ to this $\ra{29}{65}$-network, we have $g(x_1, x_2) = g(y_1, y_2)$.}
    \label{fig:gadget_29_65}
\end{figure}

This implies that the polymorphisms of the normal representation are $\{b,c\}$-canonical. It is easy to see that there cannot be a $\{b,c\}$-symmetric polymorphism. Assume for contradiction that there exists a $\{b,c\}$-majority $m$. Then $m(a,b,a)=b$, since otherwise $(m(a,b,a),m(b,b,c),m(c,b,b))$, which has to be an allowed triple, would equal $(a,b,b)$, which is a contradiction. But now we get $(m(a,b,a),m(a,a,a),m(b,c,b)) = (b,a,b)$, which is also a contradiction to $m$ being a polymorphism. As a last step assume for contradiction that there exists a $\{b,c\}$-minority $m$. Then $(m(b,b,c),m(b,c,b),m(b,a,a))$ is either $(c,c,b)$ or $(c,c,a)$, which are both forbidden triples, a contradiction. It follows from Proposition~\ref{prop:normal-canonical-schaefer} that $\NSP(\ra{29}{65})$ is $\NP$-complete.
\end{proof}

The consistent atomic $\ra{35}{37}$-networks
may be viewed as oriented graphs that do not embed
the oriented graph
$T_3 := (\{0,1,2\}; \{(0,1),(1,2),(0,2)\})$. Like $\ra{33}{37}$, it has $a$ as a flexible atom and thus a normal representation; its NSP is $\NP$-complete as well.

\begin{figure}
    \centering
    \[
    \begin{tikzcd}
        x_0 && y_0 \\
        \\
        x_1 && y_1
        \arrow["{}", no head, from=1-1, to=1-3]
        \arrow["{}", swap, no head, from=1-1, to=3-1]
        \arrow["{}", no head, from=1-3, to=3-3]
        \arrow["{}"{pos=0.6}, no head, from=3-1, to=1-3]
        \arrow["{}", swap, no head, from=3-3, to=3-1]
    \end{tikzcd}
    \]\caption{The $\ra{35}{37}$-network $K(x_0, x_1, y_0, y_1)$; an edge stands for $(r \cup \breve{r})$. In every solution $g$ to it, we have $g(x_0, x_1) = g(y_0, y_1)$.}
    \label{fig:gadget_35_37}
\end{figure}
\begin{figure}
    \centering
    \[
    \begin{tikzcd}
        x_0 && z_1 && z_3 && z_5 && y_0 \\ \\
        x_1 && z_2 && z_4 && z_6 && y_1
        \arrow[no head, from=1-1, to=1-3]
        \arrow[no head, from=1-3, to=1-5]
        \arrow[no head, from=1-5, to=1-7]
        \arrow[no head, from=1-7, to=1-9]
        \arrow[no head, from=3-1, to=3-3]
        \arrow[no head, from=3-3, to=3-5]
        \arrow[no head, from=3-5, to=3-7]
        \arrow[no head, from=3-7, to=3-9]
        \arrow[no head, from=1-1, to=3-1]
        \arrow[no head, from=1-9, to=3-9]
        \arrow[no head, from=1-7, to=3-7]
        \arrow[no head, from=1-3, to=3-3]
        \arrow[no head, from=1-5, to=3-5]
        \arrow[no head, from=3-7, to=1-9]
        \arrow[no head, from=3-3, to=1-5]
        \arrow[no head, from=1-3, to=3-1]
        \arrow[no head, from=1-7, to=3-5]
    \end{tikzcd}
    \]\caption{The $\ra{35}{37}$-network $E(x_0, x_1, y_0, y_1)$; an edge stands for $(r \cup \breve{r})$.}
    \label{fig:gadget_35_37_part2}
\end{figure}
\begin{prop}\label{prop:35_37}
    $\NSP(\ra{35}{37})$ is $\NP$-complete.
\end{prop}
\begin{proof}
Let $\fB$ be the normal representation of $\ra{35}{37}$.
It follows immediately from the fact that $(r,r,r)$ is a forbidden triple that in every solution $g$ to the $\ra{35}{37}$-network $K(x_0, x_1, y_0, y_1)$ depicted in Figure~\ref{fig:gadget_35_37}, we have $g(x_0, x_1) = g(y_0, y_1)$.
It can be checked by various case distinctions that the relation 
\begin{align*}
    R = \{(x_0, x_1, y_0, y_1) \in B^4\mid & \; (r \cup \breve{r})^\fB(x_0, x_1) \text{~and~} (r \cup \breve{r})^\fB(y_0, y_1) \\
    & \text{~and~} r^\fB(x_0, x_1) \iff r^\fB(y_0, y_1)\}
\end{align*} has the following pp-definition in $\fB$ (see Figure~\ref{fig:gadget_35_37_part2}):
\begin{align*}
    \exists z_1, z_2, z_3, z_4, z_5, z_6\,\bigl(
    & \; K(x_0,x_1,z_1,z_2) \,\land\, K(z_1,z_2, z_3,z_4)  \\
    & \,\land\, K(z_3,z_4, z_5, z_6) \,\land\, K(z_5, z_6, y_0, y_1)\bigr).
\end{align*}
This implies that the polymorphisms of the normal representation are $\{r,\breve{r}\}$-canonical. Observe that, by the Cycle Law, the triples
$(r,r,r)$, $(r,\breve{r},r)$, $(\breve{r},r,r)$, $(\breve{r},r,\breve{r})$, $(r,\breve{r},\breve{r})$, and $(\breve{r},\breve{r},\breve{r})$
are all forbidden, and that $(r,r,\breve{r})$ and $(\breve{r},\breve{r},r)$ are allowed triples.

Assume towards a contradiction that there exists a $\{r,\breve{r}\}$-symmetric polymorphism, then applying it to the allowed triples $(r, r, \breve{r})$ and $(\breve{r}, \breve{r}, r)$ yields either $(r,r,r)$ or $(\breve{r}, \breve{r}, \breve{r})$, which are both forbidden, contradiction.

For the rest of the proof, fix $x,y,z \in B$ such that $(x,y), (y,z), (z,x) \in r^\fB$.
Such elements exist because $(\breve{r})^\fB \leq r^\fB \circ r^\fB$.

Assume towards a contradiction that there exists a $\{r,\breve{r}\}$-minority $f$. Since $(x,x) \in \id^\fB$ and $(x,y) \in r^\fB$, we obtain $(f(x,x,x), f(x,x,y)) \in (r \cup \id)^\fB$ since $f$ preserves $(r \cup \id)^\fB$. Moreover, since $(x,y) \in r^\fB$, $(x, z) \in \breve{r}^\fB$, $(y,z) \in r^\fB$, we get $(f(x,x,y), f(y,z,z)) \in \breve{r}^\fB$ since $f$ is a $\{r, \breve{r}\}$-minority. Similarly, $(f(x,x,x), f(y,z,z)) \in r^\fB$. Hence, $(f(x,x,x), f(x,x,y)) \not\in \id^\fB$ and thus $(f(x,x,x), f(x,x,y)) \in r^\fB$. But now $f(x,x,x)$, $f(x,x,y)$, and $f(y,z,z)$ form a triangle of type $(r,r,r)$, which is forbidden, contradiction.

Finally, assume towards a contradiction that there exists a $\{r,\breve{r}\}$-majority $g$.
Similar reasoning as before yields that the relations between $g(x,x,x)$, $g(y,y,y)$, $g(x,x,z)$, $g(z,y,y)$ as depicted in Figure~\ref{fig:gadget_35_37_no_majo_no_mino} hold. Since $(g(y,y,y), g(x,x,z)) \in (\breve{r})^\fB$ but $(g(z,y,y), g(x,x,z)) \in r^\fB$, we get $(g(y,y,y), g(z,y,y)) \not\in \id^\fB$ and hence $(g(y,y,y), g(z,y,y)) \in r^\fB$. But now $g(x,x,x)$, $g(y,y,y)$, and $g(z,y,y)$ form a triangle of type $(r,r,r)$, which is forbidden, contradiction.

It follows from Proposition~\ref{prop:normal-canonical-schaefer} that $\NSP(\ra{35}{37})$ is $\NP$-complete.
\begin{figure}
    \[
    \begin{tikzpicture}[baseline=(current bounding box.center)]
        \node (x) at (210:1.5cm) {$x$};
        \node (y) at (90:1.5cm) {$y$};
        \node (z) at (330:1.5cm) {$z$};
        \draw[->] (x) -- (y);
        \draw[->] (y) -- (z);
        \draw[->] (z) -- (x);
    \end{tikzpicture}
    \hspace{1cm}
    \begin{tikzpicture}[baseline=(current bounding box.center)]
        \node (xxx) at (210:1.5cm) {$f(x,x,x)$};
        \node (xxy) at (90:1.5cm) {$f(x,x,y)$};
        \node (yzz) at (330:1.5cm) {$f(y,z,z)$};
        \draw[->] (xxx) -- node[left] {\scriptsize{$r \cup \id$}} (xxy);
        \draw[->] (yzz) -- (xxy);
        \draw[->] (xxx) -- (yzz);
    \end{tikzpicture}
    \hspace{1cm}
    \begin{tikzcd}
        g(x,x,x) && g(y,y,y) \\ \\
        g(x,x,z) && g(z,y,y)
        \arrow[from=1-1, to=1-3]
        \arrow["r \cup \id", from=1-3, to=3-3]
        \arrow[from=3-3, to=3-1]
        \arrow[from=1-1, to=3-3]
        \arrow[from=3-1, to=1-3]
    \end{tikzcd}
    \]
    \caption{
        A proof that the normal representation $\fB$ of $\ra{35}{37}$ has no polymorphism $f$ that is a $\{r, \breve{r}\}$-minority and no polymorphism $g$ that is a $\{r,\breve{r}\}$-majority.
    }
    \label{fig:gadget_35_37_no_majo_no_mino}
\end{figure}
\end{proof}

We take a different approach for the following three hardness proofs, in which we will directly give a pp-construction of a structure with $\NP$-hard CSP in the normal representation of the respective relation algebra.

The relation algebra $\ra{19}{37}$ has a normal representation $\fB$ with domain ${\mathbb T} \times \{0,1\}$ and the following relations:
\begin{itemize}
    \item $a^{\fB} := \{\bigl((x,0),(x,1)\bigr) \mid x \in {\mathbb T}\} \cup \{\bigl((x,1), (x,0)\bigr) \mid x \in {\mathbb T}\}$.
    \item $r^{\fB} := \bigcup_{(x,y) \in E^\T} \big \{ \bigl ((x,0),(y,0) \bigr ),
    \bigl((x,1),(y,1) \bigr), \bigl((y,0),(x,1)\bigr), \bigl((y,1),(x,0)\bigr) \big \}$.
\end{itemize}

\begin{prop}
    \label{prop:19_37}
    $\NSP(\ra{19}{37})$ is $\NP$-complete.
\end{prop}
\begin{proof}
\begin{figure}
    \centering
    \pgfmathsetlengthmacro{\trianglecolsep}{3.2cm}
    \pgfmathsetlengthmacro{\trianglerowsep}{\trianglecolsep*sqrt(3)/3}
    \[
    \begin{tikzcd}[column sep={\trianglecolsep, between origins}, row sep={\trianglerowsep, between origins}]
        & {\mathstrut z} & \\
        & & \\
        & {\mathstrut w} & \\
        {\mathstrut x} && {\mathstrut y}
        \arrow["a \cup \id", swap, no head, from=1-2, to=4-1]
        \arrow["a \cup \id", no head, from=1-2, to=4-3]
        \arrow["a \cup \id", swap, no head, from=4-1, to=4-3]
        \arrow["r \cup \id"{description, fill=white, inner sep=.8pt}, from=4-1, to=3-2]
        \arrow["r \cup \id"{description, fill=white, inner sep=.8pt}, from=3-2, to=4-3]
        \arrow["a \cup r \cup \breve{r}"{description, fill=white, inner sep=.8pt}, no head, from=1-2, to=3-2]
    \end{tikzcd}
    \]
    \caption{The relation $T(x,y,z)$ in the normal representation $\fB$ of $\ra{19}{37}$.}
    \label{fig:gadget_19_37}
\end{figure}
Let $\fB$ be the normal representation of $\ra{19}{37}$ as defined above.
Since $a \circ a = \id$, we get that $a^\fB$ is a matching. Since $\id \leq a \circ a$, every vertex is incident to an $a^\fB$-edge, which implies that $a^\fB$ is even a perfect matching.

We give a pp-construction of $(\{0, 1\}; \NAE)$ in $\fB$.
Define the ternary relation $T$ by
\begin{align}\label{eq:19_37}
\begin{split}
    T = \bigl\{(x,y,z) \in B^3 \mid ~&(a \cup \id)(x, y) \,\land\, (a \cup \id)(x, z) \,\land\, (a \cup \id)(y, z)\\ &\land\, \exists w\, \bigl((r \cup \id)(x, w) \,\land\, (r \cup \id)(w, y) \,\land\, (a \cup r \cup \breve{r})(z, w)\bigr)\bigr\}
\end{split}
\end{align}
(see Figure~\ref{fig:gadget_19_37}).

It can be seen by case distinctions that if $T(x,y,z)$ holds for some $x,y,z \in B$, then exactly two of $x, y, z$ are equal. Conversely, let $x,y,z$ belong to the same equivalence class of $(a \cup \id)^\fB$ and suppose that they are not all equal.
If $x=y$, then $z\neq x$, and we may choose $w=x$.
If $x\neq y$, then $a^\fB(x,y)$. Since
$\id \leq r\circ \breve{r}$, there exists $w\in B$ such that
$r^\fB(x,w)$. Since $a \circ r = \breve{r}$, it follows that $r(w, y)$.
Moreover, $w$ is distinct from both $x$ and $y$, and hence also from
$z\in\{x,y\}$. Thus, in both cases the constraints in
\eqref{eq:19_37} are satisfied.

By what we observed above, we get
\begin{align*}
    T \cap \{x, y\}^3 = \{(x,x,y), (x,y,x), (y,x,x), (y,y,x), (y,x,y), (x,y,y)\}
\end{align*}
for every pair $(x, y) \in B^2$ with $a^\fB(x, y)$; i.e., $(\{x, y\}; T\,\cap\,\{x, y\}^3)$ is isomorphic to $(\{0, 1\}; \NAE)$.
Since $a^\fB$ is a perfect matching, this implies that $\fC := (B; T)$ is isomorphic to a disjoint union of copies of $(\{0, 1\}; \NAE)$. Hence, $\fB$ pp-interprets a structure (namely $\fC$) which is homomorphically equivalent to $(\{0, 1\}; \NAE)$.
\end{proof}

The consistent reduced atomic $\ra{27}{65}$-networks are exactly those in which $c$ only appears in rainbow triangles, i.e., in triples $(x, y, z)$ such that the three edge labels are $\{a, b, c\}$.

Let $\R = (R; E^\R)$ be the Rado graph. The algebra $\ra{27}{65}$ has a normal representation $\fB$ with domain $R \times \{0, 1\}$ and the following relations (see Figure \ref{fig:27_65_rep}):
\begin{align*}
    &a^\fB =
    \bigcup_{(x, y) \in E^{\R}} \left\{\bigl((x, 0), (y, 0)\bigr),
    \bigl((x, 1), (y, 1)\bigr)\right\}
    \cup
    \bigcup_{(x, y) \not\in E^{\R},~ x \neq y} \left\{\bigl((x, 0), (y, 1)\bigr),
    \bigl((x, 1), (y, 0)\bigr)\right\}, \\
    &b^\fB = \bigcup_{(x, y) \not\in E^{\R},~ x \neq y} \left\{\bigl((x, 0), (y, 0)\bigr),
    \bigl((x, 1), (y, 1)\bigr)\right\}
    \cup
    \bigcup_{(x, y) \in E^{\R}} \left\{\bigl((x, 0), (y, 1)\bigr),
    \bigl((x, 1), (y, 0)\bigr)\right\}, \\
    &c^\fB = \left\{\bigl((x, 0), (x, 1)\bigr), \bigl((x, 1), (x, 0)\bigr) \mid x \in \R\right\}.
\end{align*}
\begin{figure}
    \centering
    \[
    \begin{tikzcd}
        (x,0) && (y,0) & \cdots & (x', 0) && (y', 0) \\
        && & \ddots & &&\\
        (x,1) && (y,1) & \cdots & (x', 1) && (y', 1)
        \arrow["c", swap, no head, from=1-1, to=3-1]
        \arrow["c", no head, from=1-3, to=3-3]
        \arrow["c", swap, no head, from=1-5, to=3-5]
        \arrow["c", no head, from=1-7, to=3-7]
        \arrow["a"{pos=0.6}, no head, from=1-5, to=3-7]
        \arrow["a"{pos=0.6}, no head, from=3-5, to=1-7]
        \arrow["a", no head, from=1-1, to=1-3]
        \arrow["a", swap, no head, from=3-1, to=3-3]
        \arrow["b"{pos=0.6}, no head, from=1-1, to=3-3]
        \arrow["b"{pos=0.6}, no head, from=3-1, to=1-3]
        \arrow["b", no head, from=1-5, to=1-7]
        \arrow["b", swap, no head, from=3-5, to=3-7]
    \end{tikzcd}
    \]\caption{Detail from the normal representation $\fB$ of $\ra{27}{65}$. Here, $(x, y) \in E^{\R}$ and $(x', y') \not\in E^{\R}$.}
    \label{fig:27_65_rep}
\end{figure}

\begin{prop}
    \label{prop:27_65}
    The normal representation $\fB$ of $\ra{27}{65}$ pp-constructs
    $(\{0,1\};\NAE)$; consequently, $\NSP(\ra{27}{65})$ is $\NP$-complete.
\end{prop}
\begin{proof}
\begin{figure}
    \centering
    \pgfmathsetlengthmacro{\trianglecolsep}{3.2cm}
    \pgfmathsetlengthmacro{\trianglerowsep}{\trianglecolsep*sqrt(3)/3}
    \[
    \begin{tikzcd}[column sep={\trianglecolsep, between origins}, row sep={\trianglerowsep, between origins}]
        & {\mathstrut z} & \\
        & & \\
        & {\mathstrut w} & \\
        {\mathstrut x} && {\mathstrut y}
        \arrow["c \cup \id", swap, no head, from=1-2, to=4-1]
        \arrow["c \cup \id", no head, from=1-2, to=4-3]
        \arrow["c \cup \id", swap, no head, from=4-1, to=4-3]
        \arrow["a \cup \id"{description, fill=white, inner sep=.8pt}, no head, from=4-1, to=3-2]
        \arrow["b \cup \id"{description, fill=white, inner sep=.8pt}, no head, from=4-3, to=3-2]
        \arrow["a \cup b \cup c"{description, fill=white, inner sep=.8pt}, no head, from=1-2, to=3-2]
    \end{tikzcd}
    \]
    \caption{The relation $T(x,y,z)$ in the normal representation $\fB$ of $\ra{27}{65}$.}
    \label{fig:gadget_27_65}
\end{figure}
The proof is similar to that of Proposition~\ref{prop:19_37}. Let $\fB$ be the normal representation of $\ra{27}{65}$ as defined above.
Since $c \circ c = \id$, we get that $c^\fB$ is a matching. Since $\id \leq c \circ c$, every vertex is incident to a $c^\fB$-edge, which implies that $c^\fB$ is even a perfect matching.

We give a pp-construction of $(\{0, 1\}; \NAE)$ in $\fB$.
Define the ternary relation $T$ by
\begin{align}\label{eq:27_65}
\begin{split}
    T = \bigl\{(x,y,z) \in B^3 \mid ~&(c \cup \id)(x, y) \,\land\, (c \cup \id)(x, z) \,\land\, (c \cup \id)(y, z)\\ &\land\, \exists w\, \bigl((a \cup \id)(x, w) \,\land\, (b \cup \id)(y, w) \,\land\, (a \cup b \cup c)(z, w)\bigr)\bigr\}
\end{split}
\end{align}
(see Figure~\ref{fig:gadget_27_65}).

It can be seen by case distinctions that if $T(x,y,z)$ holds for some $x,y,z \in B$, then exactly two of $x, y, z$ are equal. Conversely, let $x,y,z$ belong to the same equivalence class of $(c \cup \id)^\fB$ and suppose that they are not all equal.
If $x=y$, then $z\neq x$, and we may choose $w=x$.
If $x\neq y$, then $c^\fB(x,y)$. Since
$\id \leq a\circ a$, there exists $w\in B$ such that
$a^\fB(x,w)$. Since $c\circ a=b$, we obtain $b^\fB(y,w)$.
Moreover, $w$ is distinct from both $x$ and $y$, and hence also from
$z\in\{x,y\}$. Thus, in both cases the constraints in
\eqref{eq:27_65} are satisfied.

By what we observed above, we get
\begin{align*}
    T \cap \{x, y\}^3 = \{(x,x,y), (x,y,x), (y,x,x), (y,y,x), (y,x,y), (x,y,y)\}
\end{align*}
for every pair $(x, y) \in B^2$ with $c^\fB(x, y)$; i.e., $(\{x, y\}; T\,\cap\,\{x, y\}^3)$ is isomorphic to $(\{0, 1\}; \NAE)$.
Since $c^\fB$ is a perfect matching, this implies that $\fC := (B; T)$ is isomorphic to a disjoint union of copies of $(\{0, 1\}; \NAE)$. Hence, $\fB$ pp-interprets a structure (namely $\fC$) which is homomorphically equivalent to $(\{0, 1\}; \NAE)$.
\end{proof}

The consistent atomic $\ra{33}{37}$-networks can be seen as oriented graphs that do not embed a directed 3-cycle. Hence, $\ra{33}{37}$ has the flexible atom $a$ and therefore a normal representation. The classification from~\cite{BodirskyKnaeJAIR} cannot be applied here since only symmetric relation algebras are covered there.

\begin{prop}\label{prop:33_37}
    There is a primitive positive interpretation of $(\{0,1\};\NAE)$ in the normal representation $\fB$ of $\ra{33}{37}$; consequently,
    $\NSP(\ra{33}{37})$ is $\NP$-complete.
\end{prop}
\begin{proof}
We present a primitive positive interpretation in $\fB$ of the structure $(\{0, 1\}; \mathrm{NAE})$. It then follows from Proposition \ref{prop:pp-interpretation-implies-logspace-reducible} that $\NSP(\ra{33}{37})$ is $\NP$-hard; it is contained in $\NP$ since $\ra{33}{37}$ has a normal representation.

We work with the terminology from Definition \ref{def:pp-interpretation}. The dimension of the interpretation is $d = 2$. The domain formula is given by $\varphi_\top(x_0, x_1) := (r \cup \breve{r})(x_0, x_1)$. The coordinate map $h\colon (r \cup \breve{r})^\fB \rightarrow \{0, 1\}$ is given by $h(x_0, x_1) = 0$ if $r^\fB(x_0, x_1)$ and $h(x_0, x_1) = 1$ if $(\breve{r})^\fB(x_0, x_1)$.
For the interpretation of the relations $=$ and $\mathrm{NAE}$, consider the gadget network in Figure \ref{fig:gadget_33_37}, which allows us to encode equal orientation of two disjoint $(r \cup \breve{r})$-edges.

\usetikzlibrary{positioning,arrows}
\begin{figure}
\centering
\begin{tikzpicture}
    \node (X0) at ({150}:3.5) {$x_0$};
    \node (X1) at ({210}:3.5) {$x_1$};
    \node (ZXY0) at ({90}:3.5) {$u_0$};
    \node (ZXY1) at ({90}:1) {$u_1$};
    \node (ZXY2) at ({270}:1) {$u_2$};
    \node (ZXY3) at ({270}:3.5) {$u_3$};
    \node (Y0) at ({30}:3.5) {$y_0$};
    \node (Y1) at ({330}:3.5) {$y_1$};
     \draw (X0) -- (X1) node[midway, above] {};
     \draw[-{Straight Barb[angle'=30,scale=3]}] (X0) -- (ZXY0) node[midway, above] {};
     \draw (ZXY0) -- (ZXY1) node[midway, above] {};
     \draw (Y0) -- (Y1) node[midway, above] {};
     \draw[-{Straight Barb[angle'=30,scale=3]}] (ZXY0) -- (Y0) node[midway, above] {};
     \draw[-{Straight Barb[angle'=30,scale=3]}]  (Y1) -- (ZXY1) node[midway, above] {};
     \draw (ZXY1) -- (Y0) node[midway, above] {};
     \draw[-{Straight Barb[angle'=30,scale=3]}] (ZXY1) -- (X1) node[midway, above] {};
     \draw (ZXY0) -- (X1) node[midway, above] {};
     \draw[-{Straight Barb[angle'=30,scale=3]}] (ZXY2) -- (X0) node[midway, above] {};
     \draw (ZXY2) -- (X1) node[midway, above] {};
     \draw[-{Straight Barb[angle'=30,scale=3]}] (X1) -- (ZXY3) node[midway, above] {};
     \draw (ZXY2) -- (ZXY3) node[midway, above] {};
     \draw[-{Straight Barb[angle'=30,scale=3]}] (Y0) -- (ZXY2) node[midway, above] {};
     \draw[-{Straight Barb[angle'=30,scale=3]}] (ZXY3) -- (Y1) node[midway, above] {};
     \draw (ZXY3) -- (Y0) node[midway, above] {};
\end{tikzpicture}
    \caption{In this $\ra{33}{37}$-network, a directed edge stands for $r$, and an undirected edge stands for $(r \cup \breve{r})$. In every solution $g$ to it, we have $g(x_0, x_1) = g(y_0, y_1)$.}
    \label{fig:gadget_33_37}
\end{figure}

Let $E(x_0, x_1, y_0, y_1)$ be the formula $\exists u_0 \exists u_1 \exists u_2 \exists u_3~ \psi$, where $\psi$ is the conjunction of all the relations depicted in Figure \ref{fig:gadget_33_37}. Then our defining formula for equality is given by
\begin{align*}
    \varphi_=(x_0, x_1, y_0, y_1) := \exists z_0, z_1 \bigl(E(x_0, x_1, z_0, z_1) \land E(y_0, y_1,z_0, z_1)
    \bigr).
\end{align*}
($\varphi_=$ can obviously be transformed into an equivalent pp-formula by renaming the existentially bound variables in the occurrences of $E$ and subsequently shifting them to the front.)
By the observation about the gadget network, it is clear that $\varphi_=(x_0, x_1, y_0, y_1)$ holds if and only if $h(x_0, x_1) = h(y_0, y_1)$. The auxiliary variables $z_0$ and $z_1$ are needed since $\{x_0, x_1\}$ and $\{y_0, y_1\}$ are not necessarily disjoint.

The defining formula of $\mathrm{NAE}$ is given by
\begin{align*}
    \varphi_{\mathrm{NAE}}(x_0, x_1, y_0, y_1, z_0, z_1) := \exists v_0, v_1, v_2, q_0, \hdots, q_5
    \bigl( E(x_0, x_1, q_0, q_1) \land E(y_0, y_1, q_2, q_3)\\\land\,
    E(z_0, z_1, q_4, q_5) \land E(q_0, q_1, v_0, v_1) \land E(q_2, q_3, v_1, v_2) \land E(q_4, q_5, v_2, v_0)\bigr).
\end{align*}
The auxiliary variables $q_0, \hdots, q_5$ are needed since some of $x_0, x_1, y_0, y_1, z_0, z_1$ might be equal. Now $\varphi_{\mathrm{NAE}}(x_0, x_1, y_0, y_1, z_0, z_1)$ holds in $\fB$ if and only if there are $v_0, v_1, v_2 \in \fB$ such that $(x_0, x_1)$ has the same orientation as $(v_0, v_1)$, $(y_0, y_1)$ has the same orientation as $(v_1, v_2)$, and $(z_0, z_1)$ has the same orientation as $(v_2, v_0)$. Since the only forbidden triple of $\ra{33}{37}$ is the directed $3$-cycle, it is clear that this in turn holds if and only if $h(x_0, x_1)$, $h(y_0, y_1)$, and $h(z_0, z_1)$ are not all equal.
\end{proof}

Let $\mathbb H = (V; E)$ be the \emph{Henson graph}, i.e., the countable homogeneous graph  whose age is the class of all finite undirected triangle-free graphs.
We obtain a (normal) representation $\fB = \Hgraph^{b,a}$ of $\ra{6}{7}$:
we set $b^{\fB} := E$, $a^{\fB} := \{(u,v) \in V^2 \setminus E \mid u \neq v\}$, $\id^\fB := \{(u,v) \in V^2 \mid u=v\}$, and extend to all other relations by taking the respective Boolean combinations.
It is known that $\NSP(\ra{6}{7})$ is $\NPc$~\cite{BodirskyKnaeRamics}.

\begin{prop}\label{prop:31_37}
    The normal representations of $\ra{31}{37}$ and $\ra{52}{65}$ pp-construct $(\{0, 1\}; \NAE)$; consequently,
    $\NSP(\ra{31}{37})$ and $\NSP(\ra{52}{65})$ are $\NP$-complete.
\end{prop}
\begin{proof}
    Let $\fB$ be the normal representation of $\ra{31}{37}$. Consider the reduct $\mathfrak H$ of $\fB$ to all unions of $\id^\fB$, $a^\fB$, and $(r\cup\breve r)^\fB$, including the empty union, with relation symbols renamed according to $\id\mapsto\id$, $a\mapsto a$, and $(r\cup\breve r)\mapsto b$.
    Clearly, $\mathfrak H$ and $\mathbb H^{b,a}$ have the same age. Hence, $\CSP(\mathfrak{H}) = \CSP(\mathbb H^{b,a})$, where $\mathbb H^{b,a}$ is the relation-algebra expansion of the Henson graph defined before. Moreover, $\mathfrak{H}$ is $\omega$-categorical since it is a first-order reduct of the $\omega$-categorical structure $\fB$. It follows from Observation~\ref{obs:hom-equiv} that $\mathfrak{H}$ is homomorphically equivalent to $\mathbb H^{b,a}$. Since $\mathbb H^{b,a}$ pp-constructs $(\{0, 1\}; \NAE)$, it follows that $\mathfrak{H}$ pp-constructs $(\{0, 1\}; \NAE)$ by the transitivity of pp-constructibility. Since $\fB$ pp-defines $\mathfrak{H}$, we obtain that $\fB$ pp-constructs $(\{0, 1\}; \NAE)$.
    
   For $\ra{52}{65}$, take its normal representation $\fC$ and consider the reduct to all unions of $\id^\fC$, $(a\cup b)^\fC$, and $c^\fC$, including the empty union. Rename these relations according to $\id\mapsto\id$, $(a\cup b)\mapsto a$, and $c\mapsto b$. The same argument applies.
\end{proof}

\subsection{Hardness via Gadget Reductions}\label{sect:gadget}
Algebras considered in this section do not have a normal representation and hence do not fall within the scope of the tractability conjecture (Conjecture \ref{conj:tractability}). We will show the $\NP$-hardness of their NSPs by many-one
reductions from computational problems which are known to be $\NP$-hard. We start with the algebra $\ra{30}{65}$.

\begin{prop}\label{prop:30_65}
    $\NSP(\ra{30}{65})$ is $\NP$-complete.
\end{prop}
\begin{proof}
    Recall that $\ra{30}{65}$ has a fully universal representation (Proposition \ref{prop:30_65-31_65-repr}), so $\NSP(\ra{30}{65}) = \NCP(\ra{30}{65}) \in \NP$.
    We reduce $\NCP(\ra{6}{7})$ to $\NCP(\ra{30}{65})$. Let $N$ be a $\ra{6}{7}$-network, and let $N'$ be the network obtained from $N$ by replacing $1$ by $\id \cup\,a \cup b$.
    Then $N'$ can be interpreted as a $\ra{30}{65}$-network, and it is satisfiable if and only if $N$ is satisfiable, since the allowed triples of $\ra{6}{7}$ and those of $\ra{30}{65}$ restricted to $\{\id, a, b\}$ are identical.
\end{proof}

We have already shown that $\ra{30}{37}$ has a fully universal representation (Proposition~\ref{prop:30_37-repr}). Recall that the reduced consistent atomic networks of $\ra{30}{37}$ can be viewed as oriented graphs that do not embed the oriented graph $(\{0,1,2\}; \{(0,1),(0,2)\})$.

\usetikzlibrary{positioning,arrows}
\begin{figure}
\centering
\begin{tikzpicture}
    \node (X0) at ({150}:3.5) {$x_0$};
    \node (X1) at ({210}:3.5) {$x_1$};
    \node (ZXY0) at ({90}:3.5) {$u_0$};
    \node (ZXY1) at ({90}:1) {$u_1$};
    \node (ZXY2) at ({270}:1) {$u_2$};
    \node (ZXY3) at ({270}:3.5) {$u_3$};
    \node (Y0) at ({30}:3.5) {$y_0$};
    \node (Y1) at ({330}:3.5) {$y_1$};
     \draw (X0) -- (X1) node[midway, above] {$r \cup \breve{r}$\quad\quad\quad\quad};
     \draw[-{Straight Barb[angle'=30,scale=3]}] (ZXY0) -- (X0)
        node[midway, above] {$a \cup r\quad\quad$};
     \draw (ZXY0) -- (ZXY1)
        node[midway, above, fill=white, inner sep=1.5pt] {$r \cup \breve{r}$};
     \draw (Y0) -- (Y1)
        node[midway, above] {\quad\quad\quad$r \cup \breve{r}$};
     \draw[-{Straight Barb[angle'=30,scale=3]}] (Y0) -- (ZXY0)
        node[midway, above] {\quad\quad$a \cup r$};
     \draw (Y1) -- (ZXY0)
        node[midway, fill=white, inner sep=1.5pt] {$a$};
     \draw (ZXY1) -- (X0)
        node[midway, fill=white, inner sep=1.5pt] {$a$};
     \draw[-{Straight Barb[angle'=30,scale=3]}] (X1) -- (ZXY0)
        node[midway, fill=white, inner sep=1.5pt] {$r$};
     \draw[-{Straight Barb[angle'=30,scale=3]}] (ZXY1) -- (Y0)
        node[midway, fill=white, inner sep=1.5pt] {$r$};
     \draw[-{Straight Barb[angle'=30,scale=3]}] (X0) -- (ZXY3)
        node[midway, fill=white, inner sep=1.5pt] {$r$};
     \draw (ZXY2) -- (ZXY3)
        node[midway, above, fill=white, inner sep=1.5pt] {$r \cup \breve{r}$};
     \draw[-{Straight Barb[angle'=30,scale=3]}] (Y1) -- (ZXY3)
        node[midway, below] {\quad\quad$a \cup r$};
     \draw[-{Straight Barb[angle'=30,scale=3]}] (ZXY3) -- (X1)
        node[midway, below] {$a \cup r$\quad\quad\quad};
     \draw[-{Straight Barb[angle'=30,scale=3]}] (ZXY2) -- (Y1)
        node[midway, fill=white, inner sep=1.5pt] {$r$};
     \draw (ZXY2) -- (X1)
        node[midway, fill=white, inner sep=1.5pt] {$a$};
     \draw (ZXY3) -- (Y0)
        node[midway, fill=white, inner sep=1.5pt] {$a$};
\end{tikzpicture}
    \caption{The $\ra{30}{37}$-network $E(x_0, x_1, y_0, y_1)$ encodes equal orientation of two $(r \cup \breve{r})$-edges: In every solution $g$ to it, we have $g(x_0, x_1) = g(y_0, y_1)$.}
    \label{fig:gadget-equality-30_37}
\end{figure}

\begin{figure}
    \centering
    \begin{tikzcd}
        p_1 && p_2 && p_3 && p_4 \\
        \\
        & q_1 && q_2 && q_3
        \arrow["a \cup \id", no head, from=1-1, to=1-3]
        \arrow["a \cup \id", no head, from=1-3, to=1-5]
        \arrow["a \cup \id", no head, from=1-5, to=1-7]
        \arrow["r \cup \breve{r}", very thick, swap, no head, from=1-1, to=3-2]
        \arrow["r", swap, from=3-2, to=1-3]
        \arrow["r \cup \breve{r}", very thick, swap, no head, from=1-3, to=3-4]
        \arrow["r", swap, from=3-4, to=1-5]
        \arrow["r \cup \breve{r}", very thick, swap, no head, from=1-5, to=3-6]
        \arrow["r", swap, from=3-6, to=1-7]
        \arrow["a", bend left=40, no head, from=1-1, to=1-7]
    \end{tikzcd}
    \caption{The $\ra{30}{37}$-network $S(p_1, p_2, p_3, p_4, q_1, q_2, q_3)\colon$ In every solution to it, at least one of the three thick edges $(p_1, q_1), (p_2, q_2)$, and $(p_3, q_3)$ is labeled by $r$.}
    \label{fig:30_37_R}
\end{figure}
\begin{prop}\label{prop:30_37}
    $\NSP(\ra{30}{37})$ is $\NP$-complete.
\end{prop}
\begin{proof}
    Since $\ra{30}{37}$ has a fully universal representation, $\NSP(\ra{30}{37}) = \NCP(\ra{30}{37})$. We give a polynomial-time many-one reduction from $\CSP(\{0,1\}; \NAE)$ to $\NSP(\ra{30}{37})$. Let $\varphi$ be an arbitrary instance of it; we construct a network $N = (V, f)$ such that $N$ is a satisfiable $\ra{30}{37}$-network if and only if $\varphi$ is a positive instance of $\CSP(\{0,1\}; \NAE)$.
    
    Consider the gadget networks $E(x_0, x_1, y_0, y_1)$ and $S(p_1, p_2, p_3, p_4, q_1, q_2, q_3)$, depicted in Figure~\ref{fig:gadget-equality-30_37} and Figure~\ref{fig:30_37_R}, respectively.

    For every variable $x$ of $\varphi$, $V$ contains two variables $x_0$ and $x_1$.
    For every clause $\NAE(x,y,z)$ of $\varphi$, $N$ contains two copies
    \begin{align*}
        S(p_1, p_2, p_3, p_4, q_1, q_2, q_3) \qquad\text{and}\qquad S(p_1', p_2', p_3', p_4', q_1', q_2', q_3')
    \end{align*}
    of $S$, and copies of the constraints $E(x_0, x_1, p_1, q_1)$, $E(y_0, y_1, p_2, q_2)$, $E(z_0, z_1, p_3, q_3)$, as well as $E(x_1, x_0, p_1', q_1')$, $E(y_1, y_0, p_2', q_2')$, $E(z_1, z_0, p_3', q_3')$, where the variables $u_0, \ldots, u_3$ are replaced by fresh vertices for every such copy. On all edges $(x,y)$ where $f$ is not yet determined, we set $f(x,y) = 1$.

    It remains to verify the correctness of this reduction.
    First, assume that $N$ has a solution $g$.
    By the observations about the networks $E$ and $S$, it follows that for every clause $\NAE(x,y,z)$ of $\varphi$, at least one of $g(x_0, x_1)$, $g(y_0, y_1)$, and $g(z_0, z_1)$ is equal to $r$. Since we also added a copy of the same constraints for the reversed edges, at least one of $g(x_1, x_0)$, $g(y_1, y_0)$, and $g(z_1, z_0)$ is equal to $r$, i.e., at least one of $g(x_0, x_1)$, $g(y_0, y_1)$, and $g(z_0, z_1)$ is equal to $\breve{r}$. Hence, the function
    \begin{align*}
        s(x) := \begin{cases}
            0 \quad&\text{~if~} g(x_0, x_1) = r,\\
            1 \quad&\text{~if~} g(x_0, x_1) = \breve{r}\\
        \end{cases}
    \end{align*}
    is a solution to $\varphi$.

    For the other direction, we first construct a fixed network $M$ which is independent of $\varphi$. For every $\tau = (\tau_1, \tau_2, \tau_3) \in \NAE = \{0,1\}^3\backslash\{(0,0,0),\allowbreak (1,1,1)\}$, let $N_\tau$ be a fresh copy of the entire $(2S+6E)$-network introduced above for one clause $\NAE(x,y,z)$, i.e., the network consisting of two copies of $S$, six copies of $E$, and three pairs $(x_{0,1}^\tau, x_{1,1}^\tau)$, $(x_{0,2}^\tau, x_{1,2}^\tau)$, and $(x_{0,3}^\tau, x_{1,3}^\tau)$ corresponding to the Boolean variables. Now take the disjoint union $\bigcup_{\tau \in \NAE} N_\tau$ of all these six networks $N_\tau$; on their vertices corresponding to the Boolean variables, we define two equivalence classes $U$ and $V$ as follows: If $\tau_i = 0$, we set $x_{0,i}^\tau \in U$ and $x_{1,i}^\tau \in V$, and if $\tau_i = 1$, we set $x_{0,i}^\tau \in V$ and $x_{1,i}^\tau \in U$. For example, for $\tau = 100$, we set
    \begin{align*}
        x_{0,1}^\tau \in V,~ x_{1,1}^\tau \in U; \qquad
        x_{0,2}^\tau \in U,~ x_{1,2}^\tau \in V; \qquad
        x_{0,3}^\tau \in U,~ x_{1,3}^\tau \in V.
    \end{align*}
    We impose $\id$ on $U^2\cup V^2$ and $r$ on $U\times V$ (and $\breve r$ on $V\times U$). No auxiliary vertices belonging to different copies $N_\tau$ are identified, and we impose $1$ on all edges which are not yet set. Let $Q$ denote the finite network obtained in this way. A computer verification shows that $Q$ has a consistent atomic refinement $M$. An explicit atomic refinement $M$ of $Q$, together with a program
    verifying that $M$ refines $Q$ and is consistent, is provided in~\cite{RA4Code}.

    By the property of $E$, in every consistent atomic refinement of $Q$, the three distinguished edges of the first copy of $S$ in $N_\tau$ are oriented according to $\tau$, while those of the second copy are oriented oppositely.

    Now assume that $\varphi$ has a satisfying assignment $s$.
    Let $\fB$ be a fully universal representation of $\ra{30}{37}$. Since $M$ is a consistent atomic network, it has a realization via some function $f$ in $\fB$, i.e., a function mapping its variables into $B$ that preserves all relations.
    Since $M$ contains $\id$ on $U^2 \cup V^2$, all vertices in $U$ have the same image under $f$, and likewise all vertices in $V$ have the same image. Denote these two elements by $u$ and $v$, respectively. Since $M$ contains $r$ on $U\times V$, we have $r^{\mathfrak B}(u,v)$.

    Now let $C = \NAE(x,y,z)$ be a clause of $\varphi$, and put $\tau=(s(x),s(y),s(z))$.
    Since $s$ satisfies $C$, we have $\tau \in \NAE$. We realize the $(2S+6E)$-gadget belonging to $C$ by placing it “on top” of the already realized copy $N_\tau$ in $M\colon$ for every auxiliary vertex $w$ introduced for $C$, let $w^\tau$ denote the corresponding auxiliary vertex of $N_\tau$ and set $g(w):=f(w^\tau)$. For the Boolean variables, set
    \begin{align*}
        (g(x_0), g(x_1)) = \begin{cases}
            (u, v) \quad&\text{~if~} s(x) = 0,\\
            (v, u) \quad&\text{~if~} s(x) = 1.\\
        \end{cases}
    \end{align*}
    By the definition of $U$ and $V$, the three Boolean pairs of $N_\tau$ are mapped by $f$ to $(u,v)$ or $(v,u)$ according to the three entries of $\tau$, which are precisely $(s(x),s(y),s(z))$. Hence the entire $(2S+6E)$-gadget belonging to $C$ is realized by $g$ exactly as the fixed copy $N_\tau$ is realized under $f$. In particular, all of its constraints are satisfied.

    We do this independently for every clause of $\varphi$. There is no conflict between different clauses: their auxiliary vertices are fresh, while the image of each Boolean vertex $(x_0,x_1)$ depends only on $s(x)$ and is therefore the same in every clause containing $x$. Finally, all pairs not constrained by one of the gadgets are labeled by $1$ in $N$ and hence impose no further condition. Thus $g$ is a solution of $N$.

    Therefore, $N$ is satisfiable if and only if $\varphi$ is satisfiable. Since the reduction is polynomial-time and $\NSP(\ra{30}{37})=\NCP(\ra{30}{37})\in\NP$, the claim follows.
\end{proof}

\begin{figure}
    \centering
    \[
    \begin{tikzcd}
        x_0 && y_0 \\
        \\
        x_1 && y_1
        \arrow["b \cup c", no head, from=1-1, to=1-3]
        \arrow["b \cup c", swap, no head, from=1-1, to=3-1]
        \arrow["b \cup c", no head, from=1-3, to=3-3]
        \arrow["b \cup c"{pos=0.6}, no head, from=3-1, to=1-3]
        \arrow["b \cup c", swap, no head, from=3-1, to=3-3]
    \end{tikzcd}
    \]\caption{The $\ra{31}{65}$-network $E(x_0, x_1, y_0, y_1)\colon$ In every solution $g$ to it, we have $g(x_0, x_1) = g(y_0, y_1)$.}
    \label{fig:gadget-equality-31_65}
\end{figure}

\begin{figure}
    \centering
    \begin{tikzcd}
        p_1 && p_2 && p_3 && p_4 \\
        \\
        & q_1 && q_2 && q_3
        \arrow["c \cup \id", no head, from=1-1, to=1-3]
        \arrow["c \cup \id", no head, from=1-3, to=1-5]
        \arrow["c \cup \id", no head, from=1-5, to=1-7]
        \arrow["b \cup c", very thick, swap, no head, from=1-1, to=3-2]
        \arrow["b \cup c", swap, no head, from=3-2, to=1-3]
        \arrow["b \cup c", very thick, swap, no head, from=1-3, to=3-4]
        \arrow["b \cup c", swap, no head, from=3-4, to=1-5]
        \arrow["b \cup c", very thick, swap, no head, from=1-5, to=3-6]
        \arrow["b \cup c", swap, no head, from=3-6, to=1-7]
        \arrow["c", bend left=30, no head, from=1-1, to=1-7]
        \arrow["a \cup b", swap, no head, from=3-2, to=3-4]
        \arrow["a \cup b", swap, no head, from=3-4, to=3-6]
        \arrow["a \cup b", swap, bend right=30, no head, from=3-2, to=3-6]
    \end{tikzcd}
    \caption{The $\ra{31}{65}$-network $G(p_1, p_2, p_3, p_4, q_1, q_2, q_3)\colon$ In every solution to it, exactly one of the three thick edges $(p_1, q_1), (p_2, q_2)$, and $(p_3, q_3)$ is labeled by $c$.}
    \label{fig:31_65_R}
\end{figure}

Recall that $\ra{31}{65}$ has a fully universal representation (Proposition~\ref{prop:30_65-31_65-repr}) and that the relation $c \cup \id$ is an equivalence relation.
\begin{prop}
\label{prop:31_65}
    $\NSP(\ra{31}{65})$ is $\NP$-complete.
\end{prop}
\begin{proof}
Since $\ra{31}{65}$ has a fully universal representation, $\NSP(\ra{31}{65}) = \NCP(\ra{31}{65})$ is contained in $\NP$. We show its $\NP$-hardness by a reduction from
positive 1-in-3 SAT, which is the CSP of the structure $(\{0,1\}; \{(0,0,1), (0,1,0), (1,0,0)\})$.

Let $\varphi$ be an instance of this problem and let $V$ be the set of variables occurring in $\varphi$. We will construct an instance $(V^\prime, f)$ of $\NCP(\ra{31}{65})$ such that $(V^\prime, f)$ is satisfiable if and only if $\varphi$ is satisfiable.

Consider the $\ra{31}{65}$-networks $E(x_0, x_1, y_0, y_1)$ and $G(p_1, p_2, p_3, p_4, q_1, q_2, q_3)$ which are depicted in Figure \ref{fig:gadget-equality-31_65} and Figure \ref{fig:31_65_R}, respectively.

Let $x \in V$ be such that $x$ occurs $d$ times in $\varphi$.
We add $2d$ fresh variables $x_0^1, x_1^1, \hdots, x_0^d, x_1^d$ to $V'$
and impose the constraints of $E(x_0^i, x_1^i, x_0^j, x_1^j)$ for all distinct $i, j \in \{1, \hdots, d\}$. For every clause $(x,y,z)$ of $\varphi$, we impose the constraints of 
$E(x_0^i, x_1^i, p_1, q_1)$,
$E(y_0^j, y_1^j, p_2, q_2)$,
$E(z_0^k, z_1^k, p_3, q_3)$, and
$G(p_1, p_2, p_3, p_4, q_1, q_2, q_3)$ (where $x$ appears in $(x,y,z)$ for the $i$-th time in $\varphi$, $y$ appears for the $j$-th time and $z$ appears for the $k$-th time).

It follows immediately from the observations about the gadget networks above that if a function $g$ is a solution to $(V', f)$, then the function $s\colon V \rightarrow \{0, 1\}$ defined by $s(x) = 0$ if $g(x_0^1, x_1^1) = b$ and $s(x) = 1$ if $g(x_0^1, x_1^1) = c$ is a solution to $\varphi$. Conversely, if $s\colon V \rightarrow \{0, 1\}$ is a solution to $\varphi$, define the partial function $g$ by $g(x_0^j, x_1^j) = b$ if $s(x) = 0$ and $g(x_0^j, x_1^j) = c$ if $s(x) = 1$ (for $j \in \{1, \hdots, d\}$ if $x$ appears $d$ times in $\varphi$). A direct case analysis shows that $g$ can be extended to a solution of the entire network $(V', f)$.
\end{proof}

\subsection{Hardness from Promise Constraint Satisfaction Problems}
\label{sect:PCSPs}
In this section, we introduce a new technique for showing $\NP$-hardness of the NSP based on a recent hardness result from the area of promise constraint satisfaction problems~\cite{PCSP}.

\begin{prop}\label{prop:pcsp-trick}
    Let $\bA$ be a representable finite relation algebra with two symmetric atoms $p, q$ such that $(p, p, p)$ and $(q, q, q)$ are forbidden and $(p, q, q)$ is an allowed triple. Then $\NSP(\bA)$ is $\NP$-hard.
\end{prop}
\begin{proof}
    We will reduce PCSP($K_3, K_5$), which is known to be $\NP$-hard~\cite[Theorem 6.5]{PCSP}, to $\NSP(\bA)$. Its instances are undirected graphs which are either 3-colorable or not even 5-colorable, and the problem is to decide which of these two cases apply. So let $G = (V; E)$ be an instance of this problem. We construct the following $\bA$-network $(V, f)\colon$
    \begin{itemize}
        \item For every $x \in V$, we set $f(x, x) = \id$.
        \item For every edge $(x, y) \in E$, we set $f(x, y) = p \cup q$.
        \item For all remaining pairs of distinct vertices $(x, y)$, we set $f(x, y) = p \cup q \cup \id$.
    \end{itemize}
    Let $\fB$ be a representation of $\bA$. Since $(p, q, q)$ is an allowed triple, there are elements $w_1, w_2, w_3 \in \fB$ such that $(w_1, w_2) \in p^\fB, (w_2, w_3) \in q^\fB$ and $(w_3, w_1) \in q^\fB$.
    If $G$ is 3-colorable, then there is a mapping $c\colon V \rightarrow \{1, 2, 3\}$ witnessing its 3-colorability. It is easy to see that $s\colon x \mapsto w_{c(x)}$ satisfies $(V, f)$ in $\fB$.
    If $G$ is not even 5-colorable, then the image of any solution to $(V, f)$ has at least six elements. Since $\mathrm{R}(3, 3) = 6$, this solution has to contain either the triple $(p, p, p)$ or the triple $(q, q, q)$, which is a contradiction, so $(V, f)$ is not satisfiable in this case.
\end{proof}

Recall that both $\ra{51}{65}$ and $\ra{56}{65}$ are representable, but do not have a fully universal representation (Sections \ref{sect:51_65} and \ref{sect:56_65}).

\begin{cor}\label{cor:56_65-and-others-np-hard}
    $\NSP(\ra{51}{65})$ and $\NSP(\ra{56}{65})$ are $\NP$-complete.
\end{cor}

\begin{proof}
    Both $\ra{51}{65}$ and $\ra{56}{65}$ are representable (Proposition~\ref{prop:51rep} and Proposition~\ref{prop:56rep}) and the two atoms $p := a$ and $q := c$ satisfy the condition of Proposition~\ref{prop:pcsp-trick}. It follows that both problems are $\NP$-hard.

    By Proposition~\ref{prop:51rep}, the algebra $\ra{51}{65}$ has a finitely bounded universal representation and thus, by Lemma~\ref{lem:fin_bounded_np}, $\NSP(\ra{51}{65})$ is in $\NP$.

    By~\cite[Theorem 1.4]{56paper}, $\ra{56}{65}$ has a finitely bounded universal square  representation, so Lemma~\ref{lem:fin_bounded_np} also yields $\NSP(\ra{56}{65})\in\NP$.
\end{proof}

\begin{rem}
Proposition~\ref{prop:pcsp-trick} also shows that the
NSP for each of the algebras $\ra{1}{7}$, $\ra{5}{7}$, $\ra{1}{65}$, $\ra{2}{65}$, $\ra{5}{65}$, $\ra{10}{65}$, $\ra{12}{65}$, $\ra{15}{65}$, $\ra{16}{65}$, $\ra{17}{65}$, $\ra{39}{65}$, $\ra{55}{65}$, $\ra{62}{65}$, and $\ra{63}{65}$ is $\NP$-hard
(which was already established earlier), since each of them is representable and has two atoms satisfying the condition of Proposition~\ref{prop:pcsp-trick}.
\end{rem}

\section{Tractable Algebras with at most 4 Atoms}
\label{sect:poly-time-tractable-ras}
In this section we present polynomial-time algorithms for the network satisfaction problem for all the remaining integral relation algebras with at most four atoms.

\subsection{Tractability from the Atom Structure}
\label{subsec:siggers}
A polynomial-time reduction of $\NSP(\bA)$ to the CSP of the atom structure (Definition~\ref{def:atom-struct}) has been presented in Proposition~\ref{prop:reductionAtomStructure}; we can then use the complexity classification of the CSPs of finite conservative structures (Theorem~\ref{thm:dichotomyConservative}).
In this section we apply this technique to concrete relation algebras; we proceed in the order of increasing difficulty.
The polynomial-time tractability result from Theorem~\ref{thm:dichotomyConservative} will be used to study the NSP mostly via the following lemma.

\begin{lem}\label{lem:2or3-sym}
    Let $\bA$ be a relation algebra with a fully universal representation and atom structure $\fAo$. If $\fA_0$ has a polymorphism $f$
    which is binary symmetric or a ternary weak near unanimity operation, then $\NSP(\bA)$ is in $\Ptime$.
\end{lem}

\begin{proof}
    Since $\fAo$ is conservative, every ternary WNU polymorphism $f$ restricted to two elements is either a majority operation, a minority operation, or its minor $g(x, y) := f(x,x,y)$ is a binary symmetric operation. Hence, Theorem~\ref{thm:dichotomyConservative}
    implies that $\CSP(\fAo)$ is in $\Ptime$. Proposition~\ref{prop:reductionAtomStructure} implies that
    $\NSP(\bA)$ is in $\Ptime$ as well.
\end{proof}

\begin{figure}
\begin{center}
    \begin{tabular}{|R|LL|}\hline
         \ra{2}{2} & \id & a \\ \hline
         \id & \id & a \\
         a & a & a \\ \hline
    \end{tabular}\quad
    \begin{tabular}{|R|LLL|}\hline
         \ra{4}{7},\ra{7}{7} & \id & a & b \\ \hline
         \id & \id & a & b \\
         a & a & a & b \\
         b & b & b & b \\ \hline
    \end{tabular}\\[1em]
    \begin{tabular}{|R|LLLL|}\hline
         \ra{8}{65}, \ra{14}{65} & \id & a & b & c\\ \hline
         \id & \id & a & b & c \\
         a & a & a & b & c \\
         b & b & b & b & c \\
         c & c & c & c & c \\ \hline
    \end{tabular}\quad
    \begin{tabular}{|R|LLLL|}\hline
         \ra{20}{65}, \ra{53}{65}, \ra{61}{65}, \ra{65}{65} & \id & a & b & c\\ \hline
         \id & \id & a & b & c \\
         a & a & a & a & a \\
         b & b & a & b & c \\
         c & c & a & c & c \\ \hline
    \end{tabular}
\end{center}
\caption{Binary symmetric polymorphisms of the atom structures for the relation algebras
$\ra{2}{2}, \ra{4}{7}, \ra{7}{7}, \ra{8}{65}, \ra{14}{65}, \ra{20}{65}, \ra{53}{65}, \ra{61}{65}$, and $\ra{65}{65}$.}
\label{fig:pols}
\end{figure}

\begin{prop}
\label{prop:can-bin-sym}
    Let $\bA \in \{\ra{2}{2}, \ra{4}{7}, \ra{7}{7}, \ra{8}{65}, \ra{14}{65}, \ra{20}{65}, \ra{53}{65}, \ra{61}{65}, \ra{65}{65}\}$. Then $\NSP(\bA)$ is in $\Ptime$.
\end{prop}
\begin{proof}
    In Figure~\ref{fig:pols} we present binary symmetric polymorphisms of
    $\fA_0$. To check this we used a computer program, provided in~\cite{RA4Code}.
    The statement therefore follows from Lemma~\ref{lem:2or3-sym}.
\end{proof}

\begin{prop}
\label{prop:1_2-in-P}
    $\NSP(\ra{1}{2})$ is in $\Ptime$.
\end{prop}
\begin{proof}
    Observe that $A_0 = \{a, \id\}$ has only two elements, hence there exists a (unique) conservative ternary minority operation $f$ on $A_0$.
    It is immediate that $f$ preserves the allowed triples of $\ra{1}{2}$; hence, it is a weak near unanimity polymorphism and the statement follows from Lemma~\ref{lem:2or3-sym}.
\end{proof}

The relation algebras $\ra{3}{3}$ and $\ra{37}{37}$ are trivial in the sense that they do not have any forbidden triples of non-identity atoms; all of their non-identity atoms are flexible. It follows that all atomic $\ra{3}{3}$- and $\ra{37}{37}$-networks which do not contain a triple of the form $(\id, \id, x)$ or $(\id, x, y)$ for $x, y \in A_0\backslash \{\id\}$ and $x \neq y$ are satisfiable. This can be checked in polynomial-time, e.g., by the \hyperref[alg:pc]{Path Consistency Algorithm}. For completeness, we nevertheless give ternary weak near unanimity polymorphisms of their respective atom structures:
\begin{prop}
\label{prop:can-bin-sym2}
    Let $\bA \in \{\ra{3}{3}, \ra{37}{37}\}$. Then $\NSP(\bA)$ is in $\Ptime$.
\end{prop}
\begin{proof}
    Let $f \colon A_0^3 \to A_0$ be the operation\footnote{For $\ra{3}{3}$, this construction is from~\cite[Proposition 5.10]{KnaeuerMaster}.}
defined by
    \begin{itemize}
        \item $f(\id, \id, x) = f(\id, x, \id) = f(x, \id, \id) = x$ for all $x \in A_0$,
        \item $f(x, x, y) = f(x, y, x) = f(y, x, x) = x$ for all $x, y \in A_0\backslash\{\id\}$,
        \item $f(\id, x, y) = f(y, \id, x) = f(x, y, \id) = x$ for all $x, y \in A_0\backslash \{\id\}$ and,
        \item in the case of $\ra{37}{37}$, $f(r, \breve{r}, a) = f(a, r, \breve{r}) = f(\breve{r}, a, r) = f(\breve{r}, r, a) = f(r, a, \breve{r}) = f(a, \breve{r}, r) = a$.
    \end{itemize}
    It is then easy to verify that $f$ preserves all the relations of the atom structure; by definition, it is a weak near unanimity operation and hence the statement follows from Lemma~\ref{lem:2or3-sym}.
\end{proof}

The normal representation $\fB$ of $\ra{3}{7}$ has the property
that $a^\fB \cup \id^\fB$ is an equivalence relation with infinitely many classes of size two.

\begin{prop}[\cite{BMPP16}]\label{prop:3_7}
    Let $\fB$ be the normal representation of $\ra{3}{7}$.
    Then
    the atom structure $\fA_0$ has a ternary polymorphism
    $h$ which satisfies
    $h(b, \cdot, \cdot) = h(\cdot, b, \cdot) = h(\cdot, \cdot, b) = b$ and which satisfies $h(x,x,y)=h(x,y,x)=h(y,x,x) = y$ for all $x,y \in \{a,\id\}$.
\end{prop}

\begin{cor}\label{cor:3_7}
    $\NSP(\ra{3}{7})$ is in $\Ptime$.
\end{cor}
\begin{proof}
    Note that the operation $h$ specified in Proposition~\ref{prop:3_7} satisfies the conditions from Lemma~\ref{lem:2or3-sym} which shows that $\NSP(\ra{3}{7})$ is in $\Ptime$.
\end{proof}

The normal representation $\fB$ of $\ra{7}{65}$ has the property that $a^\fB \cup \id^\fB$ is an equivalence relation with infinitely many classes of size two.
Moreover, $a^{\fB} \cup b^{\fB} \cup \id^\fB$ is an
equivalence relation with infinitely many classes of infinite size.

\begin{prop}\label{prop:7_65}
    $\NSP(\ra{7}{65})$ is in $\Ptime$.
\end{prop}
\begin{proof}
Let $h$ be the ternary operation on the atom structure $\mathfrak A_0$ of $\ra{7}{65}$ which satisfies $h(c, \cdot, \cdot) = h(\cdot, c, \cdot) = h(\cdot, \cdot, c) = c$ and fulfils on $A_0\backslash\{c\}$ the condition specified in Proposition~\ref{prop:3_7}. By considering all cases, it is easy to see that $h$ is a weak near unanimity polymorphism of the atom structure of $\ra{7}{65}$. Thus, Lemma~\ref{lem:2or3-sym} shows that $\NSP(\ra{7}{65})$ is in $\Ptime$.
\end{proof}

The relation algebra $\ra{5}{37}$ has a normal representation, namely $({\mathbb T};r)[K^a_2]$.

\begin{prop}\label{prop:5_37}
    $\NSP(\ra{5}{37})$ is in $\Ptime$.
\end{prop}

\begin{proof}
    Let $f$ be the following ternary operation on the atom structure $\fA_0$ of $\ra{5}{37}\colon$
    \begin{itemize}
        \item $f$ is a minority operation on the edges $\{a,\id\}, \{r,\breve{r}\}$,
        \item $f(x,y,z) = r$ if $r \in \{x,y,z\} \subseteq \{a,r,\id\}$,
        \item $f(x,y,z) = \breve{r}$ if $\breve{r} \in \{x,y,z\} \subseteq \{a,\breve{r},\id\}$, and
        \item $f(\breve{d}, d, z) = f(\breve{d}, z, d) = f(z, \breve{d}, d) = d$ for $d \in \{r, \breve{r}\}$ and $z \in \{a, \id\}$.
    \end{itemize}
    One can confirm that $f$ is a weak near unanimity polymorphism of $\fA_0$. Thus, Lemma \ref{lem:2or3-sym} shows that $\NSP(\ra{5}{37})$ is in $\Ptime$.
\end{proof}

The relation algebra $\ra{6}{37}$ has the normal representation $(\T;r)[K^a_\omega]$. The relation algebra $\ra{22}{37}$ has the equivalence relation $a \cup \id$. The consistent atomic $\ra{22}{37}$-networks can be seen as oriented graphs that do not embed the oriented graph $(\{0, 1, 2\}; \{(0, 1)\})$; this algebra has a normal representation too. The polymorphism given in the proof of the next proposition is similar to the one above.

\begin{prop}\label{prop:6_37-22_37}
    $\NSP(\ra{6}{37})$ and $\NSP(\ra{22}{37})$ are in $\Ptime$.
\end{prop}

\begin{proof}
    Let $f$ be the function defined as in Proposition~\ref{prop:5_37} with the exception that $(a,a,\id), (a,\id,a)$, and $(\id,a,a)$ are mapped to $a$ instead of $\id$.
    One can confirm that $f$ is a weak near unanimity polymorphism of the atom structures of both $\ra{6}{37}$ and $\ra{22}{37}$. Thus, Lemma~\ref{lem:2or3-sym} shows that $\NSP(\ra{6}{37})$ and $\NSP(\ra{22}{37})$ are in $\Ptime$.
\end{proof}

The relation algebra $\ra{12}{37}$ has a normal representation, namely $K^a_\omega[{\mathbb T}]$.

\begin{prop}\label{prop:12_37}
   $\NSP(\ra{12}{37})$ is in $\Ptime$.
\end{prop}
\begin{proof}
Let $f$ be the ternary operation on the atom structure $\mathfrak A_0$ of $\ra{12}{37}$ which satisfies
\begin{itemize}
    \item $f(\id, \id, \id) = \id$,
    \item $f(x, y, z) = a$ if $a \in \{x, y, z\}$,
    \item $f$ is a minority operation on the edge $\{r, \breve{r}\}$,
    \item $f(\breve{d}, d, \id) = f(\breve{d}, \id, d) = f(\id, \breve{d}, d) = d$ for $d \in \{r, \breve{r}\}$, and
    \item $f(x, y, z) = d$ if $\{x, y, z\} = \{d, \id\}$ for $d \in \{r, \breve{r}\}$.
\end{itemize}
Then $f$ is a weak near unanimity polymorphism of the atom structure of $\ra{12}{37}$. Thus, Lemma~\ref{lem:2or3-sym} shows that $\NSP(\ra{12}{37})$ is in $\Ptime$.
\end{proof}

The relation algebra $\ra{19}{65}$ has the equivalence relation $b \cup \id$ with infinitely many classes of size two. Since also the triple $(a,b,c)$ is forbidden, it is easy to see that the normal representation of $\ra{19}{65}$ is the cycle product representation ${\mathbb R}^{c,a}[K_2^b]$.

\begin{prop}\label{prop:19_65}
    $\NSP(\ra{19}{65})$ is in $\Ptime$.
\end{prop}
\begin{proof}
Let $f$ be the ternary operation on the atom structure $\mathfrak A_0$ of $\ra{19}{65}$ such that
\begin{itemize}
    \item $f(x, y, z) = a$ if $a \in \{x, y, z\}$,
    \item $f$ is a minority operation on the edge $\{b, \id\}$, and
    \item $f(x, y, z) = c$ otherwise.
\end{itemize}
Then $f$ is a weak near unanimity polymorphism of the atom structure of $\ra{19}{65}$. Thus, Lemma~\ref{lem:2or3-sym} shows that $\NSP(\ra{19}{65})$ is in $\Ptime$.
\end{proof}

\begin{rem}
    We mention (without proof) that the $\CSP$ of the atom structure of every relation algebra with at most four atoms not mentioned in this section is $\NP$-complete. To check this we used a computer program to verify the conditions given in Theorem~\ref{thm:dichotomyConservative}.
\end{rem}

\subsection{Tractability via Path Consistency}
\label{sect:PCworks}
In this section we prove that the network satisfaction problem for the algebras $\ra{1}{37}$, $\ra{2}{37}$, and $\ra{8}{37}$ can be solved by the path consistency procedure (Section~\ref{sect:PC}), and hence are polynomial-time tractable.

Let $N = (V,f)$ be an $\bA$-network.
Recall from Remark~\ref{rem:PC-sound} that if the path consistency procedure returns `unsatisfiable', then $N$ is unsatisfiable.
We may therefore suppose that $f(u,v) \neq 0$ for all $u,v \in V$ and that $N$ is path consistent.

\begin{rem}\label{rem:id}
Note that if $u, v \in V$ are such that $f(u, v) = \id$, then $f(u, w) = f(v, w)$ for every other $w \in V$; we may therefore replace all occurrences of $v$ by $u$ and remove $v$ from the instance. The resulting network is still path consistent. We may thus assume that $f(u,v) \neq \id$ for all distinct $u,v$.
\end{rem}

We already discussed a certain equivalence relation which arises in 2-cycle products of relation algebras in Remark~\ref{rem:equ_relation}.
It inspires the following lemma which provides useful structural information about networks which are path consistent.

\begin{lem}\label{lem:path-equivalence}
Let $\bA$ and $\bB$ be two finite integral relation algebras such that $A_0 \cap B_0 = \{\id\}$. Then for every $\bA[\bB]$-network $N = (V, f)$ which is path consistent it holds that
\begin{align*}
    u \sim v :\iff f(u,v) \leq 1_\bA    
\end{align*}
is an equivalence relation on $V$. Moreover, if $u,v \in C_1$ and $x,y \in C_2$ are in two distinct equivalence classes and $c \not\leq f(u,x)$ for some $\id \neq c \in B_0$, then $c\not\leq f(v,y)$.
\end{lem}

\begin{proof}
    It is clear that $\sim$ is an equivalence relation. We first show that $c\not\leq f(u,y)$. We point out that for all $\alpha \in A_0$ and $\beta \in B$ we have
    $c \leq \alpha \circ \beta$
    if and only if $c \leq \beta$. By the path consistency algorithm we have
    \begin{align*}
        f(y,u) \leq f(y,x) \circ f(x,u)
        &\leq \bigcup_{a \in A_0}a \circ \left(\bigcup_{\substack{a \leq f(x,u) \\ a \in A_0}}a \cup \bigcup_{\substack{b \leq f(x,u) \\ b \in B_0}}b \right) \\
        &\leq \bigcup_{a \in A_0}a \cup \left( \bigcup_{a \in A_0}a \circ \bigcup_{\substack{b \leq f(x,u) \\ b \in B_0}}b \right).
    \end{align*}
    So we have that $\breve{c} \leq f(y,u)$ implies
    \[ \breve{c} \leq \bigcup_{\substack{b \leq f(x,u) \\ b \in B_0}}b. \]
    But this is false by assumption. So $\breve{c} \not\leq f(y,u)$ and therefore $c \not\leq f(u,y)$.
    Similarly, one can show that $c\not\leq f(v,y)$.
\end{proof}

The relation algebra $\ra{1}{37}$ has a normal representation, namely ${\mathbb Q}[K_2^a]$.

\begin{prop}\label{prop:1_37}
   Path Consistency solves $\NSP(\ra{1}{37})$.
\end{prop}
\begin{proof}
    Let $N = (V, f)$ be a path-consistent $\ra{1}{37}$-network such that $f(x,y) \neq 0$ for all $x,y \in V$.
    Since $\ra{1}{37}$ has a fully universal representation, it suffices to show that $N$ has a solution $(V,g)$ when
    viewed as an instance of $\NCP(\ra{1}{37})$.
    By Remark~\ref{rem:id} we may assume that $f(x,y) \neq \id$ for all distinct $x,y \in V$.

    First, observe that the set of all pairs $(u,v) \in V^2$ such that $f(u,v) \leq a \cup \id$ is an equivalence relation, which we denote by $\sim$.
    For each equivalence class $C$, we consider the subinstance $(C, f|_{C^2})$, which can be viewed as an instance of $\NSP(\ra{1}{2})$.
    Since $\NSP(\ra{1}{2})$ can be solved by PC~\cite{HirschCristani}, and $N$ is path consistent, each of these subinstances has a solution $g_C$.

    We want to extend these solutions to the whole network. Let $u,v \in V$ be from different classes. If $f(u,v) = 1$ for all such $u$ and $v$, then
    we may pick any linear order on the components and define $g(u,v) := r$ if the component of $u$ is smaller than the component of $v$ with respect to this order.
    So assume there are $u$ and $v$ such that $f(u,v) \neq 1$. We set $g(u,v) := r$ for all $u,v$ such that $f(u,v) \leq a \cup r \cup \id$. Since $f(u,v) = f(v,u)\breve{\phantom{o}}$, the assignment is well-defined.
    By Lemma~\ref{lem:path-equivalence}, if $u,v \in C_1$ and $x,y \in C_2$ are in two distinct equivalence classes of $\sim$, then $f(u,x) \cap (r \cup \breve{r}) = f(v,y) \cap (r \cup \breve{r})$.

    We define a partial order on the classes of $\sim$ by setting $C_1 \leq C_2$ if there exists $u \in C_1$ and $v \in C_2$ such that $g(u,v) = r$. To see that this relation is transitive, suppose that $g(u,v) = r$
    and $g(v,w) = r$. Then, since the network is path consistent, we have
    \begin{align*}
        f(u,w) &\leq f(u,v) \circ f(v,w) \\
        &\leq (a \cup r \cup \id) \circ (a \cup r \cup \id) = a \cup r \cup \id;
    \end{align*}
    hence, we already assigned $g(u,w) = r$.

    For the remaining pairs $(u,v) \in V^2$ where $f(u,v)$ is not yet an atom, we have $r \cup  \breve{r} \leq f(u, v)$. Hence, we can arbitrarily extend the partial order on the classes to a total order, and set $g(u,v) := r$ if the component of $u$ is smaller than the component of $v$ with respect to this order.
    The resulting atomic network $(V,g)$ is a solution of $(V,f)$.
\end{proof}

The relation algebra $\ra{2}{37}$ has a normal representation, namely ${\mathbb Q}[K^a_{\omega}]$.

\begin{prop}\label{prop:2_37}
   Path Consistency solves $\NSP(\ra{2}{37})$.
\end{prop}
\begin{proof}
    The proof is almost the same as for Proposition~\ref{prop:1_37}, with the only difference that we view the equivalence classes of $a \cup \id$ as instances of $\NSP(\ra{2}{2})$ instead of $\NSP(\ra{1}{2})$.
\end{proof}

The relation algebra $\ra{8}{37}$ has a normal representation, namely $K^a_\omega[{\mathbb Q}]$.

\begin{prop}\label{prop:8_37}
   Path Consistency solves $\NSP(\ra{8}{37})$.
\end{prop}
\begin{proof}
    Let $N = (V, f)$ be an $\ra{8}{37}$-network which is path consistent such that $f(x,y) \neq \id$ for all distinct $x, y \in V$.

    Observe that the set of all pairs $u,v$ such that $f(u,v) \leq r \cup \breve{r} \cup \id$ form an equivalence relation.
    On each equivalence class $C$, we create an instance of $\NCP(\ra{1}{3})$ on the $\ra{1}{3}$-network $(C, f\vert_{C^2})$. If for some class $C$ the resulting instance of $\NCP(\ra{1}{3})$ is unsatisfiable, then $N$ is unsatisfiable as well.
    Otherwise, for each equivalence class $C$ the subnetwork $(C, f\vert_{C^2})$ has a solution $g_C$.
    For all $u,v \in V$ in different equivalence classes we have $a \leq f(u,v)$.
    We define $g \colon V^2 \to \ra{8}{37}$ as the common extension
    of all the functions $g_C$ such that
     $g(u,v) := a$ if $u$ and $v$ are from distinct components.
     Then $(V, g)$ is a solution for $N$.
\end{proof}

\subsection{Divide and Conquer Algorithms}
\label{sect:fp}
In this section, we prove that $\NSP(\ra{24}{65})$
and $\NSP(\ra{17}{37})$ are in $\Ptime$.
The algorithms we present are similar in spirit to a known algorithm for
$\NSP(\ra{13}{37})$ (\cite{BodirskyKutz,BodirskyKutzAI}).

\begin{rem}\label{rem:13_37}
    The network satisfaction problem for $\ra{13}{37}$ cannot be solved by Datalog (this can be seen from the techniques in~\cite[Section 8.6]{Book}),
    and hence in particular not by the path consistency procedure, but it can be expressed in fixed point logic and hence is in $\Ptime$~\cite{Book}. It thus solves Problem 1 of Cristani and Hirsch~\cite{HirschCristani}: \emph{`Find a (finite) relation algebra $\bA$ with no universal representation\footnote{Cristani and Hirsch~\cite{HirschCristani} call a representation $\fB$ of a finite relation algebra $\bA$ \emph{universal} if every consistent $\bA$-network with $f(x,y) \neq 0$ for all vertices $x,y$ is satisfiable in $\fB$. Note that this definition differs from our usage.}
    but where the complexity of $\NSP(\bA)$ is polynomial.'}
    The existence of problems whose complement is closed under homomorphisms and that are in fixed point logic, but not in Datalog, is one of the main results of~\cite{DawarKreutzer08}.
\end{rem}

\subsubsection{The Relation Algebra \texorpdfstring{$\ra{24}{65}$}{24\textunderscore 65}}
\label{sect:24-alg}

We need the following facts about consistent atomic $\ra{24}{65}$-networks $(V,f)$.
If $d \in \{a,b,c\}$, let $G_d$ be the graph
$G_d := (V; \{(x,y) \mid f(x,y) \in \{d,\id\}\})$.
The following is straightforward.

\begin{lem}\label{lem:1}
    Let $(V,f)$ be a consistent atomic $\ra{24}{65}$-network.
    If $C_1$ and $C_2$ are different connected components of $G_d$, for some $d \in \{a,b,c\}$, then $f$ is constant on
    all pairs $(x,y)$ for $x \in C_1$ and $y \in C_2$.
\end{lem}

The following lemma is due to Sebastian Meyer (personal communication).

\begin{lem}\label{lem:2}
    Let $(V,f)$ be a consistent reduced atomic $\ra{24}{65}$-network such that the image of $f$ is not $\{\id\}$. Then at least one of the graphs
    $G_a$, $G_b$, and $G_c$ is disconnected.
\end{lem}
\begin{proof}
By induction on $|V|$.
If $|V|=1$ then the image of $f$ must be $\{\id\}$, so suppose that $|V| \geq 2$.
The case $|V|=2$ is clear. For the case that $|V| \geq 3$,
suppose for contradiction that all of $G_a$, $G_b$, and $G_c$ are connected.
 Let $x \in V$.
By the inductive assumption there exists $d \in \{a,b,c\}$ such that the graph $G_d \backslash \{x\}$ is disconnected. We may assume without loss of generality that $d=a$ and that $U_1,\dots,U_k$, for $k \geq 2$, are the connected components of $G_a \backslash \{x\}$. For each $U_i$, $i \in \{1,\dots,k\}$, we may choose $u_i \in U_i$
such that $f(x,u_i)=a$, because
$G_a$ is connected.
Moreover, there exists $v_b \in V \setminus \{x\}$
such that $f(x,v_b)=b$, because
$G_b$ is connected;
similarly, there exists $v_c \in V \setminus \{x\}$ such that $f(x,v_c)=c$.

There are the following two cases to consider:
\begin{itemize}
\item $v_b,v_c \in U_i$, for some $i \in \{1,\dots,k\}$. Let $j \in \{1,\dots,k\} \setminus \{i\}$. If $f(v_b,u_j)=c$, then
$\{x,v_b,u_j\}$ induces a rainbow.
Otherwise, $f(v_b,u_j)=b$,
and then $f(v_c,u_j)=b$ by Lemma~\ref{lem:1}.
In this case, $\{x,v_c,u_j\}$ induces a rainbow.
\item $v_b \in U_i$ and $v_c \in U_j$, for distinct $i,j \in \{1,\dots,k\}$. If $f(v_b,v_c) =b$, then $f(u_i,v_c)=b$ by Lemma~\ref{lem:1}. Then $\{x,u_i,v_c\}$ induces a rainbow. Otherwise, $f(v_b,v_c)=c$; this case can be treated similarly.
\qedhere
\end{itemize}
\end{proof}

Note that all relations of the relation algebra $\ra{24}{65}$ can be generated from the following relations:
$\overline{\id} = a \cup b \cup c$ and $R_d := (d \cup \id)$ for $d \in \{a,b,c\}$.
For example, $\overline{a} = R_b \circ R_c$.
Similarly we can obtain $\overline{b}$ and $\overline{c}$.
We can then obtain all others by appropriately intersecting the ones that we already have.
It therefore suffices to present an algorithm for the NSP restricted to
networks with the relations ${\mathcal R} := \{R_a,R_b,R_c,\overline{\id},\id,1\}$ (see Remark~\ref{rem:reduce}).

\begin{defi}
    Let $(V,g)$ be a $\ra{24}{65}$-network such that
    the image of $g$ is contained in $\mathcal R$ and let $d \in \{a,b,c\}$.
    A \emph{$d$-cut of $(V,g)$} is a partition
    of $V$ into $k \geq 2$ disjoint
    sets $C_1,\dots,C_k$
    such that if $g(x,y) \in \{R_d,\id\}$, then $x$ and $y$ must lie in the same part.
    A $d$-cut is \emph{valid} if for all $1\leq i<j\leq k$ there exists
    $d_{i,j} \in \{a,b,c\} \setminus \{d\}$ such that
    $d_{i,j} \leq g(u,x)$ and $d_{i,j} \leq g(x,u)$ for all $u \in C_i$ and $x \in C_j$.
    For $i>j$, put $d_{i,j}:=d_{j,i}$.
\end{defi}

\begin{defi}\label{def:non-id-sol}
A \emph{non-identity solution} to an $\bA$-network $(V,g)$ is a solution $(V,f)$ such that there are $x,y\in V$ with $f(x,y)\neq \id$.
\end{defi}

\begin{lem}\label{lem:24unsat}
    Let $(V,g)$ be a $\ra{24}{65}$-network such that $(V,g)$ has a non-identity solution and
    the image of $g$ is contained in $\mathcal R$. Then there is a valid $d$-cut in $(V,g)$ for some $d \in \{a,b,c\}$.
\end{lem}
\begin{proof}
    Let $(V,f)$ be a non-identity solution of $(V,g)$.
    By Lemma~\ref{lem:2}, there exists $d \in \{a,b,c\}$ such that $G_d$ is disconnected. Let $C_1,\dots,C_k \subseteq V$ be the connected components of $G_d$. Clearly, it is a $d$-cut of $(V,g)$. Lemma~\ref{lem:1}, together with the fact that $a,b,c$ are self-converse, implies that it is valid.
\end{proof}

Note that the condition from Lemma~\ref{lem:24unsat} can be checked in polynomial time: to check whether $(V,g)$ has a valid $d$-cut, we
first compute the connected components of the graph
$(V; \{(x,y) \mid g(x,y) \in \{R_d,\id\}\}$ (e.g., via depth-first search).
If the connected components provide a valid $d$-cut we are done. Otherwise, there exist two components $C_i$ and $C_j$ that witness that the cut is not valid. In this case, we merge $C_i$ and $C_j$ into one component. If the resulting $d$-cut is valid, we are done. Otherwise, we repeat contracting components. If the procedure ends up with a single component, then there does not exist a valid $d$-cut.

\begin{prop}\label{prop:24_65}
    $\NSP(\ra{24}{65})$ is in $\Ptime$.
\end{prop}
\begin{proof}
Let $(V,g)$ be a given $\ra{24}{65}$-network. As we have explained earlier, we may assume without loss of generality that the image of $g$ is contained in $\mathcal R$.

\begin{enumerate}
\item If there is no pair of variables labelled  $\overline{\id}$, then answer `satisfiable'
(we may map all of $V$ to the same point in a representation of $\ra{24}{65}$).
\item Otherwise, if for every $d \in \{a,b,c\}$ there is no valid $d$-cut in $(V,g)$ return `unsatisfiable'
(there is no solution by Lemma~\ref{lem:24unsat}).
\item Otherwise, let $C_1,\dots,C_k$ be a valid $d$-cut in $(V,g)$, for some $d \in \{a,b,c\}$, and recursively solve the subinstances induced on each of $C_1,\dots,C_k$.
\item If a recursive call returns `unsatisfiable', then return `unsatisfiable'.
\item
Otherwise, return `satisfiable' (a consistent atomic network $(V,f)$ with $f(x,y) \leq g(x,y)$ for all $x,y \in V$ can be obtained from solutions for $C_1,\dots, C_k$ by setting $f(x,y) := d_{i,j}$ for every $x\in C_i$ and $y\in C_j$ with $i\neq j$).
\end{enumerate}
The proof that this algorithm is correct follows by induction on the number of variables of the instance. In Step~5, the resulting atomic network is consistent: a triangle meeting two different parts has two equally labelled cross-edges, while a triangle meeting three different parts uses only the two atoms in $\{a,b,c\}\setminus\{d\}$, so no rainbow occurs. The algorithm runs in polynomial time.
\end{proof}

\subsubsection{The Relation Algebra \texorpdfstring{$\ra{17}{37}$}{17\textunderscore 37}}
\label{sect:17-alg}
We now present a polynomial-time algorithm for
$\NSP(\ra{17}{37})$ (note that $\ra{17}{37}$ has no fully universal representation by Lemma~\ref{lem:nfu-37}, and so this problem is different from $\NCP(\ra{17}{37})$).
Recall
from Section~\ref{sect:17_37-rep}
that
the consistent atomic $\ra{17}{37}$-networks can be represented as oriented graphs that do not contain $P_2 = (\{0,1,2\}; \{(0,1),(1,2)\})$ as an induced subgraph.

We need a structural result about such graphs proved by Bang-Jensen and Huang~\cite{BJJ95}; they obtain even more general results for directed graphs, but we only need and only state their result for oriented graphs (i.e., the special case of directed graphs without directed cycles of length two).
If $G = (V; E)$ is an oriented graph, then $\widetilde G$ denotes the undirected graph
\begin{align*}
    \widetilde G := (V,\{(x,y) \mid (x,y) \notin E, (y,x) \notin E, x \neq y\}).
\end{align*}

\begin{thm}[{\cite[Theorem 3.5]{BJJ95}}]
\label{thm:17-1}
Let $G$ be an oriented graph with no induced $P_2$.
If $G$ is not strongly connected, then it is obtained from a transitive oriented graph by substituting its vertices with strongly connected oriented graphs with no induced $P_2$. If $G$ is strongly connected, then it is obtained from a strongly connected tournament by substituting its vertices with not strongly connected oriented graphs with no induced $P_2$.
\end{thm}

\begin{lem}[{\cite[Lemma 3.4]{BJJ95}}]
\label{lem:17}
Let $G$ be an oriented graph with no induced $P_2$. If $G$ is strongly connected then $\widetilde G$ is not connected, and between each pair of components of $\widetilde G$ all edges in $G$ go in the same direction.
\end{lem}

The following is an easy consequence of these results.

\begin{thm}
\label{thm:17-2}
Let $G$ be an oriented graph with no induced $P_2$ with at least two vertices.
\begin{itemize}
    \item If $G$ is not strongly connected then it is obtained from a transitive oriented graph
    by substituting its vertices by strongly connected oriented graphs with no induced $P_2$.
    \item If $G$ is strongly connected then
    $\widetilde{G}$ is not connected, and $G$ is obtained from a (strongly connected) tournament by substituting its vertices by oriented graphs $C_1,\ldots, C_k$ such that each $C_i$ contains no induced $P_2$ and each $\widetilde{C}_i$ is connected.
\end{itemize}
\end{thm}
\begin{proof}
If $G = (V; E)$ is not strongly connected then this is exactly what Theorem~\ref{thm:17-1} tells us.
If $G$ is strongly connected then $\widetilde G$ is not connected by Lemma~\ref{lem:17}. Let $A$ and $B$ be two distinct connected components of $\widetilde G$. By Lemma~\ref{lem:17}, $A \times B$ is contained in $E(G)$,
or $B \times A$ is contained in $E(G)$. Consequently, after contracting each connected component of $\widetilde G$ (which correspond to $\widetilde{C}_i$) to a vertex, we obtain a tournament. Since $G$ is strongly connected, the tournament needs to be strongly connected as well.
\end{proof}

By Lemma~\ref{lem:nfu-37}, the induced $a$-path on 4 vertices is not satisfiable. Consequently, the satisfiable atomic $\ra{17}{37}$-networks do not contain long induced $a$-paths, which has strong structural consequences which we now explain.

Partially ordered sets (posets) $(V; E)$ that do not contain distinct elements $u,v,a,b$ such that $(u,v), (a,v), (a,b) \in E$ and all other pairs of distinct elements
from $\{u,v,a,b\}$ are not in $E$, are called \emph{$N$-free}.
It is well-known that a poset is $N$-free if and only if it is \emph{series-parallel}, i.e., is contained in the smallest class ${\mathcal C}$ of posets that contains the one-element poset, is closed under isomorphism, and is closed under the following two operations~\cite{SP}: if $(V_1; E_1), (V_2; E_2) \in {\mathcal C}$
    are such that $V_1 \cap V_2 = \varnothing$, then ${\mathcal C}$ also contains
\begin{itemize}
    \item $(V_1 \cup V_2; E_1 \cup E_2 \cup V_1 \times V_2)$ (series composition), and
    \item $(V_1 \cup V_2; E_1 \cup E_2)$ (parallel composition).
\end{itemize}

Observe that $\ra{17}{37}$ is generated by the three relations $R_r := r \cup \id$, $R_a := a \cup \id$, and $\overline{\id}$.
It therefore suffices to present an algorithm for the NSP restricted to networks from $\mathcal R := \{R_r,R_a,\overline{\id},\id,1\}$ (see Remark~\ref{rem:reduce}).
If $(V,f)$ is a $\ra{17}{37}$-network such that the image of $f$ is contained in $\mathcal R$, then we define the following two graphs:
\begin{itemize}
    \item $G := (V;\{(x,y) \mid f(x,y) \in \{R_r,\id\}\})$ and
    \item $H := (V;\{(x,y),(y,x) \mid f(x,y) \in \{R_a,\id\}\})$.
\end{itemize}

\begin{defi}\label{defn:17_37_acut}
    Let $(V,g)$ be a $\ra{17}{37}$-network such that the image of $g$ is contained in $\mathcal R$.
    An \emph{$a$-cut} is a partition of $V$ into $k \geq 2$ disjoint sets $C_1,\dots,C_k$ such that
    \begin{itemize}
        \item each $C_i$, for $1\leq i \leq k$, is a union of connected components of $H$, and
        \item for all $1\leq i<j\leq k$ there exists $d_{i,j} \in \{r,\breve{r}\}$ such that $d_{i,j} \leq g(x,y)$ and $\breve{d}_{i,j} \leq g(y,x)$ for all $x \in C_i$ and $y \in C_j$.
    \end{itemize}
    For $i>j$, put $d_{i,j}:=\breve{d}_{j,i}$.
\end{defi}

Note that, given a $\ra{17}{37}$-network $(V,g)$  such that the image of $g$ is contained in $\mathcal R$, one can in polynomial time either find an $a$-cut, or correctly say that it does not exist. To do this, we will keep a partition $\mathcal C$ of $V$ such that each part is a union of connected components of $H$. We start by computing the connected components of $H$ and partitioning $V$ to these components. Then, while there are some distinct $C,C'\in \mathcal C$ with no $d \in \{r,\breve{r}\}$ such that $d \leq g(x,y)$ and $\breve d \leq g(y,x)$ for all $x \in C$ and $y \in C'$, we merge $C$ and $C'$. This procedure either finds an $a$-cut, or ends up with $\mathcal C = \{V\}$, in which case it is easy to see by induction that all merges were necessary and thus there is no $a$-cut.

\begin{defi}\label{defn:17_37_spcut}
    Let $(V,g)$ be a $\ra{17}{37}$-network such that the image of $g$ is contained in $\mathcal R$. Let $C_1,C_2$ be a partition of $V$ into two non-empty disjoint sets such that both $C_1$ and $C_2$ are unions of strongly connected components of $G$. We say that $C_1,C_2$ is an \emph{$s$-cut} if $r \leq g(x,y)$ and $\breve r \leq g(y,x)$ for all $x \in C_1$ and $y \in C_2$, and it is a \emph{$p$-cut} if $a \leq g(x,y)$ and $a \leq g(y,x)$ for all $x \in C_1$ and $y \in C_2$. For an $s$-cut put $d_{1,2}:=r$ and $d_{2,1}:=\breve r$, and for a $p$-cut put $d_{1,2}=d_{2,1}:=a$.
\end{defi}

Similarly as above, there are polynomial-time algorithms which, given a $\ra{17}{37}$-network $(V,g)$ such that the image of $g$ is contained in $\mathcal R$, either find a $p$-cut or an $s$-cut, or correctly say that such a cut does not exist. Compute the strongly connected components $D_1,\dots,D_k$ of $G$. For $e\in\{s,p\}$ and for each $i\in\{1,\dots,k\}$ in turn, start with $C:=D_i$. While there is some component $D_j\subseteq V\setminus C$ such that the requirements of an $e$-cut fail between $C$ and $D_j$, add $D_j$ to $C$; explicitly, for $e=s$ this means that $r\nleq g(x,y)$ or $\breve r\nleq g(y,x)$ for some $x\in C,y\in D_j$, and for $e=p$ that $a\nleq g(x,y)$ or $a\nleq g(y,x)$ for some such $x,y$. If the process stops with $C\neq V$, then $C,V\setminus C$ is an $e$-cut. If it reaches $V$, try the next $D_i$. If this happens for every $D_i$, then no $e$-cut exists: indeed, if $C_1,C_2$ were an $e$-cut and $D_i\subseteq C_1$, then the closure process starting with $D_i$ could never add a component contained in $C_2$.

\begin{lem}\label{lem:nsp-17-unsat}
    Every $\ra{17}{37}$-network $(V,g)$
    with a satisfiable non-identity solution $(V,f)$ (Definition~\ref{def:non-id-sol}) such that the image of $g$ is contained in ${\mathcal R}$
    has an $a$-cut, a $p$-cut, or an $s$-cut.
\end{lem}
\begin{proof}
    Note that the set of pairs $x,y$ such that $f(x,y) = \id$ forms an equivalence relation. Let $V'$ be the set of its equivalence classes (we will use $[x]$ to denote the equivalence class of $x$).
    Then $G' = (V',\{([x],[y]) \mid f(x,y) = r\})$ is an oriented graph with at least two vertices (since the image of $f$ is not $\{\id\}$) and with no induced $P_2$ and no induced $a$-path on 4 vertices (since $(V,f)$ is satisfiable atomic). By Theorem~\ref{thm:17-2}, either $G'$ is not strongly connected, or $\widetilde G'$ is not connected.

    If $\widetilde G'$ is not connected, let $C'_1,\dots,C'_k$ be the connected components of $\widetilde{G}'$ and, for each $i\in \{1,\dots,k\}$, let $C_i$ consist of those $x\in V$ such that there is $[y]\in C_i'$ with
    $f(x,y) = \id$.
    We will show that $C_1,\dots,C_k$ is an $a$-cut of $(V,g)$: First observe that if $x$ and $y$ are in the same connected component of $H$, then $[x]$ and $[y]$ are in the same connected component of $\widetilde{G}'$, and thus each $C_i$ is a union of connected components of $H$. Next, from Theorem~\ref{thm:17-2} we know that $G'$ is obtained from a tournament
    $T$ on vertex set $\{1,\dots,k\}$ by substituting vertex $i$ by the elements of $C_i'$ for all $i\in\{1,\dots,k\}$ (Theorem~\ref{thm:17-2} does not explicitly promise that we are substituting exactly by the graphs $C_i'$, but it promises that the graphs we are substituting by are connected in the ``non-edge'' relation, and the only possibility is that these are the graphs $C_i'$).
    Put
    \begin{align*}
        d_{i,j} = \begin{cases}
        r &\text{ if $(i,j)$ is an edge of $T$, and}\\
        \breve{r} &\text{ otherwise}.
    \end{cases}
    \end{align*}
    Since $T$ is a tournament, this is well-defined. By definition of substitution we have $f(x,y) = d_{i,j}$ for all $x\in C_i$ and $y\in C_j$ for all $i\neq j\in \{1,\dots,k\}$. Since $f(x,y) \leq g(x,y)$ for all $x,y\in V$, and since $T$ is a tournament, it follows that $C_1,\dots,C_k$ is indeed an $a$-cut.

    So $G'$ is not strongly connected. Let $C'_1,\dots,C'_k$ be the strongly connected components of $G'$. By Theorem~\ref{thm:17-2} we know that there is a transitive oriented graph $T$ with vertex set $\{1,\dots,k\}$ such that $G'$ is obtained from $T$ by substituting vertex $i$ by $C_i'$ for every $i\in\{1,\dots,k\}$.  Moreover, since $(V,f)$ is satisfiable atomic, it contains no induced $a$-path on 4 vertices, and so $G'$, and thus also $T$, are $N$-free.
    Hence, by the fact mentioned above, the poset obtained from $T$ by adding all the loops is series-parallel.
    Consequently, one can split $T$ into disjoint non-empty sets $T_1$, $T_2$ such that either there is an oriented edge from every member of $T_1$ into every member of $T_2$ (series composition), or there are no edges between $T_1$ and $T_2$ at all (parallel composition).

    Let $C_1$ be the set of all $x\in V$ for which there is $i\in T_1$ and $[y]\in C_i'$ with $f(x,y)=\id$, and put $C_2 = V\setminus C_1$. Observe that both $C_1$ and $C_2$ are non-empty and that they are unions of strongly connected components of $G$ (as if $x$ and $y$ are in the same strongly connected component of $G$, then $[x]$ and $[y]$ are in the same strongly connected component of $G'$). It is then easy to see that $C_1$, $C_2$ is an $s$- or $p$-cut (based on whether there are edges between $T_1$ and $T_2$).
\end{proof}

\begin{thm}\label{prop:17_37}
    $\NSP(\ra{17}{37})$ is in $\Ptime$.
\end{thm}
\begin{proof}
Let $(V,g)$ be a given $\ra{17}{37}$-network. As we have explained earlier, we may assume without loss of generality that the image of $g$ is contained in ${\mathcal R}$.
\begin{enumerate}
    \item \label{alg1} If there is no pair of variables labelled $\overline{\id}$, then answer `satisfiable' (we may map all of $V$ to the same point in a representation of $\ra{17}{37}$).
    \item \label{alg2} Otherwise, if there is no $a$-cut, no $s$-cut, and no $p$-cut in $(V,g)$, return `unsatisfiable'.
    \item \label{alg3} If there is an $e$-cut $C_1,\dots,C_k$,
    for $e \in \{a,s,p\}$,
    recursively solve the subinstances induced on each of $C_1,\dots,C_k$.
    \item \label{alg4} If a recursive call returns `unsatisfiable', then return `unsatisfiable'.
    \item \label{alg5} Otherwise, return `satisfiable'.
\end{enumerate}
We prove the correctness of the algorithm by induction over the recursion tree of the recursive procedure. Clearly, if the algorithm answers `satisfiable' in Step~\ref{alg1}, this answer is correct.
If the algorithm returns unsatisfiable in Step~\ref{alg2} then
indeed there is no solution by Lemma~\ref{lem:nsp-17-unsat}. If the algorithm returns `unsatisfiable' in Step~\ref{alg4} then this is the correct answer by the inductive assumption. The interesting part is to prove that the answer `satisfiable' is correct in Step~\ref{alg5}. For this, we construct a solution $(V,f)$ for $(V,g)$ from the solutions $(C_i,f_i)$ of the subnetworks induced on $C_1,\dots,C_k$ as follows:
\begin{itemize}
    \item if $u,v \in C_i$ for some $i \in \{1,\dots,k\}$, then $f(u,v) := f_i(u,v)$,
    \item otherwise $u \in C_i$ and $v \in C_j$, for some distinct $i,j \in \{1,\dots,k\}$. In this case, put $f(u,v) := d_{i,j}$ (see Definitions~\ref{defn:17_37_acut} and~\ref{defn:17_37_spcut}).
\end{itemize}

It is easy to see that if $(V,f)$ contains a copy of $P_2$ then this copy fully lies in $(C_i,f_i)$ for some $i\in \{1,\dots,k\}$. Hence, by induction, there is no copy of $P_2$ in $(V,f)$. Similarly, if $(V,f)$ contains an induced $a$-path on 4 vertices, then all 4 vertices must lie in $(C_i,f_i)$ for some $i\in \{1,\dots,k\}$. Hence, $(V,f)$ indeed witnesses that $(V,g)$ is satisfiable.

It remains to argue that the algorithm runs in polynomial time. By a standard recursion tree analysis, it is enough to verify that one stage without including the recursive calls runs in polynomial time, and this follows from the observations that one can find $a$-, $s$-, and $p$-cuts in polynomial time (or determine that they do not exist).
\end{proof}

\section{Open Problems}

We view our complexity dichotomy for the network satisfaction problem for relation algebras with at most 16 elements as an encouraging basis for a more general classification result without any restriction on the number of elements, at least for the finite relation algebras with a normal representation.
The following other problems are also left open.

\begin{qu}
    Is the existence of a fully universal representation for a given finite relation algebra effectively decidable?
\end{qu}

\begin{prob}
Obtain a similar classification for the network consistency problem.
\end{prob}
For relation algebras with a fully universal representation, the NSP and the NCP have the same complexity. For most of the other representable algebras, the proofs which we present here for the computational complexity of the NSP can be easily adapted to the NCP, the exception being $\ra{17}{37}$. Finally, for algebras with no representation, additional efforts are required, since the NSP is trivial, whereas the NCP is generally not.

\begin{qu}
    What is the smallest relation algebra $\bA$ such that $\NSP(\bA)$ is not in $\NP$?
\end{qu}

\begin{prob}
Obtain a similar classification for the variant of the network satisfaction problem where the inputs are restricted to networks $(V, f)$ such that the range of $f$ consists only of atoms and $1$.
\end{prob}

\section{Summary and Tables}
In Tables~\ref{fig:2Atoms}--\ref{tab:overview_symmetric} we present Maddux's lists of integral relation algebras, together with references and comments about their representability and the computational complexity of their network satisfaction problem, thus summarizing the findings of our work.
The first column is the numbering introduced by Maddux. The following columns are labelled by triples (omitting brackets and commas for space reasons).
An entry of the table contains a triple of the form $xyz$ if it is an allowed triple of the relation algebra, i.e., if $z \leq x \circ y$ (see Section \ref{sect:relationalgebras_def}); this uniquely describes the relation algebra thanks to the cycle law (Proposition~\ref{prop:cyclelaw}).

In the tables, we use the following abbreviations (definitions can be found in Section \ref{sect:prelims}):
\begin{itemize}
    \item none: no representation exists
    \item $\neg$f.u.: the relation algebra is representable, but there is no fully universal representation.
    \item f.u.: there exists a fully universal square representation, but no normal representation.
    \item normal (n): there exists a normal representation, but the algebra has no flexible atom. In particular, the normal representations in Table~\ref{tab:overview_asymmetric} correspond to homogeneous directed graphs~\cite{Cherlin,LachlanFiniteDigraphs}.
    \item flex: there is a flexible atom (which implies the existence of a normal representation).
\end{itemize}

In cases where the representations admit a particularly simple description, we include them explicitly in the tables; formal definitions are given at the end of Section \ref{sect:normal}.

In~\cite{BodirskyKnaeRamics} it has been noted that
in the case of $2_7$ and $6_7$ the claims about the complexity  of the NSP from~\cite{HirschCristani} are incorrect, and we state the corrected results.

\begin{table}[H]
\begin{center}
\begin{tabular}{|L||L|L|L|} \hline
    \#_2 & aaa & \text{Representability} & \text{NSP Complexity} \\ \hline
    \ralabel{1}{2} & & \text{normal}, K^a_2 & \Ptime,~\text{\tabcite{HirschCristani}, Prop.~\ref{prop:1_2-in-P}} \\
    \ralabel{2}{2} & aaa & \text{flex}, K^a_\omega & \Ptime,~\text{\tabcite{HirschCristani}, Prop.~\ref{prop:can-bin-sym}} \\ \hline
    \end{tabular}
\end{center}
    \caption{Integral relation algebras with two atoms.}
    \label{fig:2Atoms}
\end{table}

\begin{table}[H]
\begin{center}
\begin{tabular}{|L||LL|L|L|} \hline
    \#_3 & rrr & rr\breve{r} & \text{Representability} & \text{NSP Complexity} \\ \hline
    \ralabel{1}{3} & rrr & & \text{normal}, \Q & \Ptime,~\text{\tabcite{PointAlgebra}} \\
    \ralabel{2}{3} & & rr\breve{r} & \text{normal}, C_3^r & \NPc,~\tabcite{HirschCristani},~\text{Lemma~\ref{lem:col}} \\
    \ralabel{3}{3} & rrr & rr\breve{r} & \text{flex}, \T & \Ptime,~\text{\tabcite{HirschCristani}, Prop.~\ref{prop:can-bin-sym2}} \\ \hline
\end{tabular}
\end{center}
\caption{Integral asymmetric relation algebras with three atoms.}
\label{fig:3AsyAtoms}
\end{table}

\begin{table}[H]
\begin{center}
\begin{tabular}{|L||LLLL|L|L|}\hline
    \#_7 & aaa & bbb & abb & baa & \text{Representability} &  \text{NSP Complexity} \\ \hline
    \ralabel{1}{7} & & & abb & & \text{normal}, K_2^b[K_2^a] & \NPc,~\tabcite{HirschCristani},~\text{Lemma~\ref{lem:col}} \\
    \ralabel{2}{7} & aaa & & abb & & \text{normal}, K_2^b[K_\omega^a] & \NPc,~\tabcite{BodirskyKnaeRamics},~\text{Prop.~\ref{prop:normal_hard_finite}} \\
    \ralabel{3}{7} & & bbb & abb & & \text{normal}, K_\omega^b[K_2^a] & \Ptime,~\tabcite{HirschCristani},~\text{Corollary~\ref{cor:3_7}} \\
    \ralabel{4}{7} & aaa & bbb & abb & & \text{normal}, K_\omega^b[K_\omega^a] & \Ptime,~\tabcite{HirschCristani},~\text{Prop.~\ref{prop:can-bin-sym}} \\
    \ralabel{5}{7} & & & abb & baa & \neg\text{f.u.}, \Z_5^{a, b}, \text{Example \ref{example:5_7_not_fu}} & \NPc,~\tabcite{HirschCristani},~\text{Lemma~\ref{lem:bounded}} \\
    \ralabel{6}{7} & aaa & & abb & baa & \text{flex}, \Hgraph^{b,a} & \NPc,~\tabcite{BodirskyKnaeRamics},~\text{Prop.~\ref{prop:prim}} \\
    \ralabel{7}{7} & aaa & bbb & abb & baa & \text{flex}, \R^{a,b} & \Ptime,~\text{\tabcite{HirschCristani}, Prop.~\ref{prop:can-bin-sym}} \\ \hline
\end{tabular}
\end{center}
\caption{Integral symmetric relation algebras with three atoms.}
\label{fig:3SymAtoms}
\end{table}

\newpage
\begin{longtable}[H]{|L||LLLLLLL||L|L|}\hline
    \#_{37}   & aaa   & rrr   & rr\breve{r}   & arr   & rar   & raa   & rra   & \text{Representation} & \text{NSP Complexity} \\ \hline \hline
    \ralabel{1}{37} &       & rrr   &               & arr   & rar   &       &       &                \text{normal}, \Q[K^a_2]  & \Ptime,~\text{Prop.~\ref{prop:1_37}}  \\
    \ralabel{2}{37} &  aaa  & rrr   &               & arr   & rar   &       &       &                \text{normal}, \Q[K^a_{\omega}]  & \Ptime,~\text{Prop.~\ref{prop:2_37}}   \\
    \ralabel{3}{37} &       &       & rr\breve{r}   & arr   & rar   &   &    &        \text{normal},C_3^r[K^a_2] & \NPc,~\text{Prop.~\ref{prop:normal_hard_finite}}  \\
    \ralabel{4}{37} &  aaa  &       & rr\breve{r}   & arr   & rar   &       &      & \text{normal},C_3^r[K^a_\omega] & \NPc,~\text{Prop.~\ref{prop:normal_hard_finite}}  \\
    \ralabel{5}{37} &       & rrr   & rr\breve{r}   & arr   & rar   &       &       &            \text{normal}, \T[K^a_2]    &   \Ptime,~\text{Prop.~\ref{prop:5_37}}  \\
    \ralabel{6}{37} & aaa   & rrr   & rr\breve{r}   & arr   & rar   &       &     &  \text{normal}, \T[K^a_{\omega}]  &                 \Ptime,~\text{Prop.~\ref{prop:6_37-22_37}}  \\
    \ralabel{7}{37} &       & rrr   &               &       &       & raa   &  &    \text{normal}, K_2^a[\Q]  & \NPc,~\text{Prop.~\ref{prop:normal_hard_finite}}  \\
    \ralabel{8}{37} & aaa   & rrr   &               &       &       & raa   &     & \text{normal}, K_{\omega}^a[\Q]  &                \Ptime,~\text{Prop.~\ref{prop:8_37}}  \\
    \ralabel{9}{37} &       &       & rr\breve{r}   &       &       & raa   &     & \text{normal},K^a_2[C_3^r] & \NPc,~\text{Prop.~\ref{prop:normal_hard_finite}}  \\
    \ralabel{10}{37}& aaa   &       & rr\breve{r}   &       &       & raa   &     &  \text{normal},K^a_\omega[C_3^r] & \NPc,~\text{Prop.~\ref{prop:2CycleProNPc}}  \\
    \ralabel{11}{37}&       & rrr   & rr\breve{r}   &       &       & raa   &   &   \text{normal}, K_2^a[\T]  & \NPc,~\text{Prop.~\ref{prop:normal_hard_finite}} \\
    \ralabel{12}{37}& aaa   & rrr   & rr\breve{r}   &       &       & raa   &     &   \text{normal}, K_{\omega}^a[\T] &              \Ptime,~\text{Prop.~\ref{prop:12_37}}   \\
    \ralabel{13}{37}& aaa   & rrr   &               & arr   &       & raa   &       &  \text{f.u., Prop.~\ref{prop:13_37-repr}} & \Ptime,~\text{Rem.~\ref{rem:13_37},~\tabcite{BodirskyKutz}}   \\
    \ralabel{14}{37}&       & rrr   &               & arr   & rar   & raa   &      & \text{none} & \Ptime,~\text{Thm.~\ref{thm:non-repr}}  \\
    \ralabel{15}{37}& aaa   & rrr   &               & arr   & rar   & raa   &  & \text{normal, $\mathbb P$} &  \NPc,~\text{Prop.~\ref{prop:15_37}},~\tabcite{BroxvallJonsson}              \\
    \ralabel{16}{37}&       & rrr   & rr\breve{r}   & arr   & rar   & raa   &    &    \text{none} & \Ptime,~\text{Thm.~\ref{thm:non-repr}}  \\
    \ralabel{17}{37}& aaa   & rrr   & rr\breve{r}   & arr   & rar   & raa   &       &     \neg\text{f.u., Lemma~\ref{lem:nfu-37}}         &    \Ptime,~\text{Prop.~\ref{prop:17_37}}    \\
    \ralabel{18}{37}&       &       &               &       &       &       & rra  & \text{normal, $C_4$} & \NPc,~\text{Prop.~\ref{prop:normal_hard_finite}} \\
    \ralabel{19}{37}&       & rrr   & rr\breve{r}   &       &       &       & rra   &         \text{normal, $\hat{\mathbb T}$}  &       \NPc,~\text{Prop.~\ref{prop:19_37}}   \\
    \ralabel{20}{37}& aaa   &       &               & arr   & rar   &       & rra   & \text{normal, $\mathbb D_2$} & \NPc,~\text{Prop.~\ref{prop:normal_hard_finite}}  \\
    \ralabel{21}{37}& aaa   & rrr   &               & arr   & rar   &       & rra   & \text{none} & \Ptime,~\text{Thm.~\ref{thm:non-repr}}  \\
    \ralabel{22}{37}& aaa   & rrr   & rr\breve{r}   & arr   & rar   &       & rra   &         \text{normal, $\mathbb D_\omega$} &        \Ptime,~\text{Prop.~\ref{prop:6_37-22_37}}      \\
    \ralabel{23}{37}&       & rrr   &               &       &       & raa   & rra   & \text{normal, $S(3)$} & \NPc,~\text{Prop.~\ref{prop:prim}}  \\
    \ralabel{24}{37}& aaa   & rrr   &               &       &       & raa   & rra   & \text{none} & \Ptime,~\text{Thm.~\ref{thm:non-repr}} \\
    \ralabel{25}{37}&       & rrr   & rr\breve{r}   &       &       & raa   & rra   & \text{none} & \Ptime,~\text{Thm.~\ref{thm:non-repr}} \\
    \ralabel{26}{37}& aaa   & rrr   & rr\breve{r}   &       &       & raa   & rra   & \text{none} &  \Ptime,~\text{Thm.~\ref{thm:non-repr}} \\
    \ralabel{27}{37}&       & rrr   &               & arr   &       & raa   & rra   & \text{none} & \Ptime,~\text{Thm.~\ref{thm:non-repr}} \\
    \ralabel{28}{37}& aaa   & rrr   &               & arr   &       & raa   & rra  & \text{none} & \Ptime,~\text{Thm.~\ref{thm:non-repr}} \\
    \ralabel{29}{37}&       & rrr   & rr\breve{r}   & arr   &       & raa   & rra   & \text{none} & \Ptime,~\text{Thm.~\ref{thm:non-repr}} \\
    \ralabel{30}{37}& aaa   & rrr   & rr\breve{r}   & arr   &       & raa   & rra   &    \text{f.u., Prop.~\ref{prop:30_37-repr}} &            \NPc,~\text{Prop.~\ref{prop:30_37}}    \\
    \ralabel{31}{37}& aaa   &       &               & arr   & rar   & raa   & rra   &         \text{flex} &       \NPc,~\text{Prop.~\ref{prop:31_37}}  \\
    \ralabel{32}{37}&       & rrr   &               & arr   & rar   & raa   & rra   & \text{none} & \Ptime,~\text{Thm.~\ref{thm:non-repr}}  \\
    \ralabel{33}{37}& aaa   & rrr   &               & arr   & rar   & raa   & rra   &      \text{flex} &       \NPc,~\text{Prop. \ref{prop:33_37}}   \\
    \ralabel{34}{37}&       &       & rr\breve{r}   & arr   & rar   & raa   & rra   &  \text{none} & \Ptime,~\text{Thm.~\ref{thm:non-repr}} \\
    \ralabel{35}{37}& aaa   &       & rr\breve{r}   & arr   & rar   & raa   & rra   &          \text{flex} & \NPc,~\text{Prop.~\ref{prop:35_37}}       \\
    \ralabel{36}{37}&       & rrr   & rr\breve{r}   & arr   & rar   & raa   & rra   & \text{flex} & \NPc,~\text{Prop.~\ref{prop:prim}} \\
    \ralabel{37}{37}& aaa   & rrr   & rr\breve{r}   & arr   & rar   & raa   & rra   &    \text{flex}  &   \Ptime,~\text{Prop.~\ref{prop:can-bin-sym2}}  \\ \hline
    \caption{\centering The relation algebras $1_{37}$--$37_{37}$ with the four atoms $a,r,\breve{r},\id$, where $a$ is symmetric and $r$ is non-symmetric
   \centering together with representability information and complexity of their NSP.}
    \label{tab:overview_asymmetric}
\end{longtable}

\newpage
\begin{longtable}[H]{|L||LLLLLLLLLL||L|L|}\hline
    \#_{65} & aaa & bbb & ccc & abb & baa & acc & caa & bcc & cbb & abc & \text{Representation} & \text{NSP Complexity}  \\ \hline \hline
    \ralabel{1}{65}&  &  &  & abb & & acc & & bcc & & & \text{n}, K_2^c[K_2^b[K^a_2]] & \NPc,~\text{Prop.~\ref{prop:normal_hard_finite}} \\
    \ralabel{2}{65}& aaa & & & abb & & acc & & bcc & & & \text{n}, K_2^c[K_2^b[K^a_\omega]] & \NPc,~\text{Prop.~\ref{prop:normal_hard_finite}}\\
    \ralabel{3}{65}& & bbb & & abb & & acc & & bcc & & & \text{n}, K_2^c[K_\omega^b[K_2^a]] & \NPc,~\text{Prop.~\ref{prop:normal_hard_finite}}\\
    \ralabel{4}{65}& aaa & bbb & & abb & & acc & & bcc & & & \text{n}, K_2^c[K_\omega^b[K_\omega^a]] & \NPc,~\text{Prop.~\ref{prop:normal_hard_finite}}\\
    \ralabel{5}{65}& & & ccc & abb & & acc & & bcc & & & \text{n}, K_\omega^c[K_2^b[K_2^a]] & \NPc,~\text{Prop.~\ref{prop:2CycleProNPc}} \\
    \ralabel{6}{65}& aaa & & ccc & abb & & acc & & bcc & & & \text{n}, K_\omega^c[K_2^b[K_\omega^a]] & \NPc,~\text{Prop.~\ref{prop:2CycleProNPc}} \\
    \ralabel{7}{65}& & bbb & ccc & abb & & acc & & bcc & & & \text{n}, K_\omega^c[K_\omega^b[K_2^a]] & \Ptime,~\text{Prop.~\ref{prop:7_65}} \\
    \ralabel{8}{65}& aaa & bbb & ccc & abb & & acc & & bcc & & & \text{n}, K_\omega^c[K_\omega^b[K_\omega^a]] & \Ptime,~\text{Prop.~\ref{prop:can-bin-sym}} \\
    \ralabel{9}{65}& & & & abb & baa & acc & & bcc & & & \neg\text{f.u.}, K_2^c[\Z_5^{a,b}] & \NPc,~\text{Prop.~\ref{prop:2CycleProNPc}} \\
    \ralabel{10}{65}& aaa & & & abb & baa & acc & & bcc & & & \text{n}, K_2^c[\Hgraph^{b,a}] & \NPc,~\text{Prop.~\ref{prop:normal_hard_finite}} \\
    \ralabel{11}{65}& aaa & bbb & & abb & baa & acc & & bcc & & & \text{n}, K_2^c[\R^{a,b}] & \NPc,~\text{Prop.~\ref{prop:normal_hard_finite}} \\
    \ralabel{12}{65}& & & ccc & abb & baa & acc & & bcc & & & \neg\text{f.u.}, K_\omega^c[\Z_5^{a,b}] & \NPc,~\text{Prop.~\ref{prop:2CycleProNPc}} \\
    \ralabel{13}{65}& aaa & & ccc & abb & baa & acc & & bcc & & & \text{n}, K_\omega^c[\Hgraph^{b,a}] & \NPc,~\text{Prop.~\ref{prop:2CycleProNPc}} \\
    \ralabel{14}{65}& aaa & bbb & ccc & abb & baa & acc & & bcc & & & \text{n}, K_\omega^c[\R^{a,b}] & \Ptime,~\text{Prop.~\ref{prop:can-bin-sym}} \\
    \ralabel{15}{65}& & & & & baa & acc & caa & bcc & & & \neg\text{f.u.}, \Z_5^{a,c}[K_2^b] & \NPc,~\text{Prop.~\ref{prop:2CycleProNPc}} \\
    \ralabel{16}{65}& aaa & & & & baa & acc & caa & bcc & & & \text{n}, \Hgraph^{c,a}[K_2^b] & \NPc,~\text{Prop.~\ref{prop:2CycleProNPc}}\\
    \ralabel{17}{65}& & bbb & & & baa & acc & caa & bcc & & & \neg\text{f.u.}, \Z_5^{a,c}[K_\omega^b] & \NPc,~\text{Prop.~\ref{prop:2CycleProNPc}} \\
    \ralabel{18}{65}& aaa & bbb & & & baa & acc & caa & bcc & & & \text{n}, \Hgraph^{c,a}[K_\omega^b] & \NPc,~\text{Prop.~\ref{prop:2CycleProNPc}} \\
    \ralabel{19}{65}& aaa & & ccc & & baa & acc & caa & bcc & & & \text{n}, \R^{c,a}[K_2^b] & \Ptime,~\text{Prop.~\ref{prop:19_65}}\\
    \ralabel{20}{65}& aaa & bbb & ccc & & baa & acc & caa & bcc & & & \text{n}, \R^{c,a}[K_\omega^b] & \Ptime,~\text{Prop.~\ref{prop:can-bin-sym}} \\
    \ralabel{21}{65}& & & & abb & baa & acc & caa & bcc & cbb & & \text{none} & \Ptime,~\text{Thm.~\ref{thm:non-repr}}\\
    \ralabel{22}{65}& aaa & & & abb & baa & acc & caa & bcc & cbb & & \text{none} & \Ptime,~\text{Thm.~\ref{thm:non-repr}} \\
    \ralabel{23}{65}& aaa & bbb & & abb & baa & acc & caa & bcc & cbb & & \text{none} & \Ptime,~\text{Thm.~\ref{thm:non-repr}} \\
    \ralabel{24}{65}& aaa & bbb & ccc & abb & baa & acc & caa & bcc & cbb & & \text{f.u., Prop.~\ref{prop:24_65-repr}} & \Ptime,~\text{Prop.~\ref{prop:24_65}} \\
    \ralabel{25}{65}& & & & & & & & & & abc & \text{n}, K_{2}^c \boxtimes_a K_{2}^b & \NPc,~\text{Prop.~\ref{prop:normal_hard_finite}} \\
    \ralabel{26}{65}& aaa & & & abb & & & & & & abc & \text{normal} & \NPc,~\text{Prop.~\ref{prop:normal_hard_finite}} \\
    \ralabel{27}{65}& aaa & bbb & & abb & baa & & & & & abc & \text{normal} & \NPc,~\text{Prop.~\ref{prop:27_65}} \\
    \ralabel{28}{65}& aaa & & & abb & & acc & & & & abc & \text{normal} & \NPc,~\text{Prop.~\ref{prop:normal_hard_finite}} \\
    \ralabel{29}{65}& aaa & bbb & ccc & & baa & & caa & & & abc & \text{n, } K_{\omega}^c \boxtimes_a K_{\omega}^b  & \NPc,~\text{Prop.~\ref{prop:29_65}} \\
    \ralabel{30}{65}& aaa & & ccc & abb & baa & & caa & & & abc & \text{f.u., Prop.~\ref{prop:30_65-31_65-repr}} & \NPc,~\text{Prop.~\ref{prop:30_65}} \\
    \ralabel{31}{65}& aaa & bbb & ccc & abb & baa & & caa & & & abc & \text{f.u., Prop.~\ref{prop:30_65-31_65-repr}} & \NPc,~\text{Prop.~\ref{prop:31_65}} \\
    \ralabel{32}{65}& aaa & & & abb & baa & acc & caa & & & abc & \text{flex} & \NPc,~\text{Prop.~\ref{prop:flexNo1CycleNSP}} \\
    \ralabel{33}{65}& aaa & bbb & & abb & baa & acc & caa & & & abc & \text{flex} & \NPc,~\text{Prop.~\ref{prop:flexNo1CycleNSP}} \\
    \ralabel{34}{65}& aaa & bbb & ccc & abb & baa & acc & caa & & & abc & \text{flex} & \NPc,~\text{Prop.~\ref{prop:34_65}} \\
    \ralabel{35}{65}& aaa & bbb & & abb & & acc & & bcc & & abc & \text{none} & \Ptime,~\text{Thm.~\ref{thm:non-repr}} \\
    \ralabel{36}{65}& aaa & bbb & ccc & abb & & acc & & bcc & & abc & \text{none} & \Ptime,~\text{Thm.~\ref{thm:non-repr}} \\
    \ralabel{37}{65}& aaa & bbb & & abb & baa & acc & & bcc & & abc & \text{none} & \Ptime,~\text{Thm.~\ref{thm:non-repr}}\\
    \ralabel{38}{65}& aaa & bbb & ccc & abb & baa & acc & & bcc & & abc & \text{none} & \Ptime,~\text{Thm.~\ref{thm:non-repr}} \\
    \ralabel{39}{65}& & & & abb & & & caa & bcc & & abc & \neg\text{f.u., Prop.~\ref{prop:39_65-repr}} & \NPc,~\text{Prop.~\ref{prop:bounded-hard}} \\
    \ralabel{40}{65}& aaa & & & abb & & & caa & bcc & & abc & \text{none} & \Ptime,~\text{Thm.~\ref{thm:non-repr}}\\
    \ralabel{41}{65}& aaa & bbb & & abb & & & caa & bcc & & abc & \text{none} & \Ptime,~\text{Thm.~\ref{thm:non-repr}}\\
    \ralabel{42}{65}& aaa & bbb & ccc & abb & & & caa & bcc & & abc & \text{none} & \Ptime,~\text{Thm.~\ref{thm:non-repr}}\\
    \ralabel{43}{65}& & & & abb & baa & & caa & bcc & & abc & \text{none} & \Ptime,~\text{Thm.~\ref{thm:non-repr}}\\
    \ralabel{44}{65}& aaa & & & abb & baa & & caa & bcc & & abc & \text{none} & \Ptime,~\text{Thm.~\ref{thm:non-repr}} \\
    \ralabel{45}{65}& & bbb & & abb & baa & & caa & bcc & & abc & \text{none} & \Ptime,~\text{Thm.~\ref{thm:non-repr}} \\
    \ralabel{46}{65}& aaa & bbb & & abb & baa & & caa & bcc & & abc & \text{normal} & \NPc,~\text{Prop.~\ref{prop:prim}} \\
    \ralabel{47}{65}& & & ccc & abb & baa & & caa & bcc & & abc & \text{none} & \Ptime,~\text{Thm.~\ref{thm:non-repr}}\\
    \ralabel{48}{65}& aaa & & ccc & abb & baa & & caa & bcc & & abc & \text{none} & \Ptime,~\text{Thm.~\ref{thm:non-repr}} \\
    \ralabel{49}{65}& & bbb & ccc & abb & baa & & caa & bcc & & abc & \text{none} & \Ptime,~\text{Thm.~\ref{thm:non-repr}} \\
    \ralabel{50}{65}& aaa & bbb & ccc & abb & baa & & caa & bcc & & abc & \text{none} & \Ptime,~\text{Thm.~\ref{thm:non-repr}}\\
    \ralabel{51}{65}& & bbb & & & baa & acc & caa & bcc & & abc & \neg\text{f.u., Lem.~\ref{lem:5165_unsat}} & \NPc,~\text{Cor.~\ref{cor:56_65-and-others-np-hard}} \\
    \ralabel{52}{65}& aaa & bbb & & & baa & acc & caa & bcc & & abc & \text{normal} & \NPc,~\text{Prop.~\ref{prop:31_37}} \\
    \ralabel{53}{65}& aaa & bbb & ccc & & baa & acc & caa & bcc & & abc & \text{normal} & \Ptime,~\text{Prop.~\ref{prop:can-bin-sym}} \\
    \ralabel{54}{65}& & & & abb & baa & acc & caa & bcc & & abc & \text{none} & \Ptime,~\text{Thm.~\ref{thm:non-repr}}\\
    \ralabel{55}{65}& aaa & & & abb & baa & acc & caa & bcc & & abc & \text{flex} & \NPc,~\text{Prop.~\ref{prop:flexNo1CycleNSP}}\\
    \ralabel{56}{65}& & bbb & & abb & baa & acc & caa & bcc & & abc & \neg\text{f.u., Prop.~\ref{prop:56_65-not-fu}} & \multicolumn{1}{l}{\makebox[0pt][l]{\text{$\NPc$,~}\text{Cor.~\ref{cor:56_65-and-others-np-hard}, \tabcite{56paper}}}} \\
    \ralabel{57}{65}& aaa & bbb & & abb & baa & acc & caa & bcc & & abc & \text{flex} & \NPc,~\text{Prop.~\ref{prop:flexNo1CycleNSP}}\\
    \ralabel{58}{65}& & & ccc & abb & baa & acc & caa & bcc & & abc & \text{none} & \Ptime,~\text{Thm.~\ref{thm:non-repr}} \\
    \ralabel{59}{65}& aaa & & ccc & abb & baa & acc & caa & bcc & & abc & \text{flex} & \NPc,~\text{Prop.~\ref{prop:flexNo1CycleNSP}}\\
    \ralabel{60}{65}& & bbb & ccc & abb & baa & acc & caa & bcc & & abc & \text{none} & \Ptime,~\text{Thm.~\ref{thm:non-repr}} \\
    \ralabel{61}{65}& aaa & bbb & ccc & abb & baa & acc & caa & bcc & & abc & \text{flex} & \Ptime,~\text{Prop.~\ref{prop:can-bin-sym}} \\
    \ralabel{62}{65}& & & & abb & baa & acc & caa & bcc & cbb & abc & \neg\text{f.u., Prop.~\ref{prop:62_65-repr}} & \NPc,~\text{Prop.~\ref{prop:bounded-hard}} \\
    \ralabel{63}{65}& aaa & & & abb & baa & acc & caa & bcc & cbb & abc & \text{flex} & \NPc,~\text{Prop.~\ref{prop:flexNo1CycleNSP}}\\
    \ralabel{64}{65}& aaa & bbb & & abb & baa & acc & caa & bcc & cbb & abc & \text{flex} & \NPc,~\text{Prop.~\ref{prop:flexNo1CycleNSP}}\\
    \ralabel{65}{65}& aaa & bbb & ccc & abb & baa & acc & caa & bcc & cbb & abc & \text{flex} & \Ptime,~\text{Prop.~\ref{prop:can-bin-sym}}
    \\
    \hline
   \caption{\centering The symmetric relation algebras $1_{65}$--$65_{65}$ with four atoms $a,b,c,\id$,\linebreak
   \centering together with representability information and complexity of their NSP.}
    \label{tab:overview_symmetric}
\end{longtable}

\section*{Disclosure}
AI tools (namely ChatGPT 5.6 Sol) were used only during the final pre-submission review of this manuscript, to assist with checking mathematical arguments, identifying and correcting errors, improving the presentation, and debugging and documenting the accompanying computer programs. The results and the original manuscript predate this use. The authors take full responsibility for the final content.

\bibliographystyle{alphaurl}
\bibliography{references}

\end{document}

%% file: makros.tex
\usepackage{float}
\usepackage{array}
\usepackage{amsfonts}
\usepackage{amsmath}
\usepackage{quiver}
\usepackage{faktor}
\usepackage{algpseudocode,algorithm}
\usepackage{longtable}

\newcolumntype{L}{>{$}l<{$}}
\newcolumntype{C}{>{$}c<{$}}
\newcolumntype{R}{>{$}r<{$}}

\allowdisplaybreaks
\newcommand{\ralabel}[2]{\phantomsection\label{#1_#2} #1_{#2}}
\newcommand{\ra}[2]{\hyperref[#1_#2]{#1_{#2}}}

\def\N{\mathbb{N}}
\def\Z{\mathbb{Z}}
\def\Q{\mathbb{Q}}
\newcommand{\Hgraph}{\mathbb{H}}
\def\R{\mathbb{R}}
\def\T{\mathbb{T}}

\def\fA{\mathfrak{A}}
\def\fAo{\mathfrak{A}_0}
\def\fB{\mathfrak{B}}
\def\fC{\mathfrak{C}}
\def\fD{\mathfrak{D}}

\def\fR{\mathfrak{R}}

\def\bA{\mathbf{A}}
\def\bB{\mathbf{B}}
\def\bC{\mathbf{C}}

\DeclareMathOperator{\Ptime}{\mathbf{P}}
\DeclareMathOperator{\NP}{\mathbf{NP}}
\DeclareMathOperator{\NPc}{\mathbf{NPc}}
\DeclareMathOperator{\Betw}{Betw}
\DeclareMathOperator{\NAE}{NAE}

\DeclareMathOperator{\Aut}{Aut}
\DeclareMathOperator{\CSP}{CSP}
\DeclareMathOperator{\NCP}{NCP}
\DeclareMathOperator{\NSP}{NSP}
\DeclareMathOperator{\Pol}{Pol}
\DeclareMathOperator{\AP}{AP}
\DeclareMathOperator{\Cy}{Cy}

\DeclareMathOperator{\RA}{RA}

\DeclareMathOperator{\id}{Id}

\makeatletter
\def\widebreve#1{\mathop{\vbox{\m@th\ialign{##\crcr\noalign{\kern\p@}
  \brevefill\crcr\noalign{\kern0.1\p@\nointerlineskip}
  $\hfil\displaystyle{#1}\hfil$\crcr}}}\limits}
\def\brevefill{$\scriptscriptstyle\m@th \setbox\z@\hbox{$\scriptscriptstyle\braceld$}
  \bracelu\leaders\vrule \@height\ht\z@ \@depth\z@\hfill\braceru$}
\makeatother

\newcommand{\tabcite}[1]{\ifmmode\mbox{\cite{#1}}\else\cite{#1}\fi}

%% file: references.bib
@misc{RA4Code,
  author = {Manuel Bodirsky and Moritz Jahn and Simon Kn{\"a}uer
            and Mat{\v{e}}j Kone{\v{c}}n{\'y} and Paul Winkler},
  title = {Computational supplement to
           "{Network Satisfaction over Four-Atom Relation Algebras:
           Complexity and Representations}"},
  year = {2026},
  howpublished = {Zenodo},
  note = {Version 1.0. Software and accompanying data},
  doi = {10.5281/zenodo.23069624}
}

@article{AndrekaMaddux,
  author = {Hajnal Andr{\'e}ka and Roger D. Maddux},
  title = {Representations for Small Relation Algebras},
  journal = {Notre Dame Journal of Formal Logic},
  volume = {35},
  number = {4},
  pages = {550--562},
  year = {1994},
  doi = {10.1305/ndjfl/1040408612},
}

@article{HirschCristani,
  author = {Matteo Cristani and Robin Hirsch},
  title = {The Complexity of Constraint Satisfaction Problems for Small Relation Algebras},
  journal = {Artificial Intelligence},
  volume = {156},
  number = {2},
  pages = {177--196},
  year = {2004},
  doi = {10.1016/j.artint.2004.02.003},
}

@article{Tarski41,
  author = {Alfred Tarski},
  title = {On the Calculus of Relations},
  journal = {The Journal of Symbolic Logic},
  volume = {6},
  number = {3},
  pages = {73--89},
  year = {1941},
  doi = {10.2307/2268577},
}

@article{LyndonRelationAlgebras,
  author = {Roger C. Lyndon},
  title = {The Representation of Relational Algebras},
  journal = {Annals of Mathematics},
  volume = {51},
  number = {3},
  pages = {707--729},
  year = {1950},
  doi = {10.2307/1969375},
}

@article{Monk,
  author = {Donald Monk},
  title = {On Representable Relation Algebras},
  journal = {Michigan Mathematical Journal},
  volume = {11},
  number = {3},
  pages = {207--210},
  year = {1964},
  doi = {10.1307/mmj/1028999131},
}

@article{TarskiRRA,
  author = {Alfred Tarski},
  title = {Contributions to the Theory of Models},
  journal = {Koninklijke Nederlandse Akademie van Wetenschappen, Proceedings},
  series = {A},
  volume = {58},
  pages = {56--64},
  year = {1955},
  note = {Parts I and II appeared in volume 57 (1954), pp. 572--588},
}

@book{Hypergroups,
  author = {Paul-Hermann Zieschang},
  title = {Hypergroups},
  publisher = {Springer},
  address = {Cham},
  year = {2023},
  doi = {10.1007/978-3-031-39489-8},
}

@article{ComerChromaticPolygroups,
  author =        {Stephen D. Comer},
  journal =       {Congressus Numerantium},
  pages =         {85--95},
  title =         {A Remark on Chromatic Polygroups},
  volume =        {38},
  year =          {1983},
}

@book{Book,
  author = {Manuel Bodirsky},
  title = {Complexity of Infinite-Domain Constraint Satisfaction},
  series = {Lecture Notes in Logic},
  volume = {52},
  publisher = {Cambridge University Press},
  address = {Cambridge},
  year = {2021},
  doi = {10.1017/9781107337534},
}

@article{Hirsch,
  author = {Robin Hirsch},
  title = {Relation Algebras of Intervals},
  journal = {Artificial Intelligence},
  volume = {83},
  number = {2},
  pages = {267--295},
  year = {1996},
  doi = {10.1016/0004-3702(95)00042-9},
}

@inproceedings{BodirskyKnaeuerDatalog23,
  author = {Manuel Bodirsky and Simon Kn{\"a}uer},
  title = {Network Satisfaction Problems Solved by $k$-Consistency},
  booktitle = {50th International Colloquium on Automata, Languages, and Programming (ICALP 2023)},
  series = {Leibniz International Proceedings in Informatics (LIPIcs)},
  volume = {261},
  pages = {116:1--116:20},
  publisher = {Schloss Dagstuhl -- Leibniz-Zentrum f{\"u}r Informatik},
  address = {Dagstuhl, Germany},
  year = {2023},
  doi = {10.4230/LIPIcs.ICALP.2023.116},
}

@article{Hirsch-Undecidable,
  author = {Robin Hirsch},
  title = {A Finite Relation Algebra with Undecidable Network Satisfaction Problem},
  journal = {Logic Journal of the IGPL},
  volume = {7},
  number = {4},
  pages = {547--554},
  year = {1999},
  doi = {10.1093/jigpal/7.4.547},
}

@article{HirschHodkinsonRepresentability,
  author = {Robin Hirsch and Ian Hodkinson},
  title = {Representability Is Not Decidable for Finite Relation Algebras},
  journal = {Transactions of the American Mathematical Society},
  volume = {353},
  number = {4},
  pages = {1403--1425},
  year = {2001},
  doi = {10.1090/S0002-9947-99-02264-3},
}

@inproceedings{BodirskyRamics,
  author = {Manuel Bodirsky},
  title = {Finite Relation Algebras with Normal Representations},
  booktitle = {Relational and Algebraic Methods in Computer Science},
  series = {Lecture Notes in Computer Science},
  volume = {11194},
  pages = {3--17},
  publisher = {Springer},
  year = {2018},
  doi = {10.1007/978-3-030-02149-8_1},
}

@article{BPP-projective-homomorphisms,
  author = {Manuel Bodirsky and Michael Pinsker and Andr{\'a}s Pongr{\'a}cz},
  title = {Projective Clone Homomorphisms},
  journal = {The Journal of Symbolic Logic},
  volume = {86},
  number = {1},
  pages = {148--161},
  year = {2021},
  doi = {10.1017/jsl.2019.23},
}

@article{Ladner,
  author =        {Richard E. Ladner},
  journal =       {Journal of the ACM},
  number =        {1},
  pages =         {155--171},
  title =         {On the Structure of Polynomial Time Reducibility},
  volume =        {22},
  year =          {1975},
}

@article{Zhuk20,
  author =        {Dmitriy Zhuk},
  journal =       {J. {ACM}},
  number =        {5},
  pages =         {30:1--30:78},
  title =         {A Proof of the {CSP} Dichotomy Conjecture},
  volume =        {67},
  year =          {2020},
  doi =           {10.1145/3402029},
}

@inproceedings{BulatovFVConjecture,
  author =        {Andrei A. Bulatov},
  booktitle =     {58th {IEEE} Annual Symposium on Foundations of
                   Computer Science, {FOCS} 2017, {B}erkeley, {CA},
                   {USA}, {O}ctober 15-17},
  pages =         {319--330},
  title =         {A Dichotomy Theorem for Nonuniform {CSP}s},
  year =          {2017},
}

@inproceedings{ZhukFVConjecture,
  author =        {Dmitriy N. Zhuk},
  booktitle =     {58th {IEEE} Annual Symposium on Foundations of
                   Computer Science, {FOCS} 2017, {B}erkeley, {CA},
                   {USA}, {O}ctober 15-17},
  pages =         {331--342},
  title =         {A Proof of {CSP} Dichotomy Conjecture},
  year =          {2017},
  eprint = {1704.01914},
  eprinttype = {arxiv},
}

@article{wonderland,
  author = {Libor Barto and Jakub Opr\v{s}al and Michael Pinsker},
  title = {The Wonderland of Reflections},
  journal = {Israel Journal of Mathematics},
  volume = {223},
  number = {1},
  pages = {363--398},
  year = {2018},
  doi = {10.1007/s11856-017-1621-9},
}

@inproceedings{BKOPP,
  author = {Libor Barto and Michael Kompatscher and Miroslav Ol\v{s}{\'a}k and Trung Van Pham and Michael Pinsker},
  title = {The Equivalence of Two Dichotomy Conjectures for Infinite Domain Constraint Satisfaction Problems},
  booktitle = {32nd Annual ACM/IEEE Symposium on Logic in Computer Science (LICS 2017)},
  pages = {1--12},
  publisher = {IEEE},
  year = {2017},
  doi = {10.1109/LICS.2017.8005128},
}

@inproceedings{BartoPinskerDichotomy,
  author =        {Libor Barto and Michael Pinsker},
  booktitle =     {Proceedings of the 31st {A}nnual {IEEE} {S}ymposium
                   on {L}ogic in {C}omputer {S}cience -- {LICS}'16},
  pages =         {615--622},
  title =         {The algebraic dichotomy conjecture for infinite
                   domain constraint satisfaction problems},
  year =          {2016},
  eprint = {1602.04353},
  eprinttype = {arxiv},
}

@article{Henson,
  author =        {C. Ward Henson},
  journal =       {Journal of Symbolic Logic},
  pages =         {494--500},
  title =         {Countable homogeneous relational systems and
                   categorical theories},
  volume =        {37},
  year =          {1972},
}

@book{Cherlin,
  author = {Gregory L. Cherlin},
  title = {The Classification of Countable Homogeneous Directed Graphs and Countable Homogeneous $n$-Tournaments},
  series = {Memoirs of the American Mathematical Society},
  volume = {131},
  number = {621},
  publisher = {American Mathematical Society},
  address = {Providence, RI},
  year = {1998},
}

@article{LachlanWoodrow,
  author =        {Alistair H. Lachlan and Robert E. Woodrow},
  journal =       {Transactions of the AMS},
  number =        {1},
  pages =         {51--94},
  title =         {Countable ultrahomogeneous undirected graphs},
  volume =        {262},
  year =          {1980},
}

@incollection{Cameron,
  author = {Peter J. Cameron},
  title = {The Random Graph},
  booktitle = {The Mathematics of Paul Erd{\H{o}}s II},
  editor = {Ronald L. Graham and Jaroslav Ne\v{s}et\v{r}il},
  series = {Algorithms and Combinatorics},
  volume = {14},
  pages = {333--351},
  publisher = {Springer},
  address = {Berlin},
  year = {1997},
}

@book{CherlinMetrically,
  author = {Gregory Cherlin},
  title = {Homogeneous Ordered Graphs, Metrically Homogeneous Graphs, and Beyond},
  series = {Lecture Notes in Logic},
  publisher = {Cambridge University Press},
  year = {2022},
  note = {Two volumes (vols. 53 and 54)},
}

@book{Maddux2006-dp,
  author = {Roger D. Maddux},
  title = {Relation Algebras},
  series = {Studies in Logic and the Foundations of Mathematics},
  volume = {150},
  publisher = {Elsevier},
  address = {Amsterdam},
  year = {2006},
}

@misc{56paper,
  author =        {Manuel Bodirsky and Santiago Guzmán-Pro and
                   Moritz Jahn and Mat{\v{e}}j Kone{\v{c}}n{\'y} and
                   Paul Winkler},
  title =         {Circular Chromatic Numbers, Signability, Relation Algebras, and Network                   Satisfaction Problems},
  year =          {2025},
  eprint = {2512.06878},
  eprinttype = {arxiv},
}

@article{BodirskyKnaeJAIR,
  author = {Manuel Bodirsky and Simon Kn{\"a}uer},
  title = {The Complexity of Network Satisfaction Problems for Symmetric Relation Algebras with a Flexible Atom},
  journal = {Journal of Artificial Intelligence Research},
  volume = {75},
  pages = {1701--1744},
  year = {2022},
  doi = {10.1613/jair.1.14195},
}

@incollection{BodirskyKnaeRamics,
  author =        {Manuel Bodirsky and Simon Kn\"{a}uer},
  title =         {Hardness of Network Satisfaction for Relation
                   Algebras with Normal Representations},
  booktitle =     {Relational and Algebraic Methods in Computer Science},
  editor =        {Uli Fahrenberg and Peter Jipsen and Michael Winter},
  series =        {Lecture Notes in Computer Science},
  volume =        {12062},
  pages =         {31--46},
  publisher =     {Springer International Publishing},
  year =          {2020},
  doi =           {10.1007/978-3-030-43520-2_3},
}

@inproceedings{PCSP,
  author = {Jakub Bul{\'{\i}}n and Andrei A. Krokhin and Jakub Opr\v{s}al},
  title = {Algebraic Approach to Promise Constraint Satisfaction},
  booktitle = {Proceedings of the 51st Annual ACM SIGACT Symposium on Theory of Computing (STOC 2019)},
  pages = {602--613},
  publisher = {ACM},
  year = {2019},
  doi = {10.1145/3313276.3316300},
}

@article{BJJ95,
  author =        {Bang-Jensen, J{\o}rgen and Huang, Jing},
  journal =       {Journal of Graph Theory},
  number =        {2},
  pages =         {141--161},
  title =         {Quasi-transitive digraphs},
  volume =        {20},
  year =          {1995},
  doi =           {10.1002/jgt.3190200205},
}

@article{tcsps-journal,
  author = {Manuel Bodirsky and Jan K{\'a}ra},
  title = {The Complexity of Temporal Constraint Satisfaction Problems},
  journal = {Journal of the ACM},
  volume = {57},
  number = {2},
  pages = {9:1--9:41},
  year = {2010},
  note = {An extended abstract appeared in the Proceedings of STOC 2008},
  doi = {10.1145/1667053.1667058},
}

@article{BodPin-Schaefer-both,
  author = {Manuel Bodirsky and Michael Pinsker},
  title = {Schaefer's Theorem for Graphs},
  journal = {Journal of the ACM},
  volume = {62},
  number = {3},
  pages = {19:1--19:52},
  year = {2015},
  note = {A conference version appeared in the Proceedings of STOC 2011, pp. 655--664},
  doi = {10.1145/2764899},
}

@book{Hodges,
  address =       {Cambridge},
  author =        {Wilfrid Hodges},
  publisher =     {Cambridge University Press},
  title =         {A shorter model theory},
  year =          {1997},
}

@book{GareyJohnson,
  address =       {San Francisco},
  author =        {Michael R. Garey and David S. Johnson},
  publisher =     {W.~H.~Freeman},
  title =         {Computers and Intractability: A Guide to the Theory
                   of {NP}-Completeness},
  year =          {1979},
}

@inproceedings{Conservative,
  author = {Andrei A. Bulatov},
  title = {Tractable Conservative Constraint Satisfaction Problems},
  booktitle = {Proceedings of the 18th IEEE Symposium on Logic in Computer Science (LICS 2003)},
  pages = {321--330},
  publisher = {IEEE Computer Society},
  address = {Ottawa, Canada},
  year = {2003},
  doi = {10.1109/LICS.2003.1210072},
}

@book{Post1941,
  author = {Emil Leon Post},
  title = {The Two-Valued Iterative Systems of Mathematical Logic},
  series = {Annals of Mathematics Studies},
  volume = {5},
  publisher = {Princeton University Press},
  address = {Princeton, NJ},
  year = {1941},
}

@inproceedings{Schaefer1978,
  author = {Thomas J. Schaefer},
  title = {The Complexity of Satisfiability Problems},
  booktitle = {Proceedings of the Tenth Annual ACM Symposium on Theory of Computing (STOC 1978)},
  pages = {216--226},
  publisher = {ACM},
  year = {1978},
  doi = {10.1145/800133.804350},
}

@article{TopoBirkhoff,
  author = {Manuel Bodirsky and Michael Pinsker},
  title = {Topological {Birkhoff}},
  journal = {Transactions of the American Mathematical Society},
  volume = {367},
  number = {4},
  pages = {2527--2549},
  year = {2015},
  doi = {10.1090/S0002-9947-2014-05975-8},
}

@inproceedings{Maddux2006,
  author = {Roger D. Maddux},
  title = {Finite Symmetric Integral Relation Algebras with No 3-Cycles},
  booktitle = {Relations and Kleene Algebra in Computer Science},
  editor = {Renate A. Schmidt},
  series = {Lecture Notes in Computer Science},
  volume = {4136},
  pages = {2--29},
  publisher = {Springer},
  year = {2006},
  doi = {10.1007/11828563_2},
}

@book{BS,
  author = {Stanley N. Burris and Hanamantagouda P. Sankappanavar},
  title = {A Course in Universal Algebra},
  series = {Graduate Texts in Mathematics},
  volume = {78},
  publisher = {Springer-Verlag},
  address = {New York},
  year = {1981},
}

@article{TarskiJonnsonOperators,
  author = {Bjarni J{\'o}nsson and Alfred Tarski},
  title = {Boolean Algebras with Operators, Part II},
  journal = {American Journal of Mathematics},
  volume = {74},
  number = {1},
  pages = {127--162},
  year = {1952},
  doi = {10.2307/2372074},
}

@book{HirschHodkinson,
  author = {Robin Hirsch and Ian Hodkinson},
  title = {Relation Algebras by Games},
  series = {Studies in Logic and the Foundations of Mathematics},
  volume = {147},
  publisher = {North-Holland/Elsevier},
  address = {Amsterdam},
  year = {2002},
}

@article{BodDalJournal,
  author = {Manuel Bodirsky and V{\'i}ctor Dalmau},
  title = {Datalog and Constraint Satisfaction with Infinite Templates},
  journal = {Journal of Computer and System Sciences},
  volume = {79},
  number = {1},
  pages = {79--100},
  year = {2013},
  note = {A preliminary version appeared in the proceedings of STACS 2006},
  doi = {10.1016/j.jcss.2012.05.012},
}

@incollection{ComerExtensionSchemes,
  author = {Stephen D. Comer},
  title = {Extension of Polygroups by Polygroups and Their Representations Using Color Schemes},
  booktitle = {Universal Algebra and Lattice Theory},
  editor = {Ralph S. Freese and Octavio C. Garcia},
  series = {Lecture Notes in Mathematics},
  volume = {1004},
  pages = {91--103},
  publisher = {Springer},
  address = {Berlin, Heidelberg},
  year = {1983},
}

@incollection{CherlinImPrim,
  author = {Gregory Cherlin},
  title = {Homogeneous Directed Graphs I: The Imprimitive Case},
  booktitle = {Logic Colloquium '85},
  editor = {{Paris Logic Group}},
  series = {Studies in Logic and the Foundations of Mathematics},
  volume = {122},
  pages = {67--88},
  publisher = {North-Holland},
  address = {Amsterdam},
  year = {1987},
}

@misc{HerFO,
  author =        {Manuel Bodirsky and Santiago Guzmán-Pro},
  title =         {Hereditary First-Order Model Checking},
  year =          {2024},
  eprint = {2411.10860},
  eprinttype = {arxiv},
}

@incollection{PointAlgebra,
  author = {Marc Vilain and Henry Kautz and Peter van Beek},
  title = {Constraint Propagation Algorithms for Temporal Reasoning: A Revised Report},
  booktitle = {Readings in Qualitative Reasoning about Physical Systems},
  editor = {Daniel S. Weld and Johan de Kleer},
  pages = {373--381},
  publisher = {Morgan Kaufmann},
  year = {1989},
}

@article{Comer1984,
  author =        {Stephen D. Comer},
  journal =       {Algebra Universalis},
  month =         feb,
  number =        {1},
  pages =         {77--94},
  title =         {Combinatorial aspects of relations},
  volume =        {18},
  year =          {1984},
  doi =           {10.1007/bf01182249},
}

@article{Duentsch,
  author = {Ivo D{\"u}ntsch},
  title = {Relation Algebras and Their Application in Temporal and Spatial Reasoning},
  journal = {Artificial Intelligence Review},
  volume = {23},
  number = {4},
  pages = {315--357},
  year = {2005},
  doi = {10.1007/s10462-004-5899-8},
}

@article{BodirskyKutzAI,
  author = {Manuel Bodirsky and Martin Kutz},
  title = {Determining the Consistency of Partial Tree Descriptions},
  journal = {Artificial Intelligence},
  volume = {171},
  number = {2--3},
  pages = {185--196},
  year = {2007},
  doi = {10.1016/j.artint.2006.12.004},
}

@article{GG1955,
  author =        {Greenwood, R. E. and Gleason, A. M.},
  journal =       {Canadian Journal of Mathematics},
  pages =         {1--7},
  title =         {Combinatorial Relations and Chromatic Graphs},
  volume =        {7},
  year =          {1955},
  doi =           {10.4153/CJM-1955-001-4},
}

@article{SP,
  author = {Jacobo Valdes and Robert E. Tarjan and Eugene L. Lawler},
  title = {The Recognition of Series Parallel Digraphs},
  journal = {SIAM Journal on Computing},
  volume = {11},
  number = {2},
  pages = {298--313},
  year = {1982},
  doi = {10.1137/0211023},
}

@article{ComerColorSchemes,
  author =        {Comer, Stephen},
  journal =       {Acta Universitatis Carolinae. Mathematica et Physica},
  month =         jan,
  number =        {2},
  pages =         {39--48},
  title =         {Constructions of color schemes},
  volume =        {24},
  year =          {1983},
  issn =          {0001-7140},
}

@article{BGP2025Circular,
  author =        {Bodirsky, Manuel and Guzmán-Pro, Santiago},
  journal =       {Journal of {G}raph {T}heory},
  number =        {4},
  pages =         {426--445},
  title =         {The {G}eneric {C}ircular {T}riangle-{F}ree {G}raph},
  volume =        {109},
  year =          {2025},
  doi =           {10.1002/jgt.23235},
}

@article{BroxvallJonsson,
  author =        {Mathias Broxvall and Peter Jonsson},
  journal =       {Artificial Intelligence},
  number =        {2},
  pages =         {179--220},
  title =         {Point algebras for temporal reasoning: Algorithms and
                   complexity},
  volume =        {149},
  year =          {2003},
}

@incollection{BP-reductsRamsey,
  author = {Manuel Bodirsky and Michael Pinsker},
  title = {Reducts of {Ramsey} Structures},
  booktitle = {Model Theoretic Methods in Finite Combinatorics},
  editor = {Martin Grohe and Johann A. Makowsky},
  series = {Contemporary Mathematics},
  volume = {558},
  pages = {489--519},
  publisher = {American Mathematical Society},
  address = {Providence, RI},
  year = {2011},
  doi = {10.1090/conm/558/11058},
}

@mastersthesis{KnaeuerMaster,
  author = {Simon Kn{\"a}uer},
  title = {Constraint Satisfaction over the Random Tournament},
  school = {TU Dresden},
  year = {2018},
}

@article{BMPP16,
  author =        {Manuel Bodirsky and Barnaby Martin and
                   Michael Pinsker and Andr{\'{a}}s Pongr{\'{a}}cz},
  journal =       {SIAM Journal on Computing},
  note =          {A conference version appeared in the Proceedings of
                   the 43rd International Colloquium on Automata,
                   Languages, and Programming, {ICALP} 2016, pages
                   119:1-119:14},
  number =        {4},
  pages =         {1224--1264},
  title =         {Constraint Satisfaction Problems for Reducts of
                   Homogeneous Graphs},
  volume =        {48},
  year =          {2019},
  doi =           {10.1137/16M1082974},
}

@inproceedings{BodirskyKutz,
  author = {Manuel Bodirsky and Martin Kutz},
  title = {Pure Dominance Constraints},
  booktitle = {STACS 2002},
  series = {Lecture Notes in Computer Science},
  volume = {2285},
  pages = {287--298},
  publisher = {Springer},
  year = {2002},
  doi = {10.1007/3-540-45841-7_23},
}

@inproceedings{DawarKreutzer08,
  author = {Anuj Dawar and Stephan Kreutzer},
  title = {On Datalog vs. {LFP}},
  booktitle = {Automata, Languages and Programming, ICALP 2008, Part II},
  editor = {Luca Aceto and Ivan Damg{\aa}rd and Leslie Ann Goldberg and Magn{\'u}s M. Halld{\'o}rsson and Anna Ing{\'o}lfsd{\'o}ttir and Igor Walukiewicz},
  series = {Lecture Notes in Computer Science},
  volume = {5126},
  pages = {160--171},
  publisher = {Springer},
  year = {2008},
  doi = {10.1007/978-3-540-70583-3_14},
}

@incollection{LachlanFiniteDigraphs,
  author = {A. H. Lachlan},
  title = {Finite Homogeneous Simple Digraphs},
  booktitle = {Logic Colloquium '81},
  editor = {Jacques Stern},
  series = {Studies in Logic and the Foundations of Mathematics},
  volume = {107},
  pages = {189--208},
  publisher = {North-Holland},
  address = {New York},
  year = {1982},
  doi = {10.1016/S0049-237X(08)71884-6},
}
